\documentclass[reqno, 12pt]{amsart}

\usepackage{amsmath}
\usepackage{amsfonts}
\usepackage{amssymb}
\usepackage{amsthm}
\usepackage{mathrsfs}
\usepackage{bbm}
\usepackage{graphicx}
\usepackage{tikz}
\usetikzlibrary{calc,decorations.markings}
\usetikzlibrary{arrows.meta}
\usetikzlibrary{positioning}
\usetikzlibrary{cd, intersections, calc, decorations.pathmorphing, arrows, decorations.pathreplacing}
\usepackage{color}
\usepackage{comment}
\usepackage[normalem]{ulem}

\usepackage{epsfig,amssymb,amsmath,version}
\usepackage{amssymb,version,graphicx,fancybox,mathrsfs}
\usepackage{subfigure}
\usepackage{color}
\usepackage{stmaryrd}
\usepackage{multirow}
\usepackage{booktabs,siunitx}
\usepackage{booktabs}
\usepackage{multicol,epstopdf}
\usepackage{xypic}
\usepackage[]{graphicx}
\usepackage[all]{xy}
\usepackage{tikz,ifthen}
\usetikzlibrary{positioning, math, decorations.markings, arrows.meta}
\usetikzlibrary{calc,shapes,cd}
\usetikzlibrary{knots} 
\usepgfmodule{decorations}
\allowdisplaybreaks
\usepackage{bm}

\usepackage{adjustbox}

\usepackage{bmpsize}

\usepackage{mathtools}

\usepackage{enumerate}

\usepackage{subfiles}
\usepackage[colorlinks, linkcolor=blue, citecolor=red, urlcolor=cyan,pagebackref]{hyperref}

\allowdisplaybreaks

\numberwithin{equation}{section}

\usepackage{cleveref}
\crefname{thm}{Theorem}{Theorems}
\crefname{cor}{Corollary}{Corollaries}
\crefname{lem}{Lemma}{Lemmas}
\crefname{sublem}{Sublemma}{Sublemmas}
\crefname{prop}{Proposition}{Propositions}
\crefname{def}{Definition}{Definitions}
\crefname{example}{Example}{Examples}
\crefname{claim}{Claim}{Claims}
\crefname{conj}{Conjecture}{Conjectures}
\crefname{conv}{Notation}{Notations}
\crefname{rem}{Remark}{Remarks}
\crefname{rmk}{Remark}{Remarks}
\crefname{prob}{Problem}{Problems}
\crefname{figure}{Figure}{Figures}
\crefname{table}{Table}{Tables}
\crefname{section}{Section}{Sections}
\crefname{subsection}{Section}{Sections}
\crefname{appendix}{Appendix}{Appendices}
\crefname{introthm}{Theorem}{Theorems}
\crefname{introprop}{Proposition}{Propositions}
\crefname{introcor}{Corollary}{Corollaries}
\crefname{introconj}{Conjecture}{Conjectures}

\newtheorem{thm}{Theorem}[section]
\newtheorem{prop}[thm]{Proposition}
\newtheorem{cor}[thm]{Corollary}
\newtheorem{lem}[thm]{Lemma}

\newtheorem{introthm}{Theorem}
\newtheorem{introprop}[introthm]{Proposition}

\theoremstyle{definition}
\newtheorem{dfn}[thm]{Definition}

\newtheorem{conj}[thm]{Conjecture}

\theoremstyle{remark}

\newtheorem{rem}[thm]{Remark}

\newcommand{\beq}{\begin{equation}}
\newcommand{\eeq}{\end{equation}}

\newcommand{\ot}{\leftarrow}

\tikzset{
    squigarrow/.style={-{Classical TikZ Rightarrow[length=4pt]}, decorate, decoration={snake, amplitude=1.8pt, pre length=2pt, post length=3pt}}
}

\newcommand{\bC}{\mathbb C}

\newcommand{\bJ}{\mathbb J}
\newcommand{\bN}{\mathbb N}
\newcommand{\bP}{\mathbb P}

\newcommand{\bR}{\mathbb R}
\newcommand{\bT}{\mathbb T}
\newcommand{\bZ}{\mathbb Z}

\newcommand{\cA}{\mathcal A}

\newcommand{\cF}{\mathcal F}

\newcommand{\cN}{\mathcal N}
\newcommand{\cO}{\mathcal O}

\newcommand{\cS}{\mathcal S}

\newcommand{\cX}{\mathcal X}
\newcommand{\cY}{\mathcal Y}
\newcommand{\cZ}{\mathcal Z}

\newcommand{\fB}{\mathfrak{B}}

\newcommand{\fS}{\mathfrak S}

\newcommand{\sfC}{\mathsf{C}}

\newcommand{\sfH}{\mathsf{H}}
\newcommand{\sfK}{\mathsf{K}}
\newcommand{\sfP}{\mathsf{P}}
\newcommand{\sfQ}{\mathsf{Q}}

\newcommand{\bk}{\mathbf k}

\newcommand{\pr}{{\rm pr}}

\newcommand{\barA}{\overline{\cA}}
\newcommand{\barH}{\overline{\sfH}}
\newcommand{\barP}{\overline{\sfP}}
\newcommand{\barQ}{\overline{\sfQ}}

\newcommand{\Id}{{\rm Id}}

\DeclareMathOperator{\SL}{\mathrm{SL}}
\DeclareMathOperator{\skeleton}{\textbf{sk}}
\DeclareMathOperator{\im}{\mathrm{Im}}
\DeclareMathOperator{\kernel}{\mathrm{Ker}}
\DeclareMathOperator{\coker}{\mathrm{Coker}}

\DeclareMathOperator{\Int}{\mathrm{Int}}

\DeclareMathOperator{\Specm}{\mathrm{MaxSpec}}
\DeclareMathOperator{\Frac}{\mathrm{Frac}}

\DeclareMathOperator{\rankZ}{\text{rank}_{\mathcal{Z}}}

\newcommand{\gaa}{\mathsf{g}}
\newcommand{\tra}{{\rm tr}_{\lambda}^A}
\newcommand{\A}{\mathcal A_{\hat q}(\Sigma,\lambda)}
\newcommand{\Ap}{\mathcal A_{\hat{q}}^{+}(\Sigma,\lambda)}
\newcommand{\rA}{\overline{\mathcal A}_{\hat q}(\Sigma,\lambda)}
\newcommand{\rAp}{\overline{\mathcal A}^{+}_{\mathbbm{v}}(\Sigma,\lambda)}

\newcommand{\Sn}{\cS_n(\Sigma,{\mathbbm{v}})}

\newcommand{\cAl}{\cA_{\hat q}(\Sigma,\lambda)}
\newcommand{\cXl}{\cX_{\hat q}(\Sigma,\lambda)}
\newcommand{\barX}{\overline{\cX}}

\newcommand{\Xbll}{\cX^{\rm bl}_{\hat q}(\Sigma,\lambda)}

\def\bk{\mathbf{k}}

\def\barH{\overline{\sfH}}

\def\barK{\overline{\sfK}}
\def\barKl{\overline{\sfK}_{\lambda^\ast}}
\def\barKt{\overline{\sfK}_\tau}

\def\barQ{\overline{\sfQ}}
\def\barQl{\overline{\sfQ}_{\lambda^\ast}}
\def\barV{\overline{V}}
\def\barVl{\overline{V}_\lambda}
\def\barVlast{\overline{V}_{\lambda^\ast}}
\def\barVt{\overline{V}_\tau}
\def\obV{{\mathring{\barV}}}
\def\obVlast{\obV_{\lambda^\ast}}
\def\obVl{\obV_{\lambda}}

\def\Oq{\mathcal{O}_q(\mathrm{SL}(n))}
\def\Ov{\mathcal{O}_{q}(\mathrm{SL}(n))}
\def\Ou{\mathcal{O}_{\varepsilon}(\mathrm{SL}(n))}
\def\buu{{\mathbf u}}
\def\ot{\otimes}

\def\Bu{\cS_n(\fB,\mathbbm{u})}
\def\Bv{\cS_n(\fB,\mathbbm{v})}

\def\proj{\mathbf{pr}}

\def\tfk{\mathbf{k}}
\def\tft{\mathbf{t}}

\def \Zmp{\mathbb Z_{m'}}

\def\tfb{\mathbf{b}}
\def\tff{\mathbf{h}}
\def\ttk{\mathbf{k}}
\def\Lp{\Lambda_{\partial}}

\def\ZA{\mathcal Z(\A)}
\def\tZS{\text{Frac}(\mathcal Z(\Sn))}
\def\tZA{\text{Frac}(\mathcal Z(\A))}
\def\tSn{\widetilde{\cS_n}(\Sigma,\mathbbm{v})}
\def\tA{\widetilde{\cA_{\hat{q}}}(\Sigma,\lambda)}

\def\rSn{\overline \cS_n(\Sigma,\mathbbm{v})}
\def\Si{\Sigma}

\newcommand{\cev}[1]{\reflectbox{\ensuremath{\vec{\reflectbox{\ensuremath{#1}}}}}}
\DeclareMathOperator{\Irrep}{{\rm Irrep}}
\DeclareMathOperator{\Azumaya}{\mathsf{Azm}}

\title{Center of stated $\mathrm{SL}(n)$-skein algebras}

\author[Hiroaki Karuo]{Hiroaki Karuo}
\address{Hiroaki Karuo, Department of Mathematics, Gakushuin University, Mejiro, Toshima-ku, Tokyo, Japan.}
\email{hiroaki.karuo@gakushuin.ac.jp}

\author[Zhihao Wang]{Zhihao Wang}
\address{Zhihao Wang, School of Mathematics, Korea Institute for Advanced Study (KIAS), 85 Hoegi-ro, Dongdaemun-gu, Seoul 02455, Republic of Korea}
\email{zhihaowang@kias.re.kr}

\date{}

\begin{document}
\maketitle

\begin{abstract}
In the paper, we show some properties of (reduced) stated $\mathrm{SL}(n)$-skein algebras related to their centers for essentially bordered pb surfaces, especially their centers, finitely generation over their centers, and their PI-degrees. 
The proofs are based on the quantum trace maps, embeddings of (reduced) stated $\mathrm{SL}(n)$-skein algebras into quantum tori appearing in higher Teichm\"uller theory. 
Thanks to the Unicity theorem in [BG02, FKBL19], we can understand the representation theory of (reduced) stated $\mathrm{SL}(n)$-skein algebras. 
Moreover, the applications are beyond low-dimensional topology. For example, we can access to the representation theory of unrestricted quantum moduli algebras, and that of quantum higher cluster algebras potentially. 
\end{abstract}

\setcounter{tocdepth}{1}
\tableofcontents

\def\cR{\mathcal R}

\section{Introduction}
\subsection{Background}
Before explaining (reduced) stated $\SL(n)$-skein algebras, which is the subject of this paper, we will first explain $\SL(n)$-skein algebras. For a domain $\cR$ with a distinguished invertible element $\hat{q}$ and an oriented surface $\Sigma$, the $\SL(n)$-skein algebra is a noncommutative $\cR$-algebra associated with $\Sigma$ and is known to be a quantization of the $\SL(n,\bC)$-character variety of $\Sigma$ when $\cR=\bC$ \cite{Sik05, Bul97, PS19}. This algebra is generated by oriented $n$-valent graphs, and relations are imposed based on the $U_q(\frak{sl}_n)$-Reshetikhin--Turaev functor \cite{RT91}. While it is generally a noncommutative algebra, it is characterized by being understandable through diagrammatic computations. It is known that the $\SL(2)$-skein algebra is isomorphic to the Kauffman bracket skein algebra \cite{Prz99, Tur91}, and the $\SL(3)$-skein algebra is isomorphic to the Kuperberg skein algebra \cite{Kup96}. Specifically, in the case of $n=2$, it is deeply related to the quantum Teichm\"uller space \cite{BW11, Le19}, topological quantum field theory \cite{BHMV95}, and quantum moduli algebras \cite{Fai20, BFR23}, etc. 

As a generalization of the $\SL(2)$-skein algebra, the (reduced) stated $\SL(2)$-skein algebra was introduced by incorporating elements called stated skeins and imposing appropriate relations \cite{Le18, CL22}. 
This was inspired by the construction of the quantum trace map by Bonahon and Wong, which embeds the ${\rm SL}(2)$-skein algebra into the quantum Teichm\"uller space \cite{BW11}, and has the advantage of reducing discussions on complicated surfaces to simpler pieces like ideal triangles. 
Furthermore, this skein algebra is deeply related to quantum cluster algebras \cite{Mul16, LY23}, quantum Teichm\"uller spaces \cite{Le19, CL22}, topological quantum field theory \cite{CL22a}, and factorization homology \cite{BZBJ18, Coo23}. 
Through the skein algebra, phenomena in other fields can be captured diagrammatically, indicating its significance extends beyond low-dimensional topology. 

Similarly to the case of $n=2$, the (reduced) stated $\SL(n)$-skein algebra was introduced for general $n$ by considering stated $1$-$n$-valent graphs \cite{LS21, LY23}; see also \cite{Hig23} for $n=3$. This generalization is reasonable from the following results. 
\begin{enumerate}
    \item The stated SL$(n)$-skein algebra of a bigon is isomorphic to the quantum coordinate ring $\cO_q(\SL(n))$ as a Hopf algebra (for $n=2$, see \cite{CL22}; for general $n$, see \cite{LS21}).
    \item For a certain class of surfaces and  generic $\hat{q}$, it is isomorphic to the unrestricted quantum moduli algebra (also called the graph algebra) introduced by Alekseev--Grosse--Schomerus \cite{AGS95} and  Buffenoir--Roche \cite{BR95} independently \cite{Fai20, BFR23}.
    \item A homomorphism (quantum trace map) to the Fock--Goncharov algebra \cite{FG06, FG09} appearing in the context of the (quantum) higher Teichm\"uller theory is constructed, and under certain conditions, it is injective (for $n=2$, see \cite{BW11, Le19}; for $n=3$, see \cite{Kim20, Kim21, Dou24}; for general $n$, see \cite{LY23}).
\end{enumerate}

While we aim to contribute to other fields through understanding the structure of the stated $\SL(n)$-skein algebra, a fundamental and important first step is to understand the representation theory of the algebra. 
The Unicity theorem is effective in understanding the representation theory of noncommutative algebras \cite{FKBL19,BG02}. The Unicity theorem states that when an $\cR$-algebra is \textbf{almost Azumaya}, that is, if it satisfies 
\begin{enumerate}
    \item it is finitely generated as an $\cR$-algebra,
    \item it has no zero-divisors,
    \item it is finitely generated as a module over its center $Z$,
\end{enumerate}
then its finite-dimensional (irreducible) representations can be understood using the Azumaya locus (a Zariski dense open subset of MaxSpec($Z$)). Conditions (1) and (2) have already been shown in \cite{LY23}, so the remaining problem is to show the condition (3).

When determining the center of the (stated) $\SL(2)$ or $\SL(3)$-skein algebra, the use of its basis and an appropriate filtration was essential in \cite{FKBL19,KW24, Yu23}. However, applying this approach to the (stated) $\SL(n)$-skein algebra is challenging. One reason is that we do not know a `good' basis of the (stated) $\SL(n)$-skein algebra (no complexity for applying the confluence theory of \cite{SW07} has been found).

Bonahon--Higgins \cite{BH23} showed that, for general $n$, the image of a loop under `higher Chebyshev polynomials' gives as a central element of the $\SL(n)$-skein algebra, analogous to the case when $n=2$. However, these are multivariate polynomials defined by recurrence relations, which makes them more difficult to handle as elements of the skein algebra.
See \cite{KLW} for related works.

On the other hand, 
it is proved in \cite{Wan23}
that, analogous to the case of $n=2$ \cite{BL22,KQ24}, the $m'$-th power of an arc (a $1$-valent graph) is a central element of the stated $\SL(n)$-skein algebra. 
From the relation \eqref{eq:sticking} introduced in Section~\ref{sec:skein}, we know that every $n$-web (a generator of the $\SL(n)$-skein algebra) can be expressed as a sum and product of arcs, leading us to the understanding that certain images of all $n$-webs are included in the center. This implies that considering the stated $\SL(n)$-skein algebra, rather than the $\SL(n)$-skein algebra, allows for a more manageable treatment of the structures involved.

\subsection{Main results}
In this subsection, we will explain our results. 
Let $\hat q^2$ be a primitive $m''$-th root of unity. 
Let $d'$ denote the greatest common divisor of $n$ and $m''$, and set $m'=m''/d'$. Define $d$ to be the greatest common divisor of $2n$ and $m'$, and set $m =m'/d$. 

A {\bf pb surface} $\Sigma$ is obtained from a compact oriented surface $\overline{\Sigma}$ by removing finite points such that every boundary component of $\Sigma$ is diffeomorphic to $(0,1)$. 
A pb surface is {\bf essentially bordered} if every connected component has a non-empty boundary.

Fix a pb surface $\Sigma$ and let $\cS_n(\Sigma)$ (resp. $\overline\cS_n(\Sigma)$) denote the (resp. reduced) stated $\SL(n)$-skein algebra of $\Sigma$. See Section~1.3 for reasons why we deal with not only stated $\SL(n)$-skein algebras but also their reduced version.

As we shall see, the Unicity theorem (Theorem \ref{thm:Unicity-alomost}) is applied to the (reduced) stated $\SL(n)$-skein algebra. Thus studying the center of 
the (reduced) stated $\SL(n)$-skein algebra is crucial to understand the representation theory of the (reduced) stated $\SL(n)$-skein algebra.

Let $v$ be a boundary puncture of the pb surface $\Sigma$. For $i \in \{1, 2, \dots, n\}$, a stated $n$-web diagram $\gaa_{i}^{v} \in \cS_n(\Sigma)$ is defined in \cite{LY23} (we use $\gaa_{u_i}^e$ instead of $\gaa_{i}^{v}$ in Section~\ref{sub_center}), where $\gaa_{i}^{v}$ is equal, up to a power of $-q$, to
\begin{align}
\begin{array}{c}\includegraphics[scale=0.38]{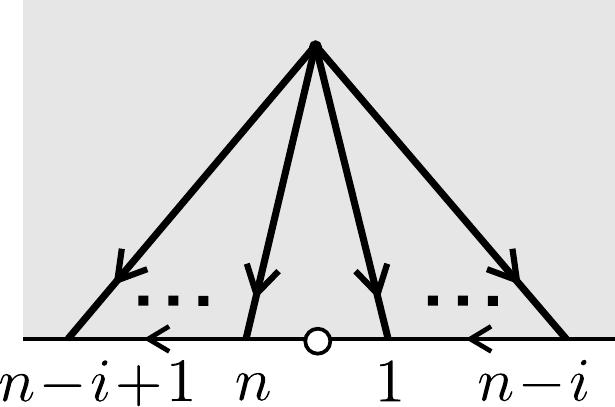}\end{array}
\end{align}
with the white dot being $v$. 
Consider an even boundary component $\partial$ of $\overline{\Sigma}$ whose punctures are labeled as $v_1, v_2, \cdots, v_r$ ($r$ is even) along the orientation of $\partial\overline{\Sigma}$. Then, for any $0 \leq k \leq m'$ and $1 \leq i \leq n-1$, 
\begin{align}\label{introeq-central}
(\gaa_{i}^{v_1})^k (\gaa_{i}^{v_2})^{m'-k} \cdots (\gaa_{i}^{v_{r-1}})^k (\gaa_{i}^{v_r})^{m'-k}
\end{align}
is a central element of $\cS_n(\Sigma)$ (Lemma~\ref{boundary_center}). 
Let $\mathsf{B}$ denote the set of the central elements in \eqref{introeq-central}. 
Then, we describe the center of the $A$-version quantum torus $\A$ explicitly. 
\begin{introthm}[Theorem~\ref{center_torus}]\label{intro:center_torus}
Let $\Sigma$ be a triangulable essentially
bordered pb surface without interior punctures, and $\lambda$ be a triangulation of $\Sigma$.
If $m''$ is odd, the center $\cZ(\A)$ of $\A$ is generated by $\tra(\mathsf{B})$ and $\{a^{\mathbf{k}}\mid \mathbf{k} \in \Lambda_{m'} \}$. 
\end{introthm}

To show Theorem~\ref{intro:center_torus}, we decompose a certain matrix into block matrices according to the interior of $\overline{\Sigma}$ and the boundary components of $\overline{\Sigma}$, and describe some of them explicitly. If a boundary component of $\overline{\Sigma}$ has even (resp. odd) number of punctures, the corresponding block matrix is non-invertible (resp. invertible). The invertibility of each block matrix is important to describe the center, and the non-invertibility specifically gives us the central element \eqref{introeq-central}.

Suppose $\Sigma$ is a triangulable pb surface with a triangulation $\lambda$.
As we shall see, the $X$-version (resp. $A$-version) quantum trace map defined in \cite{LY23} is an algebra homomorphism from the stated $\SL(n)$-skein algebra to the quantum torus associated with an anti-symmetric integer matrix indexed by the $V_{\lambda}$ (resp. $V_{\lambda}'$); see Section \ref{sec-quantumtrace}. 
There is a matrix $\mathsf{K}_\lambda$ in \eqref{eq-Klambda}, originally given in \cite{LY23}, 
defining a transition map between the above two quantum tori such that this transition map is compatible with the quantum trace maps \cite[Theorem 11.7]{LY23}. 
Define
$$\Lambda_{m'}^{+}= \{\mathbf{k}\in\mathbb N^{V_{\lambda}'} \mid 
\mathbf{k}\sfK_{\lambda} =\bm{0} \text{ in }\mathbb Z_{m'} \}.$$

Suppose $\Sigma$ is an essentially bordered pb surface and contains no interior punctures. Then for any $v\in V_{\lambda}'$, there is an important element $\gaa_v\in \cS_n(\Sigma)$, given in Section~\ref{sec:quantum_trace}.
We have $\gaa_u\gaa_v = \hat q^{2\mathsf P_{\lambda}(u,v)}\gaa_v\gaa_u$ for $u,v\in V_{\lambda}'$ from \cite[Lemmas 4.13 and 4.10]{LY23}.
Here ${\mathsf P}_{\lambda}$ is the anti-symmetric integer matrix for the $A$-version quantum torus.
Like the Weyl normalization in \eqref{eq:Weyl}, for any ${\bf k}\in\mathbb Z^{V_{\lambda}'}$, define
\begin{align}
\gaa^{\mathbf{k}} := [\prod_{v\in V_{\lambda}'}\gaa_v^{\mathbf{k}(v)}]_{\rm Weyl}\in {\cS}_n(\Sigma).
\end{align}
Let $\cY_{\lambda}$ denote the subalgebra of $\cS_n(\Sigma)$ generated by $\{\gaa^{\mathbf{k}}\mid \textbf k \in \Lambda_{m'}^{+}\}$ and $\mathsf{B}$.

Thanks to the 'sandwiched property' (Theorem~\ref{traceA}), we know that the center $\mathcal{Z}(\cS_n(\Sigma,\mathbbm{v}))$ of $\cS_n(\Sigma,\mathbbm{v})$ equals  $\cS_n(\Sigma,\mathbbm{v})\cap \cZ(\A)$ 
under the condition on $\Sigma$ in  Theorem~\ref{intro:center_torus}. 
Using this, we have the following.
\begin{introthm}[{Theorem~\ref{thm-main-center-skein}}]\label{mainthm-tro-1}
Let $\Sigma$ be a triangulable essentially bordered pb surface without interior punctures, and let $\lambda$ be a triangulation of $\Sigma$.
Assume that $m''$ is odd and $\hat{q}^2$ is a primitive $m''$-th root of unity.
 We have 
$$\mathcal{Z}(\cS_n(\Sigma)) = \{x\in \cS_n(\Sigma)\mid \text{$\exists  \mathbf{k}\in \mathbb{N}^{V_{\lambda}'}  $  such that  $ \gaa^{m'\mathbf{k}}x\in\mathcal Y_{\lambda}$}\}.$$
\end{introthm}

The following theorem and the Unicity theorem for almost Azumaya algebras, proved in \cite{BG02, FKBL19} (see Theorem \ref{thm:Unicity-alomost}), imply the Unicity theorem for (reduced) stated $\SL(n)$-skein algebras. This generalizes the work in \cite{Wan23} to all roots of unity. 
See \cite{FKBL19,Kor21} for related works when $n=2$, and \cite{KW24} when $n=3$.

\begin{introprop}[{Propositions~\ref{thm;azumaya} and \ref{thm-Unicity-reduced}}]\label{introthm:azumaya}
Suppose $\mathbbm{v}$ is a root of unity and $\Sigma$ be an essentially bordered pb surface. We have  the following:
(a) the stated $\SL(n)$-skein algebra $\cS_n(\Sigma)$ is almost Azumaya;
(b) if $n=2,3$, or $n>3$ and $\Sigma$ is a polygon, then the (reduced) stated $\SL(n)$-skein algebra 
$\overline\cS_n(\Sigma)$ is almost Azumaya.
\end{introprop}
Note that we will use the injectivity of the "quantum trace map" to describe the center of $\overline\cS_n(\Sigma)$ and related results. 
Due to the requirement of the condition (b) for injectivity in \cite{LY23}, we have to put the condition for some claims in the reduced case.

As we mentioned, to complete Theorem~\ref{introthm:azumaya}, the remaining part is to show the condition (3) in the definition of almost Azumaya. In stead of (3), we show finite generation as a module over the subalgebra generated by $m'$-th powers of arcs. This implies that we do not need the whole center to show (3).
We emphasize that Theorem~\ref{introthm:azumaya} does not require the condition that $\Sigma$ has no interior punctures. This means that our results are beyond consequences from 'sandwiched property'; see Section~\ref{sec:Azumaya} for more details.

The Unicity theorem classifies the finite-dimensional irreducible representations 
of the (reduced) stated $\SL(n)$-skein algebra. It states that any point in a Zariski open dense subset of the maximal spectrum of the center corresponds to a unique irreducible representation of the (reduced) stated $\SL(n)$-skein algebra. All these representations have the same dimension, which equals the square root of the rank defined in Section \ref{sub-Almost-Azumaya-algebra}.  The following theorem formulates this rank for the stated $\SL(n)$-skein algebra.

The following rank was not known and gives the highest dimension among finite dimensional irreducible representations of $\cS_n(\Sigma,\mathbbm{v})$. See Section~\ref{sec-Unicity-Theorem} for more details.
\begin{introthm}[{Theorem~\ref{thm:rank}}]\label{mainthm-tro-2}
Let $\Sigma$ be a triangulable essentially bordered pb surface without interior punctures, and
let $r(\Sigma) := \# (\partial \Sigma) - \chi(\Sigma)$, where $\chi(\Sigma)$ denotes the Euler characteristic of $\Sigma$.
Suppose $\overline\Sigma$ contains $t$ boundary components having even number of boundary punctures.
Assume that $m''$ is odd and $\hat{q}^2$ is a primitive $m''$-th root of unity.
We have $$\rankZ \cS_n(\Sigma,\mathbbm{v})=\rankZ \A= d^{r(\Sigma)-t}m^{(n^2-1)r(\Sigma)-t(n-1)}.$$
\end{introthm}

The isomorphism between stated skein algebras and unrestricted quantum moduli algebras (a.k.a graph algebras) in \cite{BFR23} should be extended to roots of unity. 
When $\Sigma$ has only one connected boundary component, \cite{BFR23} and the above results imply that we can access to the representation theory of the graph algebra obtained from $\Sigma$. 
In particular, this implies that Theorem~\ref{mainthm-tro-2} generalizes \cite[Theorem 1.3]{BR24} to the case when genus is  greater than $0$ in an appropriate sense. 
Baseilhac--Faitg--Roche \cite{BFR24} are also working to extend Theorem 1.3 of \cite{BR24} to arbitrary genera and simple Lie algebras.

Suppose $\Sigma$ is an essentially bordered pb surface and contains no interior punctures. Let $\lambda$ be a triangulation of $\Sigma$. 
Like the non-reduced case, there are two quantum trace maps, the $X$-version and the $A$-version, for the reduced stated $\SL(n)$-skein algebra. Both of the $X$-version and $A$-version quantum tori are associated to anti-symmetric integer matrices indexed by a finite set $\overline V_\lambda$; see
Section \ref{sec-quantumtrace}.
Like the non-reduced case, there is a transition matrix $\overline{\mathsf{K}}_\lambda$ in \eqref{eq-surgen-exp} for the reduced case, originally given in \cite{LY23}, between the $A$-version quantum torus and the $X$-version quantum torus.

Note that $\overline\cS_n(\Sigma)$ is a quotient algebra of $\cS_n(\Sigma)$ and $\overline V_\lambda\subset V_{\lambda}'$. 
For any $v\in \overline V_{\lambda}$, let $\bar\gaa_v\in \overline\cS_n(\Sigma)$ denote the image of $\gaa_v$ by the projection  $\cS_n(\Sigma)\to \overline\cS_n(\Sigma)$.
When the $A$-version quantum trace map is injective,  we have $\bar\gaa_u\bar\gaa_v = \hat q^{2\overline{\mathsf P}_{\lambda}(u,v)}\bar\gaa_v\bar\gaa_u$ for $u,v\in \overline{V}_{\lambda}$ from Theorem \ref{traceA}.
Here $\overline{\mathsf P}_{\lambda}$ is the anti-symmetric matrix for the $A$-version quantum torus.
Like the non-reduced case, for any ${\bf k}\in\mathbb Z^{\overline V_{\lambda}}$, we  define
\begin{align}
\bar \gaa^{\mathbf{k}} := [\prod_{v\in \overline V_{\lambda}}\bar\gaa_v^{\mathbf{k}(v)}]_{\rm Weyl}\in\overline {\cS}_n(\Sigma).
\end{align}

Suppose $\overline{\Sigma}$ has a boundary component $\partial$ such that  
the connected components of $\partial\cap \Sigma$ are labeled as $e_1,e_2,\cdots,e_r$ with respect to the orientation of $\partial$ (the orientation of $\partial$ is induced by the orientation of $\overline\Sigma$).
As we shall see in Section \ref{subsec:FGalg}, the finite set $\overline V_{\lambda}$ consists of points in $\Sigma$, called {\bf small vertices}.
For each $1\leq t\leq r$, there are $n-1$ small vertices contained in $e_t$, labeled as $u_{t,1},\cdots,u_{t,n-1}$ with respect to the orientation of $\partial$. 
 Suppose $r$ is even. For $k\in\mathbb N$ and $1\leq i\leq n-1$, the element
\begin{align}\label{eqintro-reduced-central1}
\bar \gaa_{u_{1,i}}^k \bar \gaa_{u_{2,n-i}}^k\cdots \bar \gaa_{u_{r-1,i}}^k \bar \gaa_{u_{r,n-i}}^k
\end{align}
is central in $\overline \cS_n(\Sigma)$ (Lemma \ref{reduced-boundary_center}). 
Suppose $r$ is odd. For  $k\in\mathbb N$ and $1\leq i\leq \lfloor\frac{n}{2}\rfloor$ ($\lfloor\,\cdot\, \rfloor$ denotes the floor function), the element
\begin{align}\label{eqintro-reduced-central2}
\bar \gaa_{u_{1,i}}^k \bar \gaa_{u_{1,n-i}}^k \bar \gaa_{u_{2,i}}^k \bar \gaa_{u_{2,n-i}}^k\cdots \bar \gaa_{u_{r,i}}^k \bar \gaa_{u_{r,n-i}}^k
\end{align}
is central in $\overline \cS_n(\Sigma)$ (Lemma \ref{reduced-boundary_center}).

Let $\overline \Lambda_{\partial}$ be the subset of $\mathbb Z^{\overline V_{\lambda}}$ describing the images of the central elements \eqref{eqintro-reduced-central1} and \eqref{eqintro-reduced-central2} by the $A$-version quantum trace map; see Section \ref{sec:central_generic} for the explicit definition of $\overline \Lambda_{\partial}$. 
We also define
$$\overline \Lambda_{m'}= \{\mathbf{k}\in\mathbb Z^{\overline V_{\lambda}} \mid 
\mathbf{k}\barK_{\lambda} =\bm{0} \text{ in }\mathbb Z_{m'}\}
,\quad \overline\Lambda_{m'}^{+}= \{\mathbf{k}\in\mathbb N^{\overline V_{\lambda}} \mid 
\mathbf{k} \barK_{\lambda} =\bm{0} \text{ in }\mathbb Z_{m'} \}.$$

Let $\overline \Lambda_z$ denote the subgroup of $\mathbb Z^{\overline V_{\lambda}}$ generated by $\overline \Lambda_{m'}$ and $\overline \Lambda_{\partial}$.

\begin{introthm}[Theorem~\ref{center_torus-reduced}]\label{thm:intro5}
Let $\Sigma$ be a triangulable essentially
bordered pb surface without interior punctures, and $\lambda=\mu$ be a triangulation of $\Sigma$ introduced in Section~\ref{sub:quantum-torus-reduced}.
If $m''$ is odd, $\mathcal Z(\rA)$ is generated by $\{a^{\mathbf{k}}\mid \mathbf{k} \in \overline\Lambda_\partial\cup \overline\Lambda_{m'} \}$
\end{introthm}

An advantage of the use of a certain triangulation $\mu$ is that we can reuse some arguments in non-reduced case.

Let $\overline{\cY}_{\lambda}$ denote the subalgebra of $\overline \cS_n(\Sigma)$ generated by $\{\bar\gaa^{\mathbf{k}}\mid \textbf k \in \overline\Lambda_{m'}^{+}\}$ and  the central elements in \eqref{eqintro-reduced-central1} and \eqref{eqintro-reduced-central2}.

Similarly to the non-reduced case, we know that the center $\mathcal{Z}(\overline\cS_n(\Sigma,\mathbbm{v}))$ of $\overline\cS_n(\Sigma,\mathbbm{v})$ equals  $\overline\cS_n(\Sigma,\mathbbm{v})\cap \cZ(\rA)$ from the 'sandwiched property' under the condition on $\Sigma$ in Theorem~\ref{introthm-main4}. 
Using this, we have the following.

\begin{introthm}[{Theorem~\ref{main-thm-reduced-center}}]
\label{introthm-main4}
Assume that $m''$ is odd and $\hat{q}^2$ is a primitive $m''$-th root of unity. Let $\Sigma$ be a triangulable essentially bordered pb surface without interior punctures, and let $\lambda=\mu$ be a triangulation of $\Sigma$ introduced in Section~\ref{sub:quantum-torus-reduced}. We require $\Sigma$ to be a polygon if $n>3$.
 We have 
$$\mathcal{Z}(\overline\cS_n(\Sigma,\mathbbm{v})) = \{x\in \overline\cS_n(\Sigma,\mathbbm{v})\mid \text{$\exists  \mathbf{k}\in \mathbb{N}^{\overline V_{\lambda}}  $  such that  $ \bar\gaa^{m'\mathbf{k}}x\in\overline{\mathcal Y}_{\lambda}$}\}.$$
\end{introthm}

The following theorem formulates the rank of the reduced stated $\SL(n)$-skein algebra over its center, which equals 
the square of
the maximal dimension of finite-dimensional irreducible representations of the reduced stated $\SL(n)$-skein algebra.

\begin{introthm}[{Theorems~\ref{thm-PI-reducedA} and \ref{main-thm-reduced-PI}}]
\label{introthm-main5}
Assume that $m''$ is odd and $\hat{q}^2$ is a primitive $m''$-th root of unity. 
Suppose $\Sigma$ is a triangulable essentially bordered pb surface and contains no interior punctures. 
Assume $\overline\Sigma$ has $b$ boundary components among which there are $t$ boundary components of $\overline{\Sigma}$ with even punctures. 

(a) We have 
$$ \rankZ \rA=
d^{r(\Sigma)-t} m^{(n^2-1)r(\Sigma)-\binom{n}{2}(\#\partial\Sigma)-t(n-1)-(b-t)\lfloor\frac{n}{2}\rfloor},$$
where $\binom{n}{2}=\frac{n(n-1)}{2}$ and $\lfloor\,\cdot\, \rfloor$ denotes the floor function. 

(b) In addition, we require $\Sigma$ to be a polygon if $n>3$. 
Then we have 
$$\rankZ\overline\cS_n(\Sigma,\mathbbm{v}) = \rankZ \rA=
d^{r(\Sigma)-t} m^{(n^2-1)r(\Sigma)-\binom{n}{2}(\#\partial\Sigma)-t(n-1)-(b-t)\lfloor\frac{n}{2}\rfloor}.$$
\end{introthm}

For Theorem~\ref{introthm-main4} and Theorem~\ref{introthm-main5}~(b), we require $\Sigma$ to be a polygon when $n>3$ because we need the injectivity of the $A$-version quantum trace map. 
In Theorem \ref{thm:intro5} and Theorem \ref{introthm-main5} (a), we formulate the center and the PI-degree of the $A$-quantum torus for the reduced case when the essentially bordered pb surface $\Sigma$ has no interior punctures, i.e., under much more mild conditions.
So, once we have the injectivity of the $A$-version quantum trace map for the reduced case, we could remove the restriction for $\Sigma$ being a polygon when $n>3$ in Theorem~\ref{introthm-main4} and Theorem~\ref{introthm-main5}~(b).

From \cite[Theorem 4.1]{New72}, the matrix $\sfP_\lambda$ (resp. $\overline{\sfP}_\lambda$) can be regarded as the direct sum of $2\times 2$ anti-symmetric integer matrices and zero-matrices. 
Note that each $2\times 2$ anti-symmetric matrix is determined by an integer and these integers capture the non-commutativity of $\cA_q(\Sigma,\lambda)$ (resp. $\overline{\cA}_q(\Sigma,\lambda)$). 
Hence, understanding of these integers is also important and we show each integer is $1$ or $n$ up to a power of $2$ and how many such numbers there are in Proposition~\ref{prop-anti-decom-P} using Theorems~\ref{mainthm-tro-2} (resp. Theorem~\ref{introthm-main5}).

Dealing with general $n$ is much more challenging than the cases when $n=2,3$. One reason, as we mentioned before, is that there is no `good' basis for the (reduced) stated $\SL(n)$-skein algebra. Meanwhile, there are some properties of general $n$ that are not reflected by the cases when $n=2,3$.

On each boundary component $\partial$ of $\overline\Sigma$, there are central elements of $\overline\cS_n(\Sigma)$ obtained by taking products of stated corner arcs around punctures in $\partial$; see Lemma \ref{reduced-boundary_center}.
When we consider these central elements for reduced stated $\SL(n)$-skein algebras, we need to consider two different cases as in Lemma \ref{reduced-boundary_center}.
When $r$ is odd ($r$ is the number of punctures contained in $\partial$), the boundary central elements in  \eqref{eqintro-reduced-central2}  satisfy some symmetric property, which is not reflected by the case when $n=2$ \cite{Kor21}.
To prove Theorem~\ref{introthm-main4} and Theorem~\ref{introthm-main5}~(b), we have to deal with this symmetric property.

When $n=2,3$, Theorems \ref{mainthm-tro-1}, \ref{mainthm-tro-2}, \ref{introthm-main4}, and \ref{introthm-main5} were studied in \cite{FKBL19,KW24,Kor21,Yu23}.
 When $n=2$ (resp. $n=3$), there are only three (resp. four) cases for $d$.
They considered these theorems by considering different cases for $d$.  
The number of cases of $d$ tends to infinity when $n$ is large enough. This makes it challenging to deal with general $n$.

Some results and techniques (e.g. describing the center and computing the PI-degree using a quantum torus) in the proofs can be applied not only to skein algebras but also to other non-commutative algebras with `sandwiched property'; see, e.g. Section~\ref{sec:Azumaya}. Hence, The impact of this paper is not limited to low-dimensional topology.

\subsection{Further directions}
In this paper, we deal with only odd roots of unity. 
We have already had some results for even roots of unity and will investigate this case subsequently \cite{KW25}.

In the current setting (without interior punctures), it is known that the reduced stated skein algebra $\overline{\cS}_2(\Sigma)$ is isomorphic to the quantum cluster algebra associated with $\mathfrak{sl}_2$ \cite{Mul16, LY22}. 
It is also known that $\overline{\cS}_3(\Sigma)$ is contained in the quantum cluster algebra associated with $\mathfrak{sl}_3$ in the same setting, and it is conjectured that they are isomorphic \cite{IY23, LY23}. 
Considering these results, it is expected that $\overline{\cS}_n(\Sigma)$ is isomorphic to the quantum cluster algebra associated with $\mathfrak{sl}_n$. 
If the expectation is correct, the results shown in this paper lead us to the understanding of the representation theory of quantum higher cluster algebras.

In the future, it is expected that the Unicity theorem will also be applicable to $\SL(n)$-skein algebras not treated in this paper. The benefit in such cases lies in the decomposition of the $\SL(n)$-character variety into symplectic leaves having a representation-theoretic interpretation for the skein algebra. Prior studies include \cite{GJS19, KK22, Yu23a, FKBL23, KW24}.

\subsection*{Acknowledgements}
The authors are grateful to St\'ephane Baseilhac, Matthieu Faitg, and Philippe Roche for sharing their results and a project in progress on (unrestricted) quantum moduli algebras. The authors are also grateful to Thang T. Q.  L\^e and H.-K. Kim for valuable comments on the introduction. 
H. K. was supported by JSPS KAKENHI Grant Number JP23K12976. 
Z. W. was supported by the NTU research scholarship from the Nanyang Technological University (Singapore) and the PhD scholarship from the University of Groningen (Netherlands).

\section{Preliminaries}
In this section, we will recall some definitions and known results related to (reduced) stated $\SL(n)$-skein algebras defined in \cite{LS21,LY23}.

\subsection{Notations}\label{notation}
Suppose $n\geq 2$ is an integer.  
Let $\bN$ denote the set of all non-negative integers, and $\bZ_k:=\bZ/k\bZ$.

Our ground ring is a commutative domain $\cR$ with an invertible element $\hat{q}$. We set 
$\mathbbm{v}= \hat{q}^n$, $q= \hat{q}^{2n^2}$ so that $q^{1/2n^2} = \hat q$ and $\mathbbm{v}^{1/n} = \hat q$, and define the following constants:
\begin{align*}
\mathbbm{c}_{i,\mathbbm{v}}= (-q)^{n-i} q^{\frac{n-1}{2n}},\quad
\mathbbm{t}_\mathbbm{v}= (-1)^{n-1} q^{\frac{n^2-1}{n}},\quad 
\mathbbm{a}_\mathbbm{v} =   q^{\frac{n+1-2n^2}{4}}.
\end{align*}

\paragraph{\textbf{Condition}}
We will refer the following conditions as Condition ($\ast$): \\
Let $\hat q^2$ is a primitive $m''$-th root of unity. 
Let $d'$ denote the greatest common divisor of $n$ and $m''$, and set $m'=m''/d'$. Define $d$ to be the greatest common divisor of $2n$ and $m'$, and set $m =m'/d$. Then we have $\mathbbm{v}^2$ is a primitive $m'$-th root of unity, and $q^2$ is a primitive $m$-th root of unity.

\subsection{Punctured bordered surfaces}
A {\bf punctured bordered surface} (or {\bf pb surface} for simplicity) $\Sigma$ is obtained from a compact oriented surface $\overline{\Si}$ by removing finite points, which are called punctures, such that every boundary component of $\Sigma$ is diffeomorphic to an open interval. The puncture in the interior of $\overline{\Si}$ is called the {\bf interior puncture}, the one that lies in $\partial \overline{\Si}$ is called the {\bf boundary puncture}.
An {\bf essentially bordered pb surface}
is a pb surface such that every connected component has a non-empty boundary.

An \textbf{even boundary component} 
(resp. \textbf{odd boundary component}) 
of $\overline{\Sigma}$ is a connected boundary component of $\overline{\Sigma}$ which has even 
(resp. odd) 
number of boundary punctures of $\Sigma$.

The orientation of $\partial \Si$ induced by the orientation of $\Si$ is called the {\bf positive orientation} of $\partial \Si$. The one opposite to the positive orientation of $\partial\Si$ is called the {\bf negative orientation} of $\partial \Si$.
In the paper, we always assume that the orientations of depicted 
surfaces point to the readers, i.e., the surfaces are equipped with the orientations in counter-clockwise.  
When we mention the orientation of a boundary component of a surface, we always mean its positive orientation, i.e., the orientation induced by the orientation of the surface.

An {\bf ideal arc} $c$ of $\Sigma$ is an embedding from $(0,1)$ to $\Sigma$ which can be extended to an immersion $[0,1]\to \overline{\Sigma}$ such that $c(0)$ and $c(1)$ are punctures. We identify an ideal arc with its image on $\Si$.
An ideal arc is \textbf{trivial} if it is null-homotopic.

For a pb surface $\Sigma$, 
$\Sigma$ is \textbf{triangulable} if every connected component of it has at least one ideal point and is neither the once- or twice-punctured sphere, the monogon, nor the bigon. 
An \textbf{(ideal) triangulation} of a triangulable surface $\Sigma$ is a maximal collection $\lambda$ of non-trivial ideal arcs which are pairwise disjoint except at punctures and pairwise non-isotopic. 
In this paper, we consider ideal triangulations up to isotopy.

\subsection{$n$-webs and their diagrams}
A thickening of a pb surface $\Sigma$ is $\Si\times(-1,1)$, we identify $\Si$ with $\Si\times \{0\}$. For any point $(x,t)\in\Si\times(-1,1)$, we call $t$ the \textbf{height} of this point. 

An {\bf $n$-web} $\alpha$ in $\Si\times(-1,1)$ is a disjoint union of oriented closed curves and a directed finite graph properly embedded into $\Si\times(-1,1)$, satisfying the following requirements:
\begin{enumerate}
    \item $\alpha$ only contains $1$-valent or $n$-valent vertices. Each $n$-valent vertex is a source or a  sink. The set of $1$-valent vertices is denoted as $\partial \alpha$, which are called \textbf{endpoints} of $\alpha$. For any boundary component $c$ of $\Si$, we require $\partial\alpha\cap (c\times(-1,1))$ have distinct heights.
    \item Every edge of the graph is an embedded oriented  closed interval  in $\Si\times(-1,1)$.
    \item $\alpha$ is equipped with a transversal \textbf{framing}. 
    \item The set of half-edges at each $n$-valent vertex is equipped with a  cyclic order. 
    \item $\partial \alpha$ is contained in $\partial\Si\times (-1,1)$ and the framing at these endpoints is given by the positive direction of $(-1,1)$.
\end{enumerate}
We will consider $n$-webs up to (ambient) \textbf{isotopy} which are continuous deformations of $n$-webs in their class. 
The empty $n$-web, denoted by $\emptyset$, is also considered as an $n$-web, with the convention that $\emptyset$ is only isotopic to itself. 

A {\bf state} for $\alpha$ is a map $s\colon\partial\alpha\rightarrow \{1,2,\cdots,n\}$. A {\bf stated $n$-web} in $\Si\times(-1,1)$ is an $n$-web equipped with a state.

We say the (stated) $n$-web $\alpha$ is in {\bf vertical position} if 
\begin{enumerate}
    \item the framing at everywhere is given by the positive direction of $(-1,1)$,
    \item $\alpha$ is in general position with respect to the projection  $\text{pr}\colon \Si\times(-1,1)\rightarrow \Si\times\{0\}$,
    \item at every $n$-valent vertex, the cyclic order of half-edges as the image of $\text{pr}$ is given by the positive orientation of $\Si$ (drawn counter-clockwise in pictures).
\end{enumerate}

For every (stated) $n$-web $\alpha$, we can isotope $\alpha$ to be in vertical position. For each boundary component $c$ of $\Si$, the heights of $\partial\alpha\cap (c\times(-1,1))$ determine a linear order on  $c\cap \text{pr}(\alpha)$.
Then a {\bf (stated) $n$-web diagram} of $\alpha$ is $\text{pr}(\alpha)$ equipped with the usual over/underpassing at each double point and a linear order on $c\cap \text{pr}(\alpha)$ for each boundary component $c$ of $\Si$.

A stated $n$-web diagram $\alpha$ is called {\bf negatively ordered}  if the linear order on $\alpha\cap c$, for each boundary component $c$ of $\Si$, is indicated by the negative orientation of $c$.

\subsection{Stated $\SL(n)$-skein algebras and their reduced version}\label{sec:skein}
Let $\fS_n$ denote the permutation group on the set $\{1,2,\cdots,n\}$. 
For an integer $i\in\{1,2,\cdots,n\}$, we use $\bar{i}$ to denote $n+1-i$. 
Recall that $q=\mathbbm{v}^{2n}$. 

\def\M {M,\cN}

The \textbf{stated $\SL(n)$-skein algebra} $\cS_n(\Si,\mathbbm{v})$ of $\Si$ is
the quotient module of the $\cR$-module freely generated by the set 
 of all isotopy classes of stated 
$n$-webs in $\Sigma\times (-1,1)$ subject to  relations \eqref{w.cross}-\eqref{wzh.eight}.

\beq\label{w.cross}
q^{\frac{1}{n}} 
\raisebox{-.20in}{

\begin{tikzpicture}
\tikzset{->-/.style=

{decoration={markings,mark=at position #1 with

{\arrow{latex}}},postaction={decorate}}}
\filldraw[draw=white,fill=gray!20] (-0,-0.2) rectangle (1, 1.2);
\draw [line width =1pt,decoration={markings, mark=at position 0.5 with {\arrow{>}}},postaction={decorate}](0.6,0.6)--(1,1);
\draw [line width =1pt,decoration={markings, mark=at position 0.5 with {\arrow{>}}},postaction={decorate}](0.6,0.4)--(1,0);
\draw[line width =1pt] (0,0)--(0.4,0.4);
\draw[line width =1pt] (0,1)--(0.4,0.6);
\draw[line width =1pt] (0.4,0.6)--(0.6,0.4);
\end{tikzpicture}
}
- q^{-\frac {1}{n}}
\raisebox{-.20in}{
\begin{tikzpicture}
\tikzset{->-/.style=

{decoration={markings,mark=at position #1 with

{\arrow{latex}}},postaction={decorate}}}
\filldraw[draw=white,fill=gray!20] (-0,-0.2) rectangle (1, 1.2);
\draw [line width =1pt,decoration={markings, mark=at position 0.5 with {\arrow{>}}},postaction={decorate}](0.6,0.6)--(1,1);
\draw [line width =1pt,decoration={markings, mark=at position 0.5 with {\arrow{>}}},postaction={decorate}](0.6,0.4)--(1,0);
\draw[line width =1pt] (0,0)--(0.4,0.4);
\draw[line width =1pt] (0,1)--(0.4,0.6);
\draw[line width =1pt] (0.6,0.6)--(0.4,0.4);
\end{tikzpicture}
}
= (q-q^{-1})
\raisebox{-.20in}{

\begin{tikzpicture}
\tikzset{->-/.style=

{decoration={markings,mark=at position #1 with

{\arrow{latex}}},postaction={decorate}}}
\filldraw[draw=white,fill=gray!20] (-0,-0.2) rectangle (1, 1.2);
\draw [line width =1pt,decoration={markings, mark=at position 0.5 with {\arrow{>}}},postaction={decorate}](0,0.8)--(1,0.8);
\draw [line width =1pt,decoration={markings, mark=at position 0.5 with {\arrow{>}}},postaction={decorate}](0,0.2)--(1,0.2);
\end{tikzpicture}
},
\eeq 
\beq\label{w.twist}
\raisebox{-.15in}{
\begin{tikzpicture}
\tikzset{->-/.style=
{decoration={markings,mark=at position #1 with
{\arrow{latex}}},postaction={decorate}}}
\filldraw[draw=white,fill=gray!20] (-1,-0.35) rectangle (0.6, 0.65);
\draw [line width =1pt,decoration={markings, mark=at position 0.5 with {\arrow{>}}},postaction={decorate}](-1,0)--(-0.25,0);
\draw [color = black, line width =1pt](0,0)--(0.6,0);
\draw [color = black, line width =1pt] (0.166 ,0.08) arc (-37:270:0.2);
\end{tikzpicture}}
= \mathbbm{t}_\mathbbm{v}
\raisebox{-.15in}{
\begin{tikzpicture}
\tikzset{->-/.style=
{decoration={markings,mark=at position #1 with
{\arrow{latex}}},postaction={decorate}}}
\filldraw[draw=white,fill=gray!20] (-1,-0.5) rectangle (0.6, 0.5);
\draw [line width =1pt,decoration={markings, mark=at position 0.5 with {\arrow{>}}},postaction={decorate}](-1,0)--(-0.25,0);
\draw [color = black, line width =1pt](-0.25,0)--(0.6,0);
\end{tikzpicture}}
,
\eeq
\beq\label{w.unknot}
\raisebox{-.20in}{
\begin{tikzpicture}
\tikzset{->-/.style=
{decoration={markings,mark=at position #1 with
{\arrow{latex}}},postaction={decorate}}}
\filldraw[draw=white,fill=gray!20] (0,0) rectangle (1,1);
\draw [line width =1pt,decoration={markings, mark=at position 0.5 with {\arrow{>}}},postaction={decorate}](0.45,0.8)--(0.55,0.8);
\draw[line width =1pt] (0.5 ,0.5) circle (0.3);
\end{tikzpicture}}
= (-1)^{n-1} [n]\ 
\raisebox{-.20in}{
\begin{tikzpicture}
\tikzset{->-/.style=
{decoration={markings,mark=at position #1 with
{\arrow{latex}}},postaction={decorate}}}
\filldraw[draw=white,fill=gray!20] (0,0) rectangle (1,1);
\end{tikzpicture}}
,\ \text{where}\ [n]=\frac{q^n-q^{-n}}{q-q^{-1}},
\eeq
\beq\label{wzh.four}
\raisebox{-.30in}{
\begin{tikzpicture}
\tikzset{->-/.style=
{decoration={markings,mark=at position #1 with
{\arrow{latex}}},postaction={decorate}}}
\filldraw[draw=white,fill=gray!20] (-1,-0.7) rectangle (1.2,1.3);
\draw [line width =1pt,decoration={markings, mark=at position 0.5 with {\arrow{>}}},postaction={decorate}](-1,1)--(0,0);
\draw [line width =1pt,decoration={markings, mark=at position 0.5 with {\arrow{>}}},postaction={decorate}](-1,0)--(0,0);
\draw [line width =1pt,decoration={markings, mark=at position 0.5 with {\arrow{>}}},postaction={decorate}](-1,-0.4)--(0,0);
\draw [line width =1pt,decoration={markings, mark=at position 0.5 with {\arrow{<}}},postaction={decorate}](1.2,1)  --(0.2,0);
\draw [line width =1pt,decoration={markings, mark=at position 0.5 with {\arrow{<}}},postaction={decorate}](1.2,0)  --(0.2,0);
\draw [line width =1pt,decoration={markings, mark=at position 0.5 with {\arrow{<}}},postaction={decorate}](1.2,-0.4)--(0.2,0);
\node  at(-0.8,0.5) {$\vdots$};
\node  at(1,0.5) {$\vdots$};
\end{tikzpicture}}=(-q)^{\frac{n(n-1)}{2}}\cdot \sum_{\sigma\in \fS_n}
(-q^{\frac{1-n}n})^{\ell(\sigma)} \raisebox{-.30in}{
\begin{tikzpicture}
\tikzset{->-/.style=
{decoration={markings,mark=at position #1 with
{\arrow{latex}}},postaction={decorate}}}
\filldraw[draw=white,fill=gray!20] (-1,-0.7) rectangle (1.2,1.3);
\draw [line width =1pt,decoration={markings, mark=at position 0.5 with {\arrow{>}}},postaction={decorate}](-1,1)--(0,0);
\draw [line width =1pt,decoration={markings, mark=at position 0.5 with {\arrow{>}}},postaction={decorate}](-1,0)--(0,0);
\draw [line width =1pt,decoration={markings, mark=at position 0.5 with {\arrow{>}}},postaction={decorate}](-1,-0.4)--(0,0);
\draw [line width =1pt,decoration={markings, mark=at position 0.5 with {\arrow{<}}},postaction={decorate}](1.2,1)  --(0.2,0);
\draw [line width =1pt,decoration={markings, mark=at position 0.5 with {\arrow{<}}},postaction={decorate}](1.2,0)  --(0.2,0);
\draw [line width =1pt,decoration={markings, mark=at position 0.5 with {\arrow{<}}},postaction={decorate}](1.2,-0.4)--(0.2,0);
\node  at(-0.8,0.5) {$\vdots$};
\node  at(1,0.5) {$\vdots$};
\filldraw[draw=black,fill=gray!20,line width =1pt]  (0.1,0.3) ellipse (0.4 and 0.7);
\node  at(0.1,0.3){$\sigma_{+}$};
\end{tikzpicture}},
\eeq
where the ellipse enclosing $\sigma_+$  is the minimum crossing positive braid representing a permutation $\sigma\in \fS_n$ and $\ell(\sigma)=\#\{(i,j)\mid 1\leq i<j\leq n,\ \sigma(i)>\sigma(j)\}$ is the length of $\sigma\in \fS_n$.

\beq
   \raisebox{-.30in}{
\begin{tikzpicture}
\tikzset{->-/.style=
{decoration={markings,mark=at position #1 with
{\arrow{latex}}},postaction={decorate}}}
\filldraw[draw=white,fill=gray!20] (-1,-0.7) rectangle (0.2,1.3);
\draw [line width =1pt](-1,1)--(0,0);
\draw [line width =1pt](-1,0)--(0,0);
\draw [line width =1pt](-1,-0.4)--(0,0);
\draw [line width =1.5pt](0.2,1.3)--(0.2,-0.7);
\node  at(-0.8,0.5) {$\vdots$};
\filldraw[fill=white,line width =0.8pt] (-0.5 ,0.5) circle (0.07);
\filldraw[fill=white,line width =0.8pt] (-0.5 ,0) circle (0.07);
\filldraw[fill=white,line width =0.8pt] (-0.5 ,-0.2) circle (0.07);
\end{tikzpicture}}
   = 
   \mathbbm{a}_\mathbbm{v} \sum_{\sigma \in \fS_n} (-q)^{\ell(\sigma)}\,  \raisebox{-.30in}{
\begin{tikzpicture}
\tikzset{->-/.style=
{decoration={markings,mark=at position #1 with
{\arrow{latex}}},postaction={decorate}}}
\filldraw[draw=white,fill=gray!20] (-1,-0.7) rectangle (0.2,1.3);
\draw [line width =1pt](-1,1)--(0.2,1);
\draw [line width =1pt](-1,0)--(0.2,0);
\draw [line width =1pt](-1,-0.4)--(0.2,-0.4);
\draw [line width =1.5pt,decoration={markings, mark=at position 1 with {\arrow{>}}},postaction={decorate}](0.2,1.3)--(0.2,-0.7);
\node  at(-0.8,0.5) {$\vdots$};
\filldraw[fill=white,line width =0.8pt] (-0.5 ,1) circle (0.07);
\filldraw[fill=white,line width =0.8pt] (-0.5 ,0) circle (0.07);
\filldraw[fill=white,line width =0.8pt] (-0.5 ,-0.4) circle (0.07);
\node [right] at(0.2,1) {$\sigma(n)$};
\node [right] at(0.2,0) {$\sigma(2)$};
\node [right] at(0.2,-0.4){$\sigma(1)$};
\end{tikzpicture}},\label{eq:sticking}
\eeq
\beq \label{wzh.six}
\raisebox{-.20in}{
\begin{tikzpicture}
\tikzset{->-/.style=
{decoration={markings,mark=at position #1 with
{\arrow{latex}}},postaction={decorate}}}
\filldraw[draw=white,fill=gray!20] (-0.7,-0.7) rectangle (0,0.7);
\draw [line width =1.5pt,decoration={markings, mark=at position 1 with {\arrow{>}}},postaction={decorate}](0,0.7)--(0,-0.7);
\draw [color = black, line width =1pt] (0 ,0.3) arc (90:270:0.5 and 0.3);
\node [right]  at(0,0.3) {$i$};
\node [right] at(0,-0.3){$j$};
\filldraw[fill=white,line width =0.8pt] (-0.5 ,0) circle (0.07);
\end{tikzpicture}}   = \delta_{\bar j,i }\,  \mathbbm{c}_{i,\mathbbm{v}}\ \raisebox{-.20in}{
\begin{tikzpicture}
\tikzset{->-/.style=
{decoration={markings,mark=at position #1 with
{\arrow{latex}}},postaction={decorate}}}
\filldraw[draw=white,fill=gray!20] (-0.7,-0.7) rectangle (0,0.7);
\draw [line width =1.5pt](0,0.7)--(0,-0.7);
\end{tikzpicture}},
\eeq
\beq \label{wzh.seven}
\raisebox{-.20in}{
\begin{tikzpicture}
\tikzset{->-/.style=
{decoration={markings,mark=at position #1 with
{\arrow{latex}}},postaction={decorate}}}
\filldraw[draw=white,fill=gray!20] (-0.7,-0.7) rectangle (0,0.7);
\draw [line width =1.5pt](0,0.7)--(0,-0.7);
\draw [color = black, line width =1pt] (-0.7 ,-0.3) arc (-90:90:0.5 and 0.3);
\filldraw[fill=white,line width =0.8pt] (-0.55 ,0.26) circle (0.07);
\end{tikzpicture}}
= \sum_{i=1}^n  (\mathbbm{c}_{\bar i,\mathbbm{v}})^{-1}\, \raisebox{-.20in}{
\begin{tikzpicture}
\tikzset{->-/.style=
{decoration={markings,mark=at position #1 with
{\arrow{latex}}},postaction={decorate}}}
\filldraw[draw=white,fill=gray!20] (-0.7,-0.7) rectangle (0,0.7);
\draw [line width =1.5pt,decoration={markings, mark=at position 1 with {\arrow{>}}},postaction={decorate}](0,0.7)--(0,-0.7);
\draw [line width =1pt](-0.7,0.3)--(0,0.3);
\draw [line width =1pt](-0.7,-0.3)--(0,-0.3);
\filldraw[fill=white,line width =0.8pt] (-0.3 ,0.3) circle (0.07);
\filldraw[fill=black,line width =0.8pt] (-0.3 ,-0.3) circle (0.07);
\node [right]  at(0,0.3) {$i$};
\node [right]  at(0,-0.3) {$\bar{i}$};
\end{tikzpicture}},
\eeq
\beq\label{wzh.eight}
\raisebox{-.20in}{

\begin{tikzpicture}
\tikzset{->-/.style=

{decoration={markings,mark=at position #1 with

{\arrow{latex}}},postaction={decorate}}}
\filldraw[draw=white,fill=gray!20] (-0,-0.2) rectangle (1, 1.2);
\draw [line width =1.5pt,decoration={markings, mark=at position 1 with {\arrow{>}}},postaction={decorate}](1,1.2)--(1,-0.2);
\draw [line width =1pt](0.6,0.6)--(1,1);
\draw [line width =1pt](0.6,0.4)--(1,0);
\draw[line width =1pt] (0,0)--(0.4,0.4);
\draw[line width =1pt] (0,1)--(0.4,0.6);
\draw[line width =1pt] (0.4,0.6)--(0.6,0.4);
\filldraw[fill=white,line width =0.8pt] (0.2 ,0.2) circle (0.07);
\filldraw[fill=white,line width =0.8pt] (0.2 ,0.8) circle (0.07);
\node [right]  at(1,1) {$i$};
\node [right]  at(1,0) {$j$};
\end{tikzpicture}
} =q^{-\frac{1}{n}}\left(\delta_{{j<i} }(q-q^{-1})\raisebox{-.20in}{

\begin{tikzpicture}
\tikzset{->-/.style=

{decoration={markings,mark=at position #1 with

{\arrow{latex}}},postaction={decorate}}}
\filldraw[draw=white,fill=gray!20] (-0,-0.2) rectangle (1, 1.2);
\draw [line width =1.5pt,decoration={markings, mark=at position 1 with {\arrow{>}}},postaction={decorate}](1,1.2)--(1,-0.2);
\draw [line width =1pt](0,0.8)--(1,0.8);
\draw [line width =1pt](0,0.2)--(1,0.2);
\filldraw[fill=white,line width =0.8pt] (0.2 ,0.8) circle (0.07);
\filldraw[fill=white,line width =0.8pt] (0.2 ,0.2) circle (0.07);
\node [right]  at(1,0.8) {$i$};
\node [right]  at(1,0.2) {$j$};
\end{tikzpicture}
}+q^{\delta_{i,j}}\raisebox{-.20in}{

\begin{tikzpicture}
\tikzset{->-/.style=

{decoration={markings,mark=at position #1 with

{\arrow{latex}}},postaction={decorate}}}
\filldraw[draw=white,fill=gray!20] (-0,-0.2) rectangle (1, 1.2);
\draw [line width =1.5pt,decoration={markings, mark=at position 1 with {\arrow{>}}},postaction={decorate}](1,1.2)--(1,-0.2);
\draw [line width =1pt](0,0.8)--(1,0.8);
\draw [line width =1pt](0,0.2)--(1,0.2);
\filldraw[fill=white,line width =0.8pt] (0.2 ,0.8) circle (0.07);
\filldraw[fill=white,line width =0.8pt] (0.2 ,0.2) circle (0.07);
\node [right]  at(1,0.8) {$j$};
\node [right]  at(1,0.2) {$i$};
\end{tikzpicture}
}\right),
\eeq
where   
$\delta_{j<i}= 
\begin{cases}
1  & j<i\\
0 & \text{otherwise}
\end{cases},\ 
\delta_{i,j}= 
\begin{cases} 
1  & i=j\\
0  & \text{otherwise}
\end{cases}$, and the white dot and the black dot represent the direction of the corresponding stated $n$-webs.  
The arrows on the boundary edges represent the height orders of the endpoints of the $n$-webs with increasing heights in the direction of the arrow.

 \def \Sv{\cS_n(\Si,\mathbbm{v})}

 The algebra structure for $\cS_n(\Sigma,\mathbbm{v})$ is given by stacking the stated $n$-webs, i.e., for any two stated $n$-webs $\alpha,\alpha'\in\Si\times(-1,1)$, the product $\alpha\alpha'$ is defined by stacking $\alpha$ above $\alpha'$.

For a boundary puncture $p$ of a pb surface $\Sigma$, a corner arcs $C(p)_{ij}$ and $\cev{C}(p)_{ij}$ are stated arcs depicted as in Figure \ref{Fig;badarc}.
For a boundary puncture $p$ which is not on a monogon component of $\Sigma$, set 
$$C_p=\{C(p)_{ij}\mid i<j\},\quad\cev{C}_p=\{\cev{C}(p)_{ij}\mid i<j\}.$$  
Each element of $C_p\cup \cev{C}_p$ is called a \emph{bad arc} at $p$. 
\begin{figure}[h]
    \centering
    \includegraphics[width=150pt]{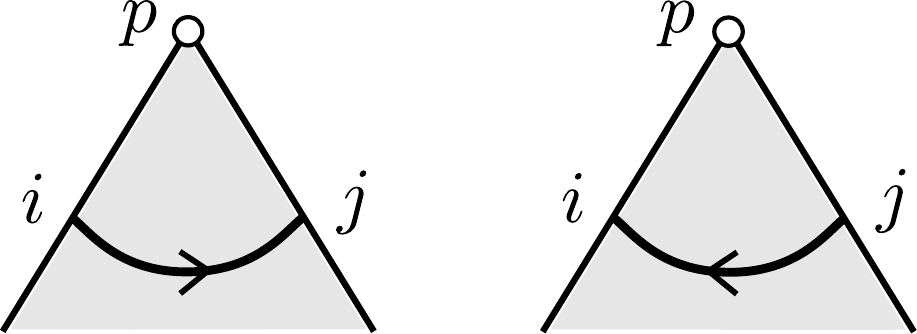}
    \caption{The left is $C(p)_{ij}$ and the right is $\cev{C}(p)_{ij}$.}\label{Fig;badarc}
\end{figure}

\def\barS{\overline{\cS}}
For a pb surface $\Sigma$, $$\barS_n(\Sigma,\mathbbm{v}) = \cS_n(\Sigma,\mathbbm{v})/I^{\text{bad}}$$ 
is called the \textbf{reduced stated $\SL(n)$-skein algebra}, defined in \cite{LY23}, where $I^{\text{bad}}$ is the two-sided ideal of $\cS_n(\Sigma,\mathbbm{v})$ generated by all bad arcs.

\subsection{The splitting map}

Let $e$ be an ideal arc of a pb surface $\Sigma$ such that it is contained in the interior of $\Sigma$. After cutting $\Sigma$ along $e$, we get a new pb surface $\text{Cut}_e\,{\Sigma}$, which has two copies $e_1,e_2$ for $e$ such that 
${\Sigma}= \text{Cut}_e\,{\Sigma}/(e_1=e_2)$. We use $\pr$ to denote the projection from $\text{Cut}_e\,{\Sigma}$ to $\Sigma$.  Suppose $\alpha$ is  a stated $n$-web diagram in $\Sigma$, which is transverse to $e$.
Let $s$ be a map from $e\cap\alpha$ to $\{1,2,\cdots,n\}$, and let $h$ be a linear order on $e\cap\alpha$. Then there is a lift diagram $\alpha(h,s)$ for a stated $n$-web in $\text{Cut}_e\,{\Sigma}$. 
 For $i=1,2$, the heights of the endpoints of $\alpha(h,s)$ on $e_i$ are induced by $h$ (via $\pr$), and the states of the endpoints of $\alpha(h,s)$ on $e_i$ are induced by $s$ (via $\pr$).
Then the splitting map is defined by 
$$\Theta_e(\alpha) =\sum_{s\colon \alpha \cap e \to \{1,\cdots, n\}} \alpha(h, s). $$
Furthermore $\Theta_e$ is an $\cR$-algebra homomorphism \cite{LS21}. When there is no confusion we  will omit the subscript for $\Theta_e$.

\subsection{The bigon and $\cO_q(\SL(n))$}\label{sec:bigon}

 Here we refer to \cite{KS12,LS21,LY23} for the definition of $\cO_q(\SL(n))$.
 
The {\bf bigon} $\fB$ is obtained from a closed disk $D$ by removing two points in $\partial D$. We can label the two boundary components of the bigon  by $e_l$ and $e_r$. a bigon with this labeling is called a {\bf directed bigon},
 see an example
$
\raisebox{-.20in}{

\begin{tikzpicture}
\tikzset{->-/.style=

{decoration={markings,mark=at position #1 with

{\arrow{latex}}},postaction={decorate}}}

\filldraw[draw=black,fill=gray!20] (0.5 ,0.5) circle (0.5);
\filldraw[draw=black,fill=white] (0.5,0) circle (0.05);
\filldraw[draw=black,fill=white] (0.5,1) circle (0.05);
\node [left] at(0,0.5) {$e_{l}$};
\node [right] at(1,0.5) {$e_{r}$};
\end{tikzpicture}
}
$. 
We can draw $\fB$ like $
\raisebox{-.20in}{

\begin{tikzpicture}
\tikzset{->-/.style=

{decoration={markings,mark=at position #1 with

{\arrow{latex}}},postaction={decorate}}}

\filldraw[draw=white,fill=gray!20] (0,0) rectangle (1, 1);
\draw[line width =1pt] (0,0)--(0,1);
\draw[line width =1pt] (1,0)--(1,1);
\end{tikzpicture}
}
$, and use $u_{ij}$
to denote 
$
\raisebox{-.20in}{

\begin{tikzpicture}
\tikzset{->-/.style=

{decoration={markings,mark=at position #1 with

{\arrow{latex}}},postaction={decorate}}}

\filldraw[draw=white,fill=gray!20] (0,0) rectangle (1, 1);
\draw [line width =1pt,decoration={markings, mark=at position 0.5 with {\arrow{>}}},postaction={decorate}](0,0.5)--(1,0.5);
\draw[line width =1pt] (0,0)--(0,1);
\draw[line width =1pt] (1,0)--(1,1);
\node [left] at(0,0.5) {$i$};
\node [right] at(1,0.5) {$j$};
\end{tikzpicture}
}
$, where $i,j\in\{1,2,\cdots,n\}$.
We also use $\cev{u}_{ij}$ to denote the stated $n$-web diagram in $\fB$ obtained from $a_{ij}$ by reversing its orientation.



For
$i,j,k,l\in \{1,2,\cdot,n\}$,
we have the following coefficients
\beq
 \cR^{ij}_{lk} = q^{-\frac 1n} \left(    q^{ \delta_{i,j}} \delta_{jk} \delta_{il} + (q-q^{-1})
    \delta_{j<i} \delta_{jl} \delta_{ik}\right),
 \label{R}
\eeq
where $\delta_{j<i}=1$ if $j<i$ and $\delta_{j<i}=0$ otherwise.

Let $\cO_q(M(n))$ be the associative algebra generated by  $u_{ij}$,   $i,j\in\{1,2,\cdots,n\},$
subject to the relations 
\beq
(\buu \ot \buu) \cR = \cR (\buu \ot \buu),  
\eeq
where $\cR$ is the $n^2\times n^2$ matrix given by equation \eqref{R}, and $\buu \ot \buu$ is the $n^2\times n^2$ matrix with entries $(\buu \ot \buu)^{ik}_{jl} = u_{i,j} u_{k,l}$ for $i,j,k,l\in \{1,2,\cdots,n\}$. 

Define  the element 
$$ {\det}_q(\buu):= \sum_{\sigma\in S_n} (-q)^{\ell(\sigma)}u_{1\sigma(1)}\cdots u_{n\sigma(n)} = \sum_{\sigma\in S_n} (-q)^{\ell(\sigma)}u_{\sigma(1)1}\cdots u_{\sigma(n)n}.$$

Define $\Oq$ to be  $\cO_q(M(n))/(\det_q(\buu)-1).$ Then 
$\Oq$ is a Hopf algebra with the Hopf algebra structure given by
\begin{align*}
\Delta(u_{ij})  = \sum_{k=1}^n u_{ik} \ot u_{kj}, \quad  \epsilon(u_{ij})= \delta_{ij},\quad 
S({u}_{ij} )
= (-q)^{i-j} {\det}_q(\buu^{ji}), 
\end{align*}
where $\Delta,\ \epsilon$ and $S$ denote the comultiplication, the counit and the antipode respectively, and 
$\buu^{ji}$ is the result of removing the $j$-th row and $i$-th column from $\buu$.

\begin{thm}[\cite{LS21}]\label{Hopf} 
(a)  $\cS_n(\fB,\mathbbm{v})$ is a Hopf algebra over $\cR$.

(b) There is a unique Hopf algebra isomorphism  $g_{big}\colon \cO_{q}(SL_n) \rightarrow \cS_n(\fB,\mathbbm{v}) $  defined by
  $ g_{big} (u_{ij}) = a_{ij}$.
\end{thm}

\subsection{The saturated system}\label{sec;sat_sys}
Let $f\colon\Sigma\rightarrow\Sigma'$ be a proper embedding between two pb surfaces. For a negatively ordered stated $n$-web diagram $\alpha$, we define $f_*(\alpha)$ to be a negatively ordered  stated $n$-web diagram $f(\alpha)$ in $\Sigma'$. It is easy to show that $f_*\colon\Sv\rightarrow \cS_n(\Si',\mathbbm{v})$ is an $\cR$-linear map.

For an essentially bordered pb  surface $\Sigma$, let  $B = \{b_1,\cdots,b_r \}$ be the collection of properly embedded
disjoint compact oriented arcs in $\Sigma.$ We say $B$ is a {\bf saturated system} if we have (1) after cutting $\Sigma$ along $B$, every component of the  resulting surface contains exactly one ideal point (2) $B$ is maximal under the condition (1). The existence of a saturated system is ensured as follows. First, we can cut $\Sigma$ so as to be a marked disk. Here the cutting paths are a part of a saturated system. Then the existence of a saturated system is obvious for marked disks. More precisely, any saturated system of a polygon with $k$ edges consists of $k-1$ disjoint oriented arcs.

Let $U(b_1),\cdots,U(b_r)$ be  disjoint open tubular neighborhoods of $b_1,\cdots,b_r$,  respectively. Each $U(b_i)$ is diffeomorphic with $b_i \times(0,1)$ (the diffeomorphism is orientation preserving)  and we require that $(\partial b_i) \times (0,1) \subset\partial\Sigma$. Set $U(B) = \cup_{1\leq i\leq r}U(b_i)$. Note that each $U(b_i)$ can be seen a bigon naturally, then $U(B)$ is the disjoint union of bigons.

\begin{thm}[\cite{LS21}]\label{key}
Assume  $B = \{b_1,\cdots,b_r \}$ is a  saturated system of arcs of $\Sigma$.

 (1) We have $r = r(\Sigma) := \# (\partial \Sigma) - \chi(\Sigma)$, where $\#( \partial \Sigma)$ is the number of boundary components  of $\Sigma$ and $\chi(\Sigma)$ denotes the Euler characteristics of $\Sigma$.

 (2) The embedding $U(B)\rightarrow \Sigma$  induces  an $\cR$-module isomorphism $\mathcal S_n(U(B),\mathbbm{v})\rightarrow \mathcal S_n(\Sigma,\mathbbm{v})$.
\end{thm}

\section{Quantum tori in higher Teichm\"uller theory}
\label{sec-quantumtrace}
\subsection{Quantum tori and monomial subalgebras}
Assume $m$ is a positive integer, and $\mathsf{P}$ is an $m\times m$ anti-symmetric matrix. For a ground ring $\cR$, the \textbf{quantum torus} $\mathbb{T}(\mathsf{P})$ associated with $\mathsf{P}$ is defined as 
$$\mathbb{T}(\mathsf{P}) := \cR\langle x_1^{\pm 1},x_2^{\pm 2},\cdots,x_m^{\pm 1} \rangle/(x_ix_j = \hat{q}^{2\mathsf{P}_{ij}}x_jx_i,1\leq i,j\leq m).$$
For $\mathbf{k} = (k_1,\cdots,k_m)\in\mathbb{Z}^m$, we define
\begin{align}
x^{\mathbf{k}} := [x_1^{k_1}x_2^{k_2}\cdots x_m^{k_m}]_{\rm Weyl} = 
\hat{q}^{-\sum_{1\leq i<j\leq m} k_ik_j\mathsf{P}_{ij}}x_1^{k_1}x_2^{k_2}\cdots x_m^{k_m},\label{eq:Weyl}
\end{align}
called the Weyl normalization. 
Then $\{x^{\mathbf{k}}\mid \mathbf{k}\in\mathbb{Z}^m\}$
is an $\cR$-basis of $\bT(\mathsf{P})$.
For $\mathbf{k}_1,\mathbf{k}_2\in\mathbb{Z}^m$, we have 
\begin{align}
x^{\mathbf{k}_1}x^{\mathbf{k}_2} := \hat{q}^{\langle \mathbf{k}_1,\mathbf{k}_2\rangle_{\mathsf{P}}}
x^{\mathbf{k}_1+\mathbf{k}_2},\quad 
x^{\mathbf{k}_1}x^{\mathbf{k}_2} = \hat{q}^{2\langle \mathbf{k}_1,\mathbf{k}_2\rangle_{\mathsf{P}}} 
x^{\mathbf{k}_2}x^{\mathbf{k}_1},\label{eq;q-comm}
\end{align}
with $\langle \mathbf{k}_1,\mathbf{k}_2\rangle_{\mathsf{P}} := \mathbf{k}_1 \mathsf{P} \mathbf{k}_2^{T}$, where we regard $\mathbf{k}_i$ as row vectors and $\mathbf{k}_2^{T}$ denotes the transpose of $\mathbf{k}_2$.

We define $\mathbb T_{+}(\mathsf P)$ to be the $\cR$-submodule of 
$\mathbb T(\mathsf P)$ generated by $x^{{\bf k}}$ for ${\bf k}\in \mathbb N^m$, i.e. 
$$\mathbb{T}_{+}(\mathsf{P}) := \cR\langle x_1,x_2,\cdots,x_m \rangle/(x_ix_j = \hat{q}^{2\mathsf{P}_{ij}}x_jx_i,1\leq i,j\leq m).$$

Let $\Lambda\subset\mathbb Z^m$ be a submonoid, i.e., the subset $\Lambda$ is closed under the addition. Then, 
$$\bT(\sfP;\Lambda):=\cR\text{-span}\{x^{\bf k}\mid {\bf k}\in\Lambda\}$$ is a subalgebra of $\mathbb T(\mathsf{P})$, called the \textbf{$\Lambda$-monomial subalgebra} of $\bT(\sfP)$.

Let $\mathcal{Z}(\mathbb{T}(\mathsf{P}))$ denote the center of $\mathbb{T}(\mathsf{P})$.
\begin{lem}\label{quantum}
    (a) For generic $\hat{q}$, we have 
    $\mathcal{Z}(\mathbb{T}(\mathsf{P})) = \cR\text{-span}\{x^{{\bf k}}\mid \langle {\bf k},{\bf k}'\rangle_{\mathsf{P}}=0,\forall {\bf k}'\in\mathbb{Z}^m  \}.$ 

    (b) If $\hat{q}^2$ is a root of unity of order $d$, we have 
    $\mathcal{Z}(\mathbb{T}(\mathsf{P})) = \cR\text{-span}\{x^{\bf k}\mid \langle {\bf k},{\bf k}'\rangle_{\mathsf{P}}=0\in \mathbb{Z}_d,\forall {\bf k}'\in\mathbb{Z}^m  \},$ 
    where $\mathbb{Z}_d:=\mathbb{Z}/d\mathbb{Z}$.
\end{lem}
\begin{proof}
    It is obvious that these special elements in (a) or (b) live in the center of $\mathbb{T}(\mathsf{P})$. Then it suffices to show $\mathcal{Z}(\mathbb{T}(\mathsf{P}))$ is actually the $\cR$-span of the elements in (a) or (b).

    Suppose $\sum_{\ttk\in\Lambda}c_{\ttk} x^{\ttk}\in \mathcal{Z}(\mathbb{T}(\mathsf{P}))$, where $\Lambda$ is a finite subset of $\bZ^m$ and $c_{\ttk}\neq 0$ for all $\ttk\in \Lambda$. For any $\ttk'\in\bZ^m$, from \eqref{eq;q-comm}, we have $$(\sum_{\ttk\in\Lambda}c_{\ttk} x^{\ttk})x^{\ttk'} = 
    x^{\ttk'}(\sum_{\ttk\in\Lambda}c_{\ttk} \hat{q}^{2\langle \mathbf{k},\mathbf{k}'\rangle_{\mathsf{P}} } x^{\ttk})$$
    Thus $(\sum_{\ttk\in\Lambda}c_{\ttk} x^{\ttk})x^{\ttk'} = x^{\ttk'}(\sum_{\ttk\in\Lambda}c_{\ttk} x^{\ttk})$ implies $\hat{q}^{2\langle \mathbf{k},\mathbf{k}'\rangle_{\mathsf{P}} } = 1$. This completes the proof.
\end{proof}

\begin{lem}\label{PI}
    From Lemma \ref{quantum},  
    $\mathcal Z(\mathbb T(\mathsf{P})) = \{x^{\mathbf{k}}\mid \mathbf{k}\in\Lambda\}$, where $\Lambda$ is a submonoid of $\mathbb Z^m$. Suppose $\mathbb Z^m/\Lambda=\{ \mathbf{k}+\Lambda\mid \mathbf{k}\in S\subset \mathbb Z^m\}$. 
    
    (a) $\{ x^\mathbf{k}\mid \mathbf{k}\in S\}$ is a free basis of $\mathbb T(\mathsf{P})$ over the center $\mathcal Z(\mathbb T(\mathsf{P}))$.

    (b) If $|S|=|\frac{\mathbb Z^m}{\Lambda}|$ is finite, then the rank of $\mathbb T(\mathsf{P})$ as a $\mathcal Z(\mathbb T(\mathsf{P}))$-module is $|\frac{\mathbb Z^m}{\Lambda}|$.
\end{lem}

\begin{proof}
    (a) Obviously, $\{ x^\mathbf{k}\mid \mathbf{k}\in S\}$ spans $\mathbb T(\mathsf{P})$ over  $\mathcal Z(\mathbb T(\mathsf{P}))$. Suppose 
    $\sum_{\mathbf{k}\in S_0}f_{\mathbf{k}} x^{\mathbf{k}} =0$,
    where $S_0$ is a finite subset of $S$ and $f_{\mathbf{k}}\in \mathcal Z(\mathbb T(\mathsf{P}))$ for every $\mathbf{k}\in S_0$.
    For each $\mathbf{k}\in S_0$, we suppose 
    $f_{\mathbf{k}} = \sum_{\mathbf{u}\in\Lambda_{\mathbf{k}}} l_{\mathbf{k},\mathbf{u}} x^{\mathbf{u}}$, where $\Lambda_{\mathbf{k}}$ is a finite subset of $\Lambda$ and $l_{\mathbf{k},\mathbf{u}} \in\cR$. Then 
    $$0=\sum_{\mathbf{k}\in S_0}f_{\mathbf{k}} x^{\mathbf{k}} =
    \sum_{\mathbf{k}\in S_0} \sum_{\mathbf{u}\in \Lambda_{\mathbf{k}}}l_{\mathbf{k},\mathbf{u}}  x^{\mathbf{u}+\mathbf{k}}.$$
    Since $\mathbf{u}+\mathbf{k}$ are distinct with each other, we have 
    $l_{\mathbf{k},\mathbf{u}} =0$ for $\mathbf{k}\in S_0$ and $\mathbf{u}\in \Lambda_{\mathbf{k}}$. 
    Thus $f_{\mathbf{k}} = 0$ for $\mathbf{k}\in S_0$.

    (b) The claim follows from (a).
\end{proof}

\subsection{Quantum (torus) frames}

Let $A$ be an $\cR$-domain and let $S \subset A$ be a subset of non-zero elements. Let $\mathsf{Pol}(S)$ be the $\cR$-subalgebra of $A$ generated by $S$ and $\mathsf{LPol}(S)$ be the set of all $a \in A$ such that there is an $S$-monomial $m$ satisfying $am \in \mathsf{Pol}(S)$. If $A = \mathsf{LPol}(S)$ we say $S$ \textbf{weakly generates} $A$.

\begin{dfn}[\cite{LY23}]
For an $\cR$-domain $A$, consider a finite set $S = \{a_1,\cdots,a_r\} \subset A$. We call $S$ a \textbf{quantum torus frame} for $A$ if $S$ satisfies the following conditions; 
\begin{enumerate}
    \item $S$ is $q$-commuting and each element of $S$ is non-zero, 
    \item $S$ weakly generates $A$, 
    \item the set $\{a_1^{n_1}\cdots a_r^{n_r}\mid  n_i \in \bN\}$ is $\cR$-linearly independent.
\end{enumerate}
\end{dfn}

\subsection{Weight lattice of the Lie algebra $\mathfrak{sl}_n(\bC)$}
Let $\mathbbm{J}=\{1,2,\dots, n\}$. 
For a subset $\mathbf{i}=\{i_1,i_2,\dots, i_k\}\subset \mathbbm{J}$, 
let $\bar{\mathbf{i}}=\{\bar{i}_1,\bar{i}_2,\dots, \bar{i}_k\}$ and 
$\bar{\mathbf{i}}^c=\mathbbm{J}\setminus \bar{\mathbf{i}}$. 

Let $\mathsf{L}$ be the weight lattice of the Lie algebra $\mathfrak{sl}_n(\bC)$, i.e., the free abelian group generated by $\mathsf{w}_1,\cdots, \mathsf{w}_n$, modulo $\mathsf{w}_1+\cdots+ \mathsf{w}_n=0$. 

We have a symmetric bilinear form defined by 
\begin{align}
\langle \mathsf{w}_i, \mathsf{w}_j\rangle:=\delta_{ij}-1/n. \label{eq;bracket}   
\end{align}

For an edge of $\partial \Sigma$ and a stated $n$-web diagram $\alpha$, define
\begin{align}
\mathbbm{d}_{e}(\alpha)=\sum_{x\in \partial\alpha\cap e}\mathsf{w}^{\ast}_{\overline{s(x)}}\in \mathsf{L},
\end{align}
where $s(x)$ is the state at $x$, and 
$\mathsf{w}_{\overline{s(x)}}^{\ast}=\begin{cases}
\mathsf{w}_{\overline{s(x)}}  & \text{if $\alpha$ points out of $\Sigma$ at $x$},\\
-\mathsf{w}_{s(x)} & \text{if $\alpha$ points into $\Sigma$ at $x$}.
\end{cases}$

\subsection{$n$-triangulation}

\begin{figure}[h]
    \centering
    \includegraphics[width=260pt]{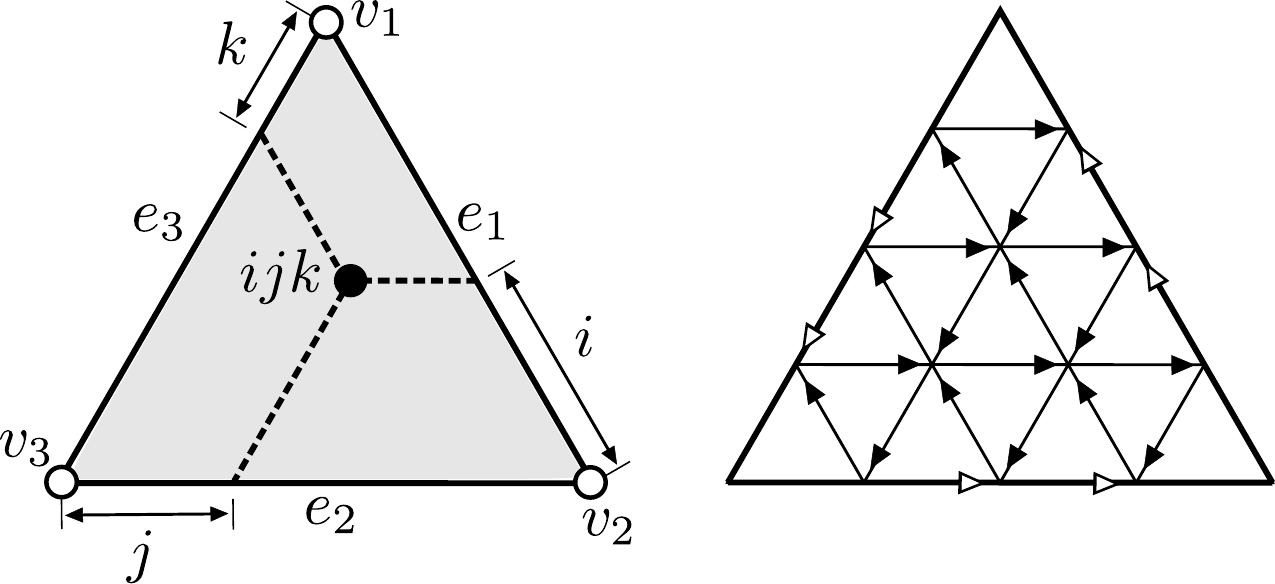}
    \caption{Barycentric coordinates $ijk$ and a $4$-triangulation with its quiver}\label{Fig;coord_ijk}
\end{figure}

Consider barycentric coordinates for $\bP_3$ so that
\begin{equation}
\bP_3=\{(i,j,k)\in\bR^3\mid i,j,k\geq 0,\ i+j+k=n\}\setminus\{(0,0,n),(0,n,0),(n,0,0)\}, 
\end{equation}
where $(i,j,k)$ (or $ijk$ for simplicity) are the barycentric coordinates. 
Let $v_1=(n,0,0)$, $v_2=(0,n,0)$, $v_3=(0,0,n)$. 
Let $e_i$ denote the edge on $\partial \bP_3$ whose endpoints are $v_i$ and $v_{i+1}$. 
See Figure \ref{Fig;coord_ijk}.

The \textbf{$n$-triangulation} of $\bP_3$ is the subdivision of $\bP_3$ into $n^2$ small triangles using lines $i,j,k=\text{constant integers}$. 
For the $n$-triangulation of $\bP_3$, the vertices and edges of all small triangles, except for $v_1,v_2,v_3$ and the small edges adjacent to them, form a quiver $\Gamma_{\bP_3}$.
An \textbf{arrow} is the direction of a small edge defined as follows. If a small edge $e$ is in the boundary $\partial\bP_3$ then $e$ has the counterclockwise direction of $\partial \bP_3$. If $e$ is in the interior then its direction is the same with that of a boundary edge of $\bP_3$ parallel to $e$. Assign weight $1$ to any boundary arrow and weight $2$ to any interior arrow.

The set $\barV=\barV_{\bP_3}$ 
of points with integer barycentric coordinates of $\bP_3$:
\begin{align}
\barV=\barV_{\bP_3} = \{ijk \in \bP_3 \mid i, j, k \in \bZ\}.
\end{align}
Each element of $\barV$ is called a \textbf{small vertex}. 

\subsection{The (extended) Fock--Goncharov algebra and balanced part}\label{subsec:FGalg}
\def\barX{\overline{\cX}}

The triangle $\bP_3$ has a unique triangulation consisting of the 3 boundary edges up to isotopy, and we will also use ${\bP_3}$ to denote the triangulation. 

By cutting $\Sigma$ along all edges not isotopic to a boundary edge, we have a disjoint union of ideal triangles. Each triangle is called a \textbf{face} of $\lambda$. Let $\cF_\lambda$ denote the set of faces. Then
\begin{equation}\label{eq.glue}
\Sigma = \Big( \bigsqcup_{\tau\in\cF_\lambda} \tau \Big) /\sim,
\end{equation}
where each face $\tau$ is a copy of $\bP_3$, and $\sim$ is the identification of edges of the faces to recover $\lambda$. 
Each face $\tau$ is characterized by a \textbf{characteristic map} $f_\tau\colon \tau \to \Sigma$, which is a homeomorphism when we restrict $f_\tau$ to $\Int\tau$ or the interior of each edge of $\tau$.

An \textbf{$n$-triangulation} of $\lambda$ is a collection of $n$-triangulations of the faces $\tau$ which are compatible with the gluing $\sim$,  where compatibility means, for any edges $b$ and $b'$ glued via $\sim$, the vertices on $b$ and $b'$ are identified. Consider the \textbf{reduced vertex set}
$$\barV_\lambda=\bigcup_{\tau\in\cF_\lambda}\barV_\tau, \quad \barV_\tau=f_\tau(\barV).$$
Each element of $\barV_\lambda$ is also called a \textbf{small vertex}. 
The images of the weighted quivers $\Gamma_\tau$ by $f_\tau$ form a quiver $\Gamma_\lambda$ on $\Sigma$.
Note that when edges $b$ and $b'$ are glued, a small edge on $b$ and the corresponding small edge of $b'$ have opposite directions, i.e., the resulting arrows are of weight $0$.

Let $\barQ_\lambda\colon \barV_\lambda\times \barV_\lambda \to \bZ$ be the signed adjacency matrix of the weighted quiver $\Gamma_\lambda$ defined by 
\begin{equation}
\barQ_\lambda(v,v') = \begin{cases} w \quad & \text{if there is an arrow from $v$ to $v'$ of weight $w$},\\
0 &\text{if there is no arrow between $v$ and $v'$}, 
\end{cases}
\end{equation}
especially we use $\barQ_{\bP_3}$ to denote $\barQ_\lambda$ when $\Sigma=\bP_3$. 

The ($n$-th root version) \textbf{Fock-Goncharov algebra} is the quantum torus of $\barQ_\lambda$:
\begin{equation}
\barX_{\hat{q}}(\Sigma,\lambda)= \bT(\barQ_\lambda) = \cR \langle x_v^{\pm 1}, v \in \barV_\lambda \rangle / (x_v x_{v'}= \hat{q}^{\, 2 \barQ_\lambda(v,v')} x_{v'}x_v \text{ for } v,v'\in \barV_\lambda ).
\end{equation}

There is another interpretation of $\barX_{\hat{q}}(\Sigma,\lambda)$ as follows. For the tensor product algebra
\begin{equation}\label{eq.Xlambda}
\widetilde{\cX}_{\hat{q}}(\Sigma,\lambda):= \bigotimes_{\tau\in \cF_\lambda} \barX_{\hat{q}}(\tau)
= \bigotimes_{\tau\in \cF_\lambda} \bT(\barQ_\tau)
= \bT\Big( \bigoplus_{\tau\in \cF_\lambda} \barQ_\tau \Big), 
\end{equation}
where $\barX_{\hat{q}}(\Sigma,\lambda)$ can be seen as the $\cR$-submodule of $\widetilde{\cX}_{\hat{q}}(\Sigma,\lambda)$ spanned by $x^{\bk}$ with $\bk({v'}) = \bk({v''})$ for any $v'$ and $v''$ glued in the identification \eqref{eq.glue}.

Consider a matrix $M_{\bP_3}\colon \barV\times\barV\to\bZ$ associated to ${\bP_3}$, and the induced map $f_\tau\colon \barV\to\barV_\lambda$ by the characteristic map. Define the \textbf{zero-extension} 
$M_\tau\colon\barV_\lambda\times\barV_\lambda\to\bZ$ of $M_{\bP_3}$ by
\begin{equation}
M_\tau(u,v)=\sum_{u'\in f_\tau^{-1}(u)}\sum_{v'\in f_\tau^{-1}(v)} M_{\bP_3}(u',v').
\end{equation}

When we see $\Gamma_\lambda$ as a union of copies of $\Gamma_{\bP_3}$, $\barQ_\lambda$ can be regarded as
\begin{equation}\label{eq-Q-def}
\barQ_\lambda = \sum_{\tau\in\cF_\lambda}\barQ_\tau,
\end{equation}
where $\barQ_\tau$ is the zero-extension of $\barQ_{\bP_3}$.

We attach an triangle $\bP_3$ to each boundary edge of $\Sigma$, and let $\Sigma^\ast$ be the resulting surface. See Figure \ref{Fig;attaching}.
Suppose the attaching edge is $e_1$ of $\bP_3$. 
For an ideal triangulation $\lambda$ of $\Sigma$, let $\lambda^\ast$ be the ideal triangulation of $\Sigma^\ast$ whose restriction to $\Sigma$ is $\lambda$. 
Let $\barV_{\lambda^\ast}$ denote the reduced vertex set of the extended $n$-triangulation. 
The $X$-vertex set $V_\lambda\subset\barV_{\lambda^\ast}$ is the subset of all small vertices not on $e_3$ in the attached triangles. 
The $A$-vertex set  $V'_\lambda\subset\barV_{\lambda^\ast}$ is the subset of all small vertices not on $e_2$ in the attached triangles. 

\begin{lem}[{\cite[Lemma 11.2]{LY23}}]\label{lem:cardinarity}
Let $\Sigma$ be a triangulable pb surface and $\lambda$ be a triangulation of $\Sigma$. Let $\binom{n}{2}:=\frac{n(n-1)}{2}$ and $r(\Sigma)=\# \partial\Sigma- \chi(\Sigma)$, where $\chi(\Sigma)$ is the Euler characteristic.
Then we have 
\begin{align}
(a)\ |V_\lambda|=(n^2-1)r(\Sigma),\qquad (b)\ |\barV_\lambda|=(n^2-1)r(\Sigma)-\binom{n}{2}(\#\partial\Sigma). \label{eq:cardinarity}
\end{align}
\end{lem}
\begin{figure}[h]
    \centering
    \includegraphics[width=90pt]{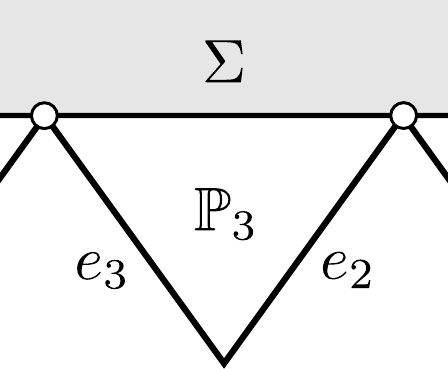}
    \caption{Attaching triangles.}\label{Fig;attaching}
\end{figure}

Let $\sfQ_\lambda\colon V_\lambda \times V_\lambda \to \bZ$ be the restriction of $ \barQ_{\lambda^\ast}\colon \barV_{\lambda^\ast} \times \barV_{\lambda^\ast} \to \bZ$. The \textbf{extended $X$-algebra} is defined as
$$\cX_{\hat{q}}(\Sigma,\lambda) = \bT(\sfQ_\lambda).$$
There is a natural identification of subalgebras $\barX_{\hat{q}}(\Sigma,\lambda)\subset\cX_{\hat{q}}(\Sigma,\lambda)\subset\barX_{\hat{q}}(\Sigma^\ast,\lambda^\ast)$.

Let $\proj_i\colon \barV \to\bZ\ (i=1,2,3)$ be the functions defined by
\begin{equation}
\proj_1(ijk)=i,\quad \proj_2(ijk)=j,\quad \proj_3(ijk)=k.\label{def:proj}
\end{equation}
Note that in \cite{LY23} they use $\bk_i$ instead of $\proj_i$. 
Let $\overline{\Lambda}_{\bP_3}\subset\bZ^{\barV}$ be the subgroup generated by $\proj_1,\proj_2,\proj_3$ and $(n\bZ)^{\barV}$. Elements in $\overline{\Lambda}_{\bP_3}$ are called \textbf{balanced}.

A vector $\bk\in\bZ^{\barV_\lambda}$ is \textbf{balanced} if its pullback $f_\tau^\ast\bk$ to ${\bP_3}$ is balanced for every face of $\lambda$, where for every face $\tau$ and its characteristic map  $f_\tau\colon\tau\to\Sigma$, the pullback $f_\tau^\ast\bk$ is a vector $\barV\to\bZ$ given by $f_\tau^\ast\bk(v)=\bk(f_\tau(v))$. 
Let $\overline{\Lambda}_\lambda$ denote the subgroup of $\bZ^{\barV_\lambda}$ generated by all the balanced vectors.

The \textbf{balanced Fock-Goncharov algebra} is the monomial subalgebra
$$\barX^{\rm bl}_{\hat{q}}(\Sigma,\lambda)=\bT(\barQ_\lambda;\overline{\Lambda}_\lambda).$$
The extended version is defined as 
$$\cX^{\rm bl}_{\hat{q}}(\Sigma,\lambda)=\bT(\sfQ_\lambda)\cap \barX_{\hat{q}}^{\rm bl}(\Sigma^\ast,\lambda^\ast)=\bT(\sfQ_\lambda;\Lambda_\lambda),$$
where $\bT(\sfQ_\lambda)$ is considered a subalgebra of $\bT(\barQ_{\lambda^{\ast}})$ by the natural embedding, and $\Lambda_\lambda=\overline{\Lambda}_{\lambda^\ast}\cap\bZ^{V_\lambda}$ is the subgroup of balanced vectors.

\subsection{The $A$-version quantum tori} \label{sec;A_tori}
In this subsection, assume $\Sigma$ does not have interior punctures.

For a small vertex $v\in \barV_\lambda$ and an ideal triangle $\nu \in \cF_\lambda$, we define its \textbf{skeleton} $\skeleton_\tau(v)\in \bZ [\barV_\tau]$ as follows. 

For a face $\nu\in\cF_\lambda$ containing $v$, suppose $v=(ijk)\in V_\nu$. Draw a weighted directed graph $Y_v$ properly embedded into $\nu$ as in the left of Figure~\ref{Fig;skeleton}, where an edge of $Y_v$ has weight $i$, $j$ or $k$ according as the endpoint lying on the edge $e_1$, $e_2$ or $e_3$ respectively. 
\begin{figure}[h]
    \centering
    \includegraphics[width=380pt]{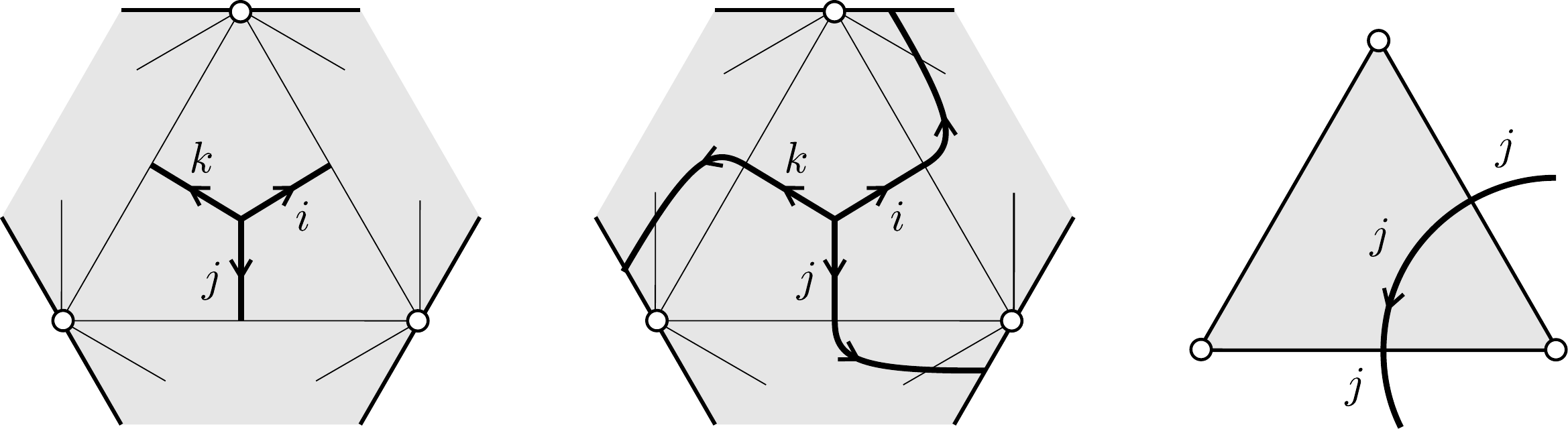}
    \caption{Left: Weighted graph $Y_v$,\quad Middle: Elongation $\widetilde{Y}_v$,\quad Right: Turning left}\label{Fig;skeleton}
\end{figure}

Elongate the nonzero-weighted edges of $Y_v$ to have an embedded weighted directed graph $\widetilde{Y}_v$ as drawn in the middle of Figure~\ref{Fig;skeleton}. Here, each edge is elongated by turning left whenever it enters a triangle. 
The part of the elongated edge in a triangle $\tau$ is called a \textbf{(arc) segment} of $\widetilde{Y}_v$ in $\tau$. In addition, we also regard $Y_v$ as a segment of $\widetilde{Y}_v$, called the \textbf{main segment}.

For the main segment $s=Y_v$, define $Y(s) = v \in \barV_\nu$. For an arc segment $s$ in a triangle $\tau$, define $Y(s)\in \barV_\tau$ to be the small vertex of the following weighted graph. 
$$
s=\begin{array}{c}\includegraphics[scale=0.27]{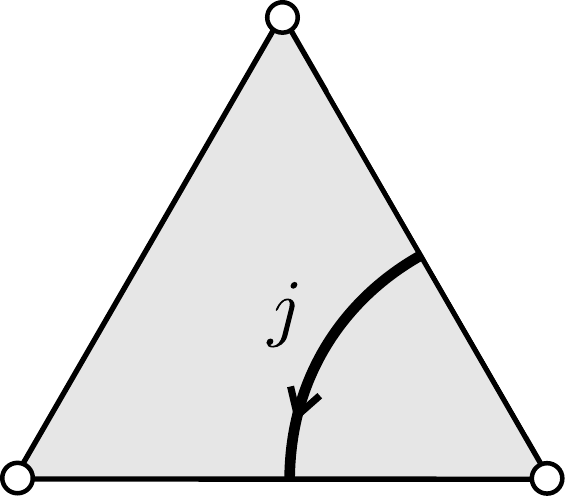}\end{array}
\longrightarrow\quad Y(s):=\begin{array}{c}\includegraphics[scale=0.27]{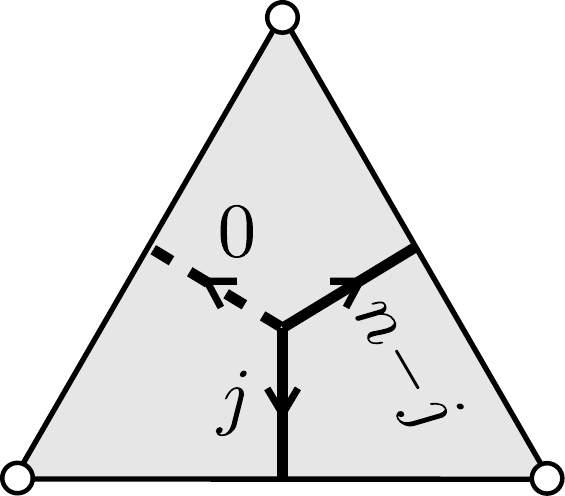}\end{array}
$$
Define $\skeleton_\tau(v)$ by
\begin{equation}
\skeleton_\tau(v) = \sum_{s \subset \tau\cap\tilde{Y}_v} Y(s) \in \bZ [\barVt],
\end{equation}
where the sum is over all the segments of $\widetilde{Y}_v$ in $\tau$.
It is known that $\skeleton_\tau(v)$ is well-defined \cite[Lemma 11.4]{LY23}. 

Define a $\bZ_3$-invariant map
$$\barK_{\bP_3}\colon \barV\times\barV\to\bZ$$
such that if $v=ijk$ and $v'=i'j'k'$ satisfy $i'\leq i$ and $j'\geq j$ then 
\begin{align}
\barK_{\bP_3}(v,v')=jk'+ki'+i'j.
\end{align}
Under identifying $\bP_3$ and a face $\tau$ of $\lambda$, we also use $\barK_{\tau}$ as $\barK_{\bP_3}$. 

Recall that $\cF_\lambda$ denotes the set of all the faces of the triangulation $\lambda$. 
For $u, v \in \barVl$ and a face $\tau \in \cF_\lambda$ containing $v$,  define 
\begin{equation}\label{eq-surgen-exp}
\barK_{\lambda}(u,v)=\barKt(\skeleton_\tau(u),v)=\sum_{s\subset\tau\cap\tilde{Y}_u}\barKt(Y(s),v).
\end{equation}
It is known that $\barK_{\lambda}$ is well-defined \cite[Lemma 11.5]{LY23}.

Define a map $p\colon \barVlast\setminus\barVl\to\barVlast\setminus V'_\lambda$ as follows. 
Every $v \in\barVlast\setminus\barVl$ has coordinates $ijk$ in an attached triangle with $k\ne0$, and $\barVlast\setminus V'_{\lambda}$ consists of vertices $ijk$ in attached triangles with $i=0$. Then
\begin{equation}\label{eq-cov-pdef}
p(v)=(0,n-k,k)\quad \text{in the same triangle}.
\end{equation}
The change-of-variable matrix $\sfC\colon V'_\lambda\times\barVlast\to\bZ$ is given by
\begin{align*}
\sfC(v,v)&=1,&&v\in V'_\lambda,\nonumber\\
\sfC(v,p(v))&=-1,&&v\in V'_\lambda\setminus\barVl,\\
\sfC(v,v')&=0,&&\text{otherwise}.\nonumber
\end{align*}

For the extended ideal triangulation $\lambda^\ast$ of $\lambda$, consider $\barK_{\lambda^\ast}\colon \barVlast\times\barVlast\to\bZ$. 
Note that the product $\sfC\barK_{\lambda^\ast}$ is a bilinear form on $V'_\lambda\times\barVlast$. 
The matrix $\sfK_\lambda$ is the restriction of $\sfC\barK_{\lambda^\ast}$:
\begin{equation}\label{eq-Klambda}
\sfK_\lambda=(\sfC\barK_{\lambda^\ast})|_{V'_\lambda\times V_\lambda}.
\end{equation}

From \cite[Lemma~11.9]{LY23}, we define the anti-symmetric integer matrices $\barP_{\lambda}$ and $\sfP_\lambda$ by 
\begin{align}\label{eq-anti-matric-P-def}
\barP_\lambda:=\barK_\lambda\barQ_\lambda\barK^t_\lambda,\qquad 
\sfP_\lambda:=\sfK_\lambda\sfQ_\lambda\sfK^t_\lambda
\end{align}

The following are \textbf{$A$-version quantum tori and quantum spaces} of $(\Sigma,\lambda)$:
\begin{align*}
\barA_{\hat{q}}(\Sigma,\lambda) =\bT(\bar{\sfP}_\lambda),\qquad
\overline{\cA^+_{\hat{q}}}(\Sigma,\lambda)=\bT_+(\bar{\sfP}_\lambda),\\
\cA_{\hat{q}}(\Sigma,\lambda)=\bT(\sfP_\lambda),\qquad
\cA_{\hat{q}}^{+}(\Sigma,\lambda)=\bT_+(\sfP_\lambda),
\end{align*}
where $\bT_+$ is defined in Section~\ref{sec-quantumtrace}. 
In the following, let $a_v$ denote the generator of $\barA_{\hat{q}}(\Sigma,\lambda)$ (resp. $\cA_{\hat{q}}(\Sigma,\lambda)$) corresponding to $v\in\barV_\lambda$ (resp. $V'_\lambda$), and 
let $a^{\bf k}$ denote the Weyl normalized monomial of $\prod_{v}a_{v}^{{\bf k}(v)}$ for ${\bf k}\in \bZ^{\barV_\lambda}$ (resp. ${\bf k}\in\bZ^{V'_\lambda}$).

\begin{thm}[{\cite[Theorem 11.7]{LY23}}]\label{thm-transition-LY}
The $\cR$-linear maps 
\begin{align}
{\psi_\lambda}&\colon  \cAl \to \cXl, \quad  a^\mathbf{k}\mapsto x^{\mathbf{k} \sfK_{\lambda}}, \ ({\mathbf{k}} \in \bZ^{V'_{\lambda}})\\
\overline \psi_\lambda &\colon \barA_{\hat{q}}(\Sigma,\lambda)\to
\overline{\mathcal X}_{\hat q}(\Sigma,\lambda),
\quad  a^\mathbf{k}\mapsto x^{\mathbf{k} \overline{\mathsf{K}}_{\lambda}}, \ ({\mathbf{k}} \in \bZ^{\overline V_{\lambda}})
\end{align}
are $\cR$-algebra embeddings with $\im \psi_\lambda=\Xbll$ and 
$\im \overline\psi_\lambda=\barX^{\rm bl}_{\hat{q}}(\Sigma,\lambda)$. 
\end{thm}

\subsection{$\SL(n)$-quantum trace map}\label{sec:quantum_trace}

For a pb surface $\Sigma$ and $\cR = \bZ[\hat{q}^{\pm1}]$, 
it is known that there is a unique reflection $\omega \colon \cS_n (\Sigma,\mathbbm{v})\to \cS_n (\Sigma,\mathbbm{v})$ (i.e., $\omega(\hat q) = \hat q^{-1}$, $\omega$ is $\mathbb Z$-linear, $\omega^2$ is the identity, and $\omega(xy)= \omega(y) \omega(x)$ for any $x,y\in \cS_n (\Sigma,\mathbbm{v})$) such that, for a stated $n$-web diagram $\alpha$, $\omega(\alpha)$ is defined from $\alpha$ by switching all the crossings and reversing the height order on each boundary edge \cite[Theorem 4.6]{LS21}. 

A stated $n$-web diagram $\alpha$ in a pb surface $\Sigma$ is \textbf{reflection-normalizable} if $\omega(\alpha) = \hat{q}^{2k}\alpha$ for $k \in \bZ$. Note that such $k$ is unique. Over any ground ring $\cR$ with a distinguished invertible element $\hat q$, 
we define the \textbf{reflection-normalization} of $\alpha$ by
$$[\alpha]_{\rm norm} := \hat{q}^{k}\alpha.$$
Then, over $\cR = \bZ[\hat{q}^{\pm1}]$, we have $\omega([\alpha]_{\rm norm}) = [\alpha]_{\rm norm}$, i.e., $[\alpha]_{\rm norm}$ is reflection invariant. 
Note the Weyl normalization of a monomial in a quantum torus corresponds with the reflection normalization.

Let $\binom{\bJ}{k}$ denote the set of all $k$-element subsets of $\bJ =\{1,\dots, n\}$. For $I\subset\bJ$, define
$$\bar{I}=\{\bar{i}\mid i\in I\},\quad
I^c=\bJ\setminus I,\quad \bar{I}^c=(\bar{I})^c.$$
For $I,J\in \binom{\bJ}{k}$, let $M^I_J(\mathbf{u})\in \cO_q(\SL(n))$ denote the quantum determinant of the $I\times J$-submatrix of $\mathbf{u}$. 
We identify $\cS_n(\fB,\mathbbm{v})= \Oq$ and also identify $u_{ij}\in \Oq$ with the stated arc given in Section~\ref{sec:bigon}. 
Let $a$ be an oriented $n$-web diagram in a pb surface $\Sigma$, and $N(a)$ be a small open tubular neighborhood of $a$ in $\Sigma$. By an identification of $\fB$ and $N(a)$ which induces $\cS_n(\fB,\mathbbm{v})\to \Oq$, let $M^I_J(a)$ be the image of $M^I_J(\mathbf{u})$ and depict it as
\begin{align}
\begin{array}{c}\includegraphics[scale=0.43]{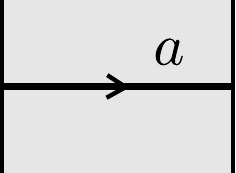}\end{array}
\longrightarrow I\!\begin{array}{c}\includegraphics[scale=0.43]{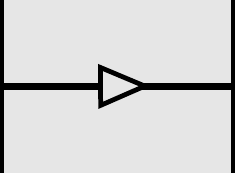}\end{array}\! J:=M^I_J(a). \label{eq:quantum_minor}
\end{align}

Lemma 4.13 in \cite{LY23} claims that the stated $n$-web diagram 
\begin{align*}
\alpha=\begin{array}{c}\includegraphics[scale=0.27]{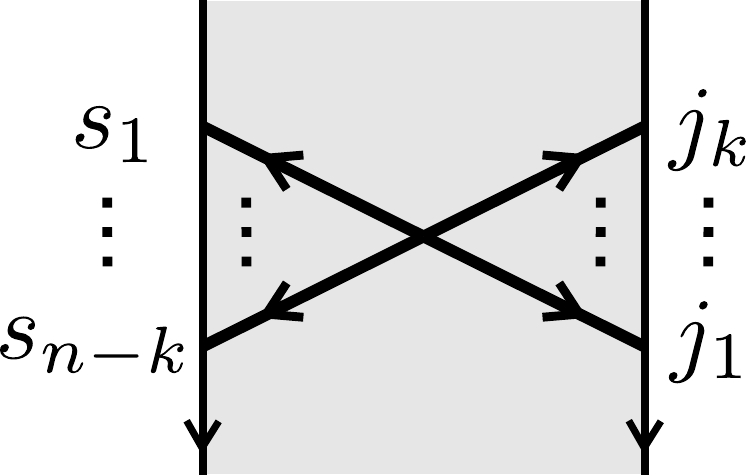}\end{array}. 
\end{align*}
is reflection-normalizable, and 
$M^{I}_{J}({\bf u})$ is equal to $\alpha$ up to a sign and a power of $q^{1/2n}$, where $J=\{j_1,j_2,\dots,j_k\},\bar{I}^c=\{s_1,s_2,\dots, s_{n-k}\}$.

For $v=(ijk) \in  \barV_\nu\subset \barV_\lambda$ with a triangle $\nu$ of $\lambda$,  consider the graph $\widetilde{Y}_v$ defined in Section~\ref{sec;A_tori}. 
By replacing a $k$-labeled edge of $\widetilde{Y}_v$ with $k$-parallel edges, we obtain a stated $n$-web $g''_v$, adjusted by a sign: 
$$\widetilde{Y}_v=
\begin{array}{c}\includegraphics[scale=0.40]{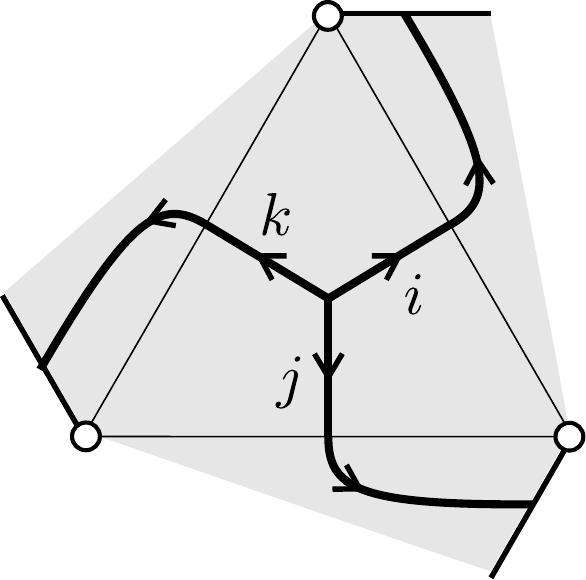}\end{array}\longrightarrow g''_v:=(-1)^{\binom{n}{2}}\begin{array}{c}\includegraphics[scale=0.50]{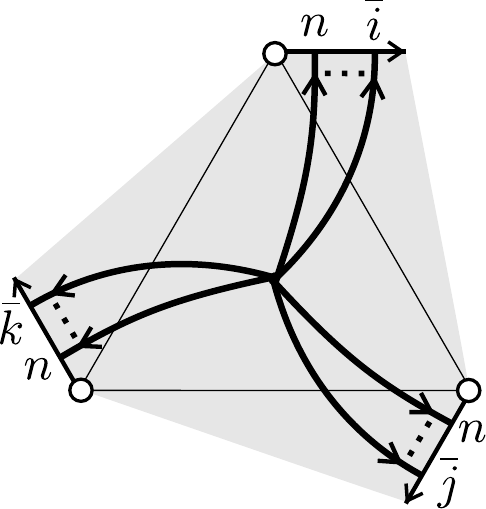}\end{array}. $$
I have changed the picture. We should confirm the change does not affect the proof.
It is known that $g''_v$ is reflection-normalizable \cite[Lemma 4.12]{LY23}. 

For an attached triangle $\nu=\bP_3$ and $v=(ijk)= \barV_{\nu}\subset V'_\lambda\setminus \barV_\lambda$, 
let $c$ denote the oriented corner arc of $\Sigma$ 
starting on $e$ and turning left all the time:  
$$
\begin{array}{c}\includegraphics[scale=0.35]{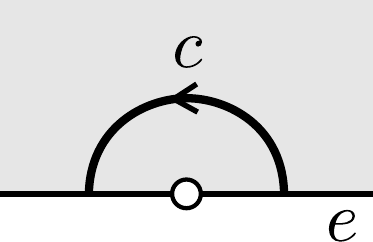}\end{array}\longrightarrow g''_v:=M^{[j+1,j+i]}_{[\bar{i},n]}(c)\begin{array}{c}\includegraphics[scale=0.47]{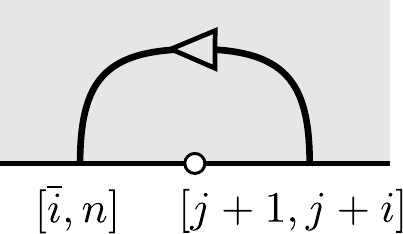}\end{array}. $$
Then it is known that the obtained element 
$g''_v$ is reflection-normalizable \cite[Lemmas 4.13 and 4.10]{LY23}

Define $\gaa_v$ as the reflection invariant of $\gaa_v''$. 
Let $\bar\gaa_v$ be the image of $\gaa_v$ in $\rSn$ by the projection $\Sn\to \rSn$. Note that $\bar\gaa_v = 0$ if $v \in V'_\lambda\setminus\barV_\lambda$.

The following theorem plays an important role in this paper, and we often call the properties \eqref{eq-sandwitch-property} and \eqref{eq-sandwitch-property_reduced} the sandwiched property.

\begin{thm} [{\cite[Theorems 13.1 and 15.5]{LY23}}]\label{traceA}
Let $\Sigma$ be a triangulable essentially bordered pb surface without interior punctures, and let $\lambda$ be an ideal triangulation of $\Sigma$. 

(a) There exists an algebra embedding $tr_{\lambda}^A:\cS_n(\Sigma,\mathbbm{v})\rightarrow \mathcal A_{\hat{q}}(\Sigma,\lambda)$ such that $\tra(\gaa_v) = a_v$ for all $v\in V_{\lambda}'$. Furthermore, we have 
\begin{align}\label{eq-sandwitch-property}
\cA^{+}_{\hat{q}}(\Sigma,\lambda)\subset \tra(\Sn)\subset \cA_{\hat{q}}(\Sigma,\lambda).
\end{align}

(b) There exists an algebra homomorphism $\overline{tr}_{\lambda}^A:\rSn\rightarrow \rA$ such that $\overline{tr}_{\lambda}^A(\bar{\gaa}_v) = a_v$ for all $v\in \overline{V}_{\lambda}$. Furthermore, we have
\begin{align}\label{eq-sandwitch-property_reduced}
    \rAp\subset \overline{tr}_{\lambda}^A(\rSn)\subset \rA.
\end{align}
If $n=2,3$, or $n>3$ and $\Sigma$ is a polygon, then $\overline{tr}_{\lambda}^A$ is injective.
\end{thm}

\begin{cor}[{\cite[Corollary 13.4]{LY23}}]
    Let $\Sigma$ be an essentially bordered pb surface without interior punctures. 

    (a) The set $\{\gaa_v\mid v\in V_{\lambda}' \}$ is a quantum torus frame for $\cS_n(\Sigma,\mathbbm{v})$.

    (b) If $\Sigma$ is a polygon, the set $\{\bar{\gaa}_v\mid v\in \overline{V}_{\lambda}\}$ is a quantum torus frame for $\overline{\cS}_n(\Sigma,\mathbbm{v})$.
\end{cor}

\section{The Frobenius map for (reduced) stated $\SL(n)$-skein algebras}
Suppose $q^{2}$ is a primitive $m$-th root of unity. Set  $\mathbbm{u} = \mathbbm{v}^{m^2}$,  $\varepsilon = q^{m^2}$, and
\begin{align*}
\mathbbm{c}_{i,\mathbbm{u}}= (-\varepsilon)^{n-i} \varepsilon^{\frac{n-1}{2n}},\quad
\mathbbm{t}_\mathbbm{u}= (-1)^{n-1} \varepsilon^{\frac{n^2-1}{n}},\quad \mathbbm{a}_\mathbbm{u}=   \varepsilon^{\frac{n+1-2n^2}{4}}.
\end{align*}
In this section, we will use the relations calculated in \cite{Wan23} to 
 construct an injective algebra homomorphism 
$\cF\colon \cS_n(\Sigma,\mathbbm{u})\rightarrow \cS_n(\Sigma,\mathbbm{v})$, called the Frobenius map, when $\Si$ is an essentailly bordered pb surface.
This generalizes the $\SL(n)$-Frobenius map in \cite{Wan23} to all roots of unity. See \cite{KLW} for a different construction using the ideal triangulation.

Let $\alpha$ to a stated framed oriented arc in $\Si\times (-1,1)$, and let $k$ be a non-negative integer. We use $\alpha^{(k)}$ to denote the stated $n$-web obtained from $\alpha$ by taking $k$ parallel copies of $\alpha$ in the framing direction. We require that $\alpha^{(k)}$ lies in a small open tubular neighborhood of $\alpha$. There is a convention that $\alpha^{(0)} = \emptyset$.  
Suppose $\alpha$ is a stated $n$-web in $\Si\times(-1,1)$ such that $\alpha =\cup_{1\leq i\leq t}\alpha_i$, where $\alpha_i$ is a stated framed oriented arc in $\Si\times (-1,1)$ for each $1\leq i\leq t$. For any non-negative integer $k$, we define $\alpha^{(k)}$ to be $\cup_{1\leq i\leq t}\alpha_i^{(k)}$.

\subsection{The Frobenius map for the bigon}
Theorem \ref{Hopf} implies $\cS_n(\fB,\mathbbm{v})\simeq \cO_{q}(\SL(n))$ and $\cS_n(\fB,\mathbbm{u})\simeq \cO_{\varepsilon}(\SL(n))$.
We have $\varepsilon = \mathbbm{u}^{2n} = \mathbbm{v}^{2nm^2} = q^{m^2} = \pm 1$.

\begin{lem}[Subsection 7.2 in \cite{PW91}]\label{Oq}
There exists a Hopf algebra homomorphism
\begin{align*}
F_n \colon \Ou \rightarrow \Ov,\quad 
u_{ij}\mapsto (u_{ij})^{m}.
\end{align*}
\end{lem}

Lemma \ref{Oq} implies an important property that, for any $1\leq i,j\leq n$, we have 
\begin{equation}\label{eq_co}
\Delta(u_{ij}^m) = \sum_{1\leq k\leq n}u_{ik}^m\ot u_{kj}^m\in \Ov\ot_{R}\Ov
\end{equation}

\begin{lem}[Lemma 7.36 in \cite{Wan23}]\label{Inj}
$F_n$ in Lemma \ref{Oq} is injective.
\end{lem}

The Hopf algebra isomorphism $g_{big}$ in Theorem \ref{Hopf} and the algebraic embedding
$F_n$ in Lemma \ref{Oq} induce an algebraic embedding $\cF\colon\cS_n(\fB,\mathbbm{u})\rightarrow \cS_n(\fB,\mathbbm{v})$ such that the following diagram commutes. 
\begin{equation}\label{diag}
\begin{tikzcd}
\Ou  \arrow[r, "F_n"]
\arrow[d, "g_{big}"]  
&  \Ov  \arrow[d, "g_{big}"] \\
 \Bu  \arrow[r, "\cF"] 
&  \Bv\\
\end{tikzcd}
\end{equation}

The Frobenius map constructed in \cite{Wan23} implies the following. For reader's convenience, we provide a brief proof.

\begin{lem}\label{lem_bi}
The Hopf algebra embedding $\cF$ in the commutative diagram (\ref{diag}) maps $a_{ij}$ (resp. $\cev{a}_{ij}$) to $a_{ij}^{(m)}$ (resp. $\cev{a}_{ij}^{(m)}$), where $i,j\in\{1,2,\cdots,n\}$.
\end{lem}
\begin{proof}
Obviously $\cF(a_{ij}) = a_{ij}^m = a_{ij}^{(m)}$.

Since $S(a_{ij}) = (-q)^{i-j}\cev{a}_{\bar{j},\bar{i}}\in \Bv$ and $g_{big}$ is a Hopf algebra isomorphism, we have 
$g_{big}(S(u_{ij})) = (-q)^{i-j}\cev{a}_{\bar{j},\bar{i}}\in \Bv$. Similarly, we have $g_{big}(S(u_{ij})) = (-\varepsilon)^{i-j}\cev{a}_{\bar{j},\bar{i}}\in \Bu$.
We have $F_n(S(u_{ij})) = S(u_{ij}^m) = S(u_{ij})^{m}$ because $F_n$ is a Hopf algebra embedding. 
Then $$(-\varepsilon)^{i-j}\cF(\cev{a}_{\bar{j},\bar{i}})
= \cF(g_{big}(S(u_{ij}))) = g_{big}(F_n(S(u_{ij})))
= g_{big}(S(u_{ij}))^m = (-q)^{m(i-j)}\cev{a}_{\bar{j},\bar{i}}^m.$$
Concrete computation shows $(-q)^m =-\varepsilon$. Thus we have 
$\cF(\cev{a}_{\bar{j},\bar{i}}) = \cev{a}_{\bar{j},\bar{i}}^m = \cev{a}_{\bar{j},\bar{i}}^{(m)}$.
\end{proof}

\subsection{The Frobenius map for general essentially bordered pb surfaces}

In this subsection, we will build the $\SL(n)$-Frobenius map for essentially bordered pb surfaces using the saturated system. 
In the following we assume all the essentially bordered pb surfaces are connected.

For an essentially bordered pb surface $\Si$ and a saturated system $B=\{b_1,\cdots,b_r\}$ of $\Si$, let $\iota^B$ denote the embedding from 
$U(B)$ to $\Sigma$.
Theorem \ref{key} implies there are two $\cR$-linear isomorphisms 
\begin{equation}\label{eq_iso}
\iota^B_{*} \colon\cS_n(U(B),\mathbbm{u})\rightarrow \cS_n(\Sigma,\mathbbm{u}),\quad \iota^B_{*} \colon\cS_n(U(B),\mathbbm{v})\rightarrow \cS_n(\Sigma,\mathbbm{v}).
\end{equation}
Since $U(B)$ is the disjoint union of bigons (see Section \ref{sec;sat_sys}) and we already built the Frobenius map $\cF$ for the bigon, we have an algebraic embedding 
\begin{equation}\label{eq_bigon}
\cF^{\ot r}\colon\cS_n(U(B),\mathbbm{u})\rightarrow \cS_n(U(B),\mathbbm{v}).
\end{equation}
The two $\cR$-linear isomorphisms in \eqref{eq_iso} and the algebraic embedding in \eqref{eq_bigon} induce an $\cR$-linear embedding $\cF_B\colon\cS_n(\Sigma,\mathbbm{u})\rightarrow \cS_n(\Sigma,\mathbbm{v})$ such that the following diagram commutes.
\begin{equation}
\begin{tikzcd}
\cS_n(U(B),\mathbbm{u})  \arrow[r, "\cF^{\ot r}"]
\arrow[d, "\iota^B_{*}"]  
&  \cS_n(U(B),\mathbbm{v})  \arrow[d, "\iota^B_{*}"] \\
 \cS_n(\Si,\mathbbm{u})  \arrow[r, "\cF_B"] 
&  \cS_n(\Si,\mathbbm{v})\\
\end{tikzcd}
\end{equation}

\begin{prop}\label{Fro}
Let $\Sigma$ be an essentially bordered pb surface and $B$ be a saturated system of $\Sigma$.

(a) For any stated $n$-web $\alpha$ in $\Si\times(-1,1)$ consisting of stated framed oriented arcs, we have $\cF_B(\alpha) = \alpha^{(m)}$. In particular, the $\cR$-linear map $\cF_B$ is independent of the choice of $B$. So, we will omit the subscript for $\cF_B$.

(b) $\cF\colon\cS_n(\Sigma,\mathbbm{u})\rightarrow \cS_n(\Si,\mathbbm{v})$ is an algebra embedding.

(c) For any interior ideal arc $c$ of $\Si$, the following diagram commutes.
\begin{equation*}
\begin{tikzcd}
\cS_n(\Si,\mathbbm{u})  \arrow[r, "\cF"]
\arrow[d, "\Theta_c"]  
&  \cS_n(\Si,\mathbbm{v})  \arrow[d, "\Theta_c"] \\
 \cS_n(\text{Cut}_c\Si,\mathbbm{u})  \arrow[r, "\cF"] 
&  \cS_n(\text{Cut}_c\Si,\mathbbm{v})\\
\end{tikzcd}
\end{equation*}
\end{prop}

In the rest of this subsection, we mainly focus on proving Proposition \ref{Fro}.
\def \tSi{\Si\times(-1,1)}

Note that
we review some relations calculated in \cite{Wan23} for $m$ parallel copies of arcs.
Although the author restricts to the case when $m$ and $2n$ are coprime, the arguments in \cite{Wan23} apply to general roots of unity. 

\begin{lem}[\cite{Wan23}, Corollary 7.14]\label{lem_splitting}
    Suppose $\alpha\in \cS_n(\Sigma,\mathbbm{v})$ is a stated $n$-web diagram in $\tSi$ consisting of arcs, and $e$ is an interior ideal arc in $\Si$ such that $\alpha$ is transverse to $e$ and $\alpha\cap e\neq\emptyset$.
    Let $h$ be a linear order on $\alpha\cap e$.
    Then we have 
    $$\Theta_e(\alpha^{(m)}) = \sum_{s\colon\alpha\cap e\rightarrow \{1,2,\cdots,n\}} \alpha(h,s)^{(m)}\in \cS_n(\text{Cut}_e\Si,\mathbbm{v}).$$
\end{lem}

Let  
$
\begin{tikzpicture}
\tikzset{->-/.style=
{decoration={markings,mark=at position #1 with
{\arrow{latex}}},postaction={decorate}}}
\draw [color = black, line width =1pt](0.6,0) --(0,0);
\draw[line width =1pt,decoration={markings, mark=at position 0.8 with {\arrow{>}}},postaction={decorate}](1,0) --(1.5,0);
\node at(0.8,0) {\small $m$};
\draw[color=black] (0.8,0) circle(0.2);
\end{tikzpicture}
$
denote the 
$m$ parallel copies of
$
\begin{tikzpicture}
\tikzset{->-/.style=
{decoration={markings,mark=at position #1 with
{\arrow{latex}}},postaction={decorate}}}
\draw [color = black, line width =1pt](0.6,0) --(1,0);
\draw[line width =1pt,decoration={markings, mark=at position 0.8 with {\arrow{>}}},postaction={decorate}](1,0) --(1.5,0);
\end{tikzpicture}
$.

The following four lemmas are straightforward from the proofs in \cite{Wan23}. 
\begin{lem}[{\cite[Lemmas 7.2, 7.3, and Corollary 7.17]{Wan23}}]\label{lem_cross}
In $\cS_n(\Si,\mathbbm{v})$, we have
\begin{equation}\label{eqq14}
\raisebox{-.10in}{
\begin{tikzpicture}
\tikzset{->-/.style=
{decoration={markings,mark=at position #1 with
{\arrow{latex}}},postaction={decorate}}}
\draw [color = black, line width =1pt](0,0) --(-0.7,0);
\draw [color = black, line width =1pt](0,0) --(0.4,0);
\draw[line width =1pt,decoration={markings, mark=at position 0.8 with {\arrow{>}}},postaction={decorate}](0.8,0) --(1.3,0);
\draw[line width =1pt,decoration={markings, mark=at position 0.8 with {\arrow{>}}},postaction={decorate}](0,0.1) --(0,0.5);
\draw [color = black, line width =1pt](0,-0.5) --(0,-0.1);
\node at(0.6,0) {\small $m$};
\draw[color=black] (0.6,0) circle(0.2);
\end{tikzpicture}}
=
 q^{-\frac{2m}{n}}
\raisebox{-.10in}{
\begin{tikzpicture}
\tikzset{->-/.style=
{decoration={markings,mark=at position #1 with
{\arrow{latex}}},postaction={decorate}}}
\draw [color = black, line width =1pt](-0.1,0) --(-0.7,0);
\draw [color = black, line width =1pt](0.1,0) --(0.4,0);
\draw[line width =1pt,decoration={markings, mark=at position 0.8 with {\arrow{>}}},postaction={decorate}](0.8,0) --(1.3,0);
\draw[line width =1pt,decoration={markings, mark=at position 0.8 with {\arrow{>}}},postaction={decorate}](0,0) --(0,0.5);
\draw [color = black, line width =1pt](0,-0.5) --(0,0);
\node at(0.6,0) {\small $m$};
\draw[color=black] (0.6,0) circle(0.2);
\end{tikzpicture}}, \qquad
\raisebox{-.10in}{
\begin{tikzpicture}
\tikzset{->-/.style=
{decoration={markings,mark=at position #1 with
{\arrow{latex}}},postaction={decorate}}}
\draw [color = black, line width =1pt](-0.1,0) --(-0.7,0);
\draw [color = black, line width =1pt](0.1,0) --(0.4,0);
\draw[line width =1pt,decoration={markings, mark=at position 0.8 with {\arrow{>}}},postaction={decorate}](0.8,0) --(1.3,0);
\draw[color = black, line width =1pt](0,0.1) --(0,0.5);
\draw [color = black, line width =1pt](0,-0.5) --(0,0.1);
\draw [color = black, line width =1pt][line width =1pt,decoration={markings, mark=at position 0.5 with {\arrow{<}}},postaction={decorate}](0,-0.5) --(0,-0.1);
\node at(0.6,0) {\small $m$};
\draw[color=black] (0.6,0) circle(0.2);
\end{tikzpicture}}
=
 q^{-\frac{2m}{n}}
\raisebox{-.10in}{
\begin{tikzpicture}
\tikzset{->-/.style=
{decoration={markings,mark=at position #1 with
{\arrow{latex}}},postaction={decorate}}}
\draw [color = black, line width =1pt](0.1,0) --(-0.7,0);
\draw [color = black, line width =1pt](0.1,0) --(0.4,0);
\draw[line width =1pt,decoration={markings, mark=at position 0.8 with {\arrow{>}}},postaction={decorate}](0.8,0) --(1.3,0);
\draw[color = black, line width =1pt](0,0.1) --(0,0.5);
\draw [color = black, line width =1pt][line width =1pt,decoration={markings, mark=at position 0.5 with {\arrow{<}}},postaction={decorate}](0,-0.5) --(0,-0.1);
\node at(0.6,0) {\small $m$};
\draw[color=black] (0.6,0) circle(0.2);
\end{tikzpicture}}
\end{equation}
where on the left-hand side of the equality
$\begin{tikzpicture}
\tikzset{->-/.style=
{decoration={markings,mark=at position #1 with
{\arrow{latex}}},postaction={decorate}}}
\draw [color = black, line width =1pt](0,0) --(-0.7,0);
\draw [color = black, line width =1pt](0,0) --(0.4,0);
\draw[line width =1pt,decoration={markings, mark=at position 0.8 with {\arrow{>}}},postaction={decorate}](0.8,0) --(1.3,0);
\node at(0.6,0) {\small $m$};
\draw[color=black] (0.6,0) circle(0.2);
\end{tikzpicture}$ is a part of $m$ parallel copies of a stated framed oriented arc and the single vertical oriented line is not a part of these $m$ parallel copies, and
\begin{equation}\label{crossing}
\raisebox{-.20in}{
\begin{tikzpicture}
\tikzset{->-/.style=
{decoration={markings,mark=at position #1 with
{\arrow{latex}}},postaction={decorate}}}
\draw [color = black, line width =1pt](0,0) --(-0.7,0);
\draw [color = black, line width =1pt](0,0) --(0.4,0);
\draw[line width =1pt,decoration={markings, mark=at position 0.8 with {\arrow{>}}},postaction={decorate}](0.8,0) --(1.3,0);
\draw[line width =1pt,decoration={markings, mark=at position 0.8 with {\arrow{>}}},postaction={decorate}](0,0.1) --(0,0.5);
\draw [color = black, line width =1pt](0,-1) --(0,-0.1);
\filldraw[fill=white] (0,-0.5) circle(0.2);
\node at(0.6,0) {\small $m$};
\node at(0,-0.5) {\small $m$};
\draw[color=black] (0.6,0) circle(0.2);
\end{tikzpicture}}
=
 \varepsilon^{-\frac{2}{n}}
\raisebox{-.20in}{
\begin{tikzpicture}
\tikzset{->-/.style=
{decoration={markings,mark=at position #1 with
{\arrow{latex}}},postaction={decorate}}}
\draw [color = black, line width =1pt](-0.1,0) --(-0.7,0);
\draw [color = black, line width =1pt](0.1,0) --(0.4,0);
\draw[line width =1pt,decoration={markings, mark=at position 0.8 with {\arrow{>}}},postaction={decorate}](0.8,0) --(1.3,0);
\draw[line width =1pt,decoration={markings, mark=at position 0.8 with {\arrow{>}}},postaction={decorate}](0,0) --(0,0.5);
\draw [color = black, line width =1pt](0,-1) --(0,0);
\node at(0.6,0) {\small $m$};
\filldraw[fill=white] (0,-0.5) circle(0.2);
\node at(0,-0.5) {\small $m$};
\draw[color=black] (0.6,0) circle(0.2);
\end{tikzpicture}},
\end{equation}
where 
$\begin{tikzpicture}
\tikzset{->-/.style=
{decoration={markings,mark=at position #1 with
{\arrow{latex}}},postaction={decorate}}}
\draw [color = black, line width =1pt](0,0) --(-0.7,0);
\draw [color = black, line width =1pt](0,0) --(0.4,0);
\draw[line width =1pt,decoration={markings, mark=at position 0.8 with {\arrow{>}}},postaction={decorate}](0.8,0) --(1.3,0);
\node at(0.6,0) {\small $m$};
\draw[color=black] (0.6,0) circle(0.2);
\end{tikzpicture}$ is part of $m$ parallel copies of a stated framed oriented arc.

\end{lem}

\begin{lem}[{\cite[Lemmas 7.6, 7.7, and Corollary 7.18]{Wan23}}]\label{lem7.5}
In $\cS_n(\Sigma,\mathbbm{v})$, we have 
$$
\raisebox{-.30in}{
\begin{tikzpicture}
\tikzset{->-/.style=
{decoration={markings,mark=at position #1 with
{\arrow{latex}}},postaction={decorate}}}
\draw [line width =1.5pt,decoration={markings, mark=at position 0.91 with {\arrow{>}}},postaction={decorate}] (0,1)--(0,-1);
\draw [color = black, line width =1pt](-1,0.5) --(-0,-0.5);
\draw [color = black, line width =1pt](-1,-0.5) --(-0.6,-0.1);
\draw [color = black, line width =1pt](-0.4,0.1) --(0,0.5);
\draw [color = black, line width =1pt](-1.5,-0.5) --(-1,-0.5);
\draw [color = black, line width =1pt](-1.5,0.5) --(-1.4,0.5);
\node [right]at(0,0.5) {\small $i$};
\node [right]at(0,-0.5) {\small $j$};
\node at(-1.2,0.5) {\small $m$};
\draw[color=black] (-1.2,0.5) circle (0.2);
\filldraw[fill=white,line width =0.8pt]  (-0.2,0.3) circle (0.085);
\filldraw[fill=white,line width =0.8pt]  (-0.2,-0.3) circle (0.085);
\end{tikzpicture}}= 
     (q^{-\frac{1}{n}+\delta_{ij}})^{m}
\raisebox{-.30in}{
\begin{tikzpicture}
\tikzset{->-/.style=
{decoration={markings,mark=at position #1 with
{\arrow{latex}}},postaction={decorate}}}
\draw [line width =1.5pt,decoration={markings, mark=at position 0.91 with {\arrow{>}}},postaction={decorate}] (0,1)--(0,-1);
\draw [color = black, line width =1pt](-1,-0.5) --(0,-0.5);
\draw [color = black, line width =1pt](-1,0.5) --(-0,0.5);
\draw [color = black, line width =1pt](-0.3,0.5) --(0,0.5);
\node [right]at(0,0.5) {\small $j$};
\node [right]at(0,-0.5) {\small $i$};
\filldraw[fill=white] (-0.6,0.5) circle (0.2);
\filldraw[fill=white,line width =0.8pt]  (-0.2,0.5) circle (0.085);
\filldraw[fill=white,line width =0.8pt]  (-0.2,-0.5) circle (0.085);
\node at(-0.6,0.5) {\small $m$};
\end{tikzpicture}},  \qquad  
\raisebox{-.30in}{
\begin{tikzpicture}
\tikzset{->-/.style=
{decoration={markings,mark=at position #1 with
{\arrow{latex}}},postaction={decorate}}}
\draw [line width =1.5pt,decoration={markings, mark=at position 0.91 with {\arrow{>}}},postaction={decorate}] (0,1)--(0,-1);
\draw [color = black, line width =1pt](-1,0.5) --(-0,-0.5);
\draw [color = black, line width =1pt](-1,-0.5) --(-0.6,-0.1);
\draw [color = black, line width =1pt](-0.4,0.1) --(0,0.5);
\draw [color = black, line width =1pt](-1.5,-0.5) --(-1,-0.5);
\draw [color = black, line width =1pt](-1.5,0.5) --(-1.4,0.5);
\node [right]at(0,0.5) {\small $i$};
\node [right]at(0,-0.5) {\small $j$};
\node at(-1.2,0.5) {\small $m$};
\draw[color=black] (-1.2,0.5) circle (0.2);
\filldraw[fill=black,line width =0.8pt]  (-0.2,0.3) circle (0.085);
\filldraw[fill=white,line width =0.8pt]  (-0.2,-0.3) circle (0.085);
\end{tikzpicture}}= 
     (q^{\frac{1}{n}-\delta_{i\bar{j}}})^{m}
\raisebox{-.30in}{
\begin{tikzpicture}
\tikzset{->-/.style=
{decoration={markings,mark=at position #1 with
{\arrow{latex}}},postaction={decorate}}}
\draw [line width =1.5pt,decoration={markings, mark=at position 0.91 with {\arrow{>}}},postaction={decorate}] (0,1)--(0,-1);
\draw [color = black, line width =1pt](-1,-0.5) --(0,-0.5);
\draw [color = black, line width =1pt](-1,0.5) --(-0,0.5);
\draw [color = black, line width =1pt](-0.3,0.5) --(0,0.5);
\node [right]at(0,0.5) {\small $j$};
\node [right]at(0,-0.5) {\small $i$};
\filldraw[fill=white] (-0.6,0.5) circle (0.2);
\filldraw[fill=white,line width =0.8pt]  (-0.2,0.5) circle (0.085);
\filldraw[fill=black,line width =0.8pt]  (-0.2,-0.5) circle (0.085);
\node at(-0.6,0.5) {\small $m$};
\end{tikzpicture}},   
%
$$
where on the left-hand side of the equality 
$\begin{tikzpicture}
\tikzset{->-/.style=
{decoration={markings,mark=at position #1 with
{\arrow{latex}}},postaction={decorate}}}
\draw [color = black, line width =1pt](0.8,0) --(1.2,0);
\draw [color = black, line width =1pt](0,0) --(0.4,0);
%
\node at(0.6,0) {\small $m$};
\draw[color=black] (0.6,0) circle(0.2);
\end{tikzpicture}$ is a part of  $m$ parallel copies of a stated framed oriented arc and the other single arc (the one stated by $i$) is not a part of these $m$ parallel copies, and 
\begin{equation}\label{height}
\raisebox{-.30in}{
\begin{tikzpicture}
\tikzset{->-/.style=
{decoration={markings,mark=at position #1 with
{\arrow{latex}}},postaction={decorate}}}
\draw [line width =1.5pt,decoration={markings, mark=at position 0.91 with {\arrow{>}}},postaction={decorate}] (0,1)--(0,-1);
\draw [color = black, line width =1pt](-1,0.5) --(-0,-0.5);
\draw [color = black, line width =1pt](-1,-0.5) --(-0.6,-0.1);
\draw [color = black, line width =1pt](-0.4,0.1) --(0,0.5);
\draw [color = black, line width =1pt](-1.5,-0.5) --(-1,-0.5);
\draw [color = black, line width =1pt](-1.5,0.5) --(-1.4,0.5);
\node [right]at(0,0.5) {\small $i$};
\node [right]at(0,-0.5) {\small $j$};
\node at(-1.2,0.5) {\small $m$};
\draw[color=black] (-1.2,0.5) circle (0.2);
\filldraw[fill=white] (-1.2,-0.5) circle (0.2);
\node at(-1.2,-0.5) {\small $m$};
\filldraw[fill=white,line width =0.8pt]  (-0.2,0.3) circle (0.085);
\filldraw[fill=white,line width =0.8pt]  (-0.2,-0.3) circle (0.085);
\end{tikzpicture}}= 
     \varepsilon^{-\frac{1}{n}+\delta_{ij}}
\raisebox{-.30in}{
\begin{tikzpicture}
\tikzset{->-/.style=
{decoration={markings,mark=at position #1 with
{\arrow{latex}}},postaction={decorate}}}
\draw [line width =1.5pt,decoration={markings, mark=at position 0.91 with {\arrow{>}}},postaction={decorate}] (0,1)--(0,-1);
\draw [color = black, line width =1pt](-1,-0.5) --(0,-0.5);
\draw [color = black, line width =1pt](-1,0.5) --(-0,0.5);
\draw [color = black, line width =1pt](-0.3,0.5) --(0,0.5);
\node [right]at(0,0.5) {\small $j$};
\node [right]at(0,-0.5) {\small $i$};
\filldraw[fill=white] (-0.6,0.5) circle (0.2);
\filldraw[fill=white,line width =0.8pt]  (-0.2,0.5) circle (0.085);
\filldraw[fill=white,line width =0.8pt]  (-0.2,-0.5) circle (0.085);
\node at(-0.6,0.5) {\small $m$};
\filldraw[fill=white] (-0.6,-0.5) circle (0.2);
\node at(-0.6,-0.5) {\small $m$};
\end{tikzpicture}},  \qquad  
\raisebox{-.30in}{
\begin{tikzpicture}
\tikzset{->-/.style=
{decoration={markings,mark=at position #1 with
{\arrow{latex}}},postaction={decorate}}}
\draw [line width =1.5pt,decoration={markings, mark=at position 0.91 with {\arrow{>}}},postaction={decorate}] (0,1)--(0,-1);
\draw [color = black, line width =1pt](-1,0.5) --(-0,-0.5);
\draw [color = black, line width =1pt](-1,-0.5) --(-0.6,-0.1);
\draw [color = black, line width =1pt](-0.4,0.1) --(0,0.5);
\draw [color = black, line width =1pt](-1.5,-0.5) --(-1,-0.5);
\draw [color = black, line width =1pt](-1.5,0.5) --(-1.4,0.5);
\node [right]at(0,0.5) {\small $i$};
\node [right]at(0,-0.5) {\small $j$};
\node at(-1.2,0.5) {\small $m$};
\draw[color=black] (-1.2,0.5) circle (0.2);
\filldraw[fill=black,line width =0.8pt]  (-0.2,0.3) circle (0.085);
\filldraw[fill=white,line width =0.8pt]  (-0.2,-0.3) circle (0.085);
\filldraw[fill=white] (-1.2,-0.5) circle (0.2);
\node at(-1.2,-0.5) {\small $m$};
\end{tikzpicture}}= 
     \varepsilon^{\frac{1}{n}-\delta_{i\bar{j}}}
\raisebox{-.30in}{
\begin{tikzpicture}
\tikzset{->-/.style=
{decoration={markings,mark=at position #1 with
{\arrow{latex}}},postaction={decorate}}}
\draw [line width =1.5pt,decoration={markings, mark=at position 0.91 with {\arrow{>}}},postaction={decorate}] (0,1)--(0,-1);
\draw [color = black, line width =1pt](-1,-0.5) --(0,-0.5);
\draw [color = black, line width =1pt](-1,0.5) --(-0,0.5);
\draw [color = black, line width =1pt](-0.3,0.5) --(0,0.5);
\node [right]at(0,0.5) {\small $j$};
\node [right]at(0,-0.5) {\small $i$};
\filldraw[fill=white] (-0.6,0.5) circle (0.2);
\filldraw[fill=white,line width =0.8pt]  (-0.2,0.5) circle (0.085);
\filldraw[fill=black,line width =0.8pt]  (-0.2,-0.5) circle (0.085);
\filldraw[fill=white] (-0.6,-0.5) circle (0.2);
\node at(-0.6,-0.5) {\small $m$};
\node at(-0.6,0.5) {\small $m$};
\end{tikzpicture}},   
%
\end{equation}
where each  
$\begin{tikzpicture}
\tikzset{->-/.style=
{decoration={markings,mark=at position #1 with
{\arrow{latex}}},postaction={decorate}}}
\draw [color = black, line width =1pt](0.8,0) --(1.2,0);
\draw [color = black, line width =1pt](0,0) --(0.4,0);
%
\node at(0.6,0) {\small $m$};
\draw[color=black] (0.6,0) circle(0.2);
\end{tikzpicture}$  is a part of  $m$ parallel copies of a stated framed oriented arc.
\end{lem}

\begin{lem}[{\cite[Lemma 7.20]{Wan23}}]\label{lem_wall}
In $\cS_n(\Sigma,\mathbbm{v})$, we have 
\begin{equation}\label{86}
\sum_{1\leq k\leq n}\mathbbm{c}_{\bar{k},\mathbbm{u}}^{-1}
\raisebox{-.30in}{
\begin{tikzpicture}
\tikzset{->-/.style=
{decoration={markings,mark=at position #1 with
{\arrow{latex}}},postaction={decorate}}}
\draw [line width =1.5pt,decoration={markings, mark=at position 0.91 with {\arrow{>}}},postaction={decorate}] (0,1)-- (0,-1);
\draw [line width =1pt] (-1.3,-0.5)--(-0.7,-0.5);
\filldraw[fill=black] (-1,-0.5) circle (0.1);
\draw [color = black, line width =1pt](-0,-0.5)--(-0.3,-0.5);
\node at(-0.5,-0.5) {\small $m$};
\node [right] at(-0,-0.5) {\small $\bar{k}$};
\draw[color=black] (-0.5,-0.5) circle (0.2);
\draw [line width =1pt] (-1.3,0.5)--(-0.7,0.5);
\draw [color = black, line width =1pt](-0,0.5)--(-0.3,0.5);
\node at(-0.5,0.5) {\small $m$};
\node [right] at(-0,0.5) {\small $k$};
\draw[color=black] (-0.5,0.5) circle (0.2);
\filldraw[fill=white] (-1,0.5) circle (0.1);
\end{tikzpicture}}
= 
\raisebox{-.30in}{
\begin{tikzpicture}
\tikzset{->-/.style=
{decoration={markings,mark=at position #1 with
{\arrow{latex}}},postaction={decorate}}}
\draw [line width =1.5pt,decoration={markings, mark=at position 0.91 with {\arrow{>}}},postaction={decorate}] (0.7,1)--(0.7,-1);
\draw [line width =1pt] (-1.3,-0.5)--(-0.7,-0.5);
\draw [color = black, line width =1pt](-0,-0.5)--(-0.3,-0.5);
\node at(-0.5,-0.5) {\small $m$};
\draw[color=black] (-0.5,-0.5) circle (0.2);
\draw [color = black, line width =1pt] (-1.3,0.5)--(-0.7,0.5);
\draw [color = black, line width =1pt](-0,0.5)--(-0.7,0.5);
\draw [color = black, line width =1pt] (0 ,-0.5) arc (-90:90:0.5);
\filldraw[fill=white] (-0.5,0.5) circle (0.1);
\end{tikzpicture}},
\end{equation}
where  all three
$\begin{tikzpicture}
\tikzset{->-/.style=
{decoration={markings,mark=at position #1 with
{\arrow{latex}}},postaction={decorate}}}
\draw [color = black, line width =1pt](0.8,0) --(1.2,0);
\draw [color = black, line width =1pt](0,0) --(0.4,0);
%
\node at(0.6,0) {\small $m$};
\draw[color=black] (0.6,0) circle(0.2);
\end{tikzpicture}$ are parts of  $m$ parallel copies of some stated framed oriented arc. 
\end{lem}

\begin{lem}[{\cite[Lemma 7.21]{Wan23}}]\label{lem_twist}
In $\cS_n(\Sigma,\mathbbm{v})$, we have 
\begin{equation*}
\raisebox{-.10in}{
\begin{tikzpicture}
\tikzset{->-/.style=
{decoration={markings,mark=at position #1 with
{\arrow{latex}}},postaction={decorate}}}
\draw [color = black, line width =1pt](-1,0)--(-0.25,0);
\draw [color = black, line width =1pt](0,0)--(0.6,0);
\draw [color = black, line width =1pt](1,0)--(1.5,0);
\node at(0.8,0) {\small $m$};
\draw[color=black] (0.8,0) circle (0.2);
\draw [color = black, line width =1pt] (0.166 ,0.08) arc (-37:270:0.2);
\end{tikzpicture}}=
\mathbbm{t}_\mathbbm{u}
\raisebox{-.10in}{
\begin{tikzpicture}
\tikzset{->-/.style=
{decoration={markings,mark=at position #1 with
{\arrow{latex}}},postaction={decorate}}}
\draw [color = black, line width =1pt](-0.3,0)--(0.5,0);
\draw [color = black, line width =1pt](0.23,0)--(0.6,0);
\draw [color = black, line width =1pt](1,0)--(1.5,0);
\node at(0.8,0) {\small $m$};
\draw[color=black] (0.8,0) circle (0.2);
\end{tikzpicture}},\qquad
\raisebox{-.10in}{
\begin{tikzpicture}
\tikzset{->-/.style=
{decoration={markings,mark=at position #1 with
{\arrow{latex}}},postaction={decorate}}}
\draw [color = black, line width =1pt](-1,0)--(0,0);
\draw [color = black, line width =1pt](0.23,0)--(0.6,0);
\draw [color = black, line width =1pt](1,0)--(1.5,0);
\node at(0.8,0) {\small $m$};
\draw[color=black] (0.8,0) circle (0.2);
\draw [color = black, line width =1pt] (0 ,0) arc (-90:215:0.2);
\end{tikzpicture}}
= \mathbbm{t}_\mathbbm{u}^{-1}
\raisebox{-.10in}{
\begin{tikzpicture}
\tikzset{->-/.style=
{decoration={markings,mark=at position #1 with
{\arrow{latex}}},postaction={decorate}}}
\draw [color = black, line width =1pt](-0.3,0)--(0.5,0);
\draw [color = black, line width =1pt](0.23,0)--(0.6,0);
\draw [color = black, line width =1pt](1,0)--(1.5,0);
\node at(0.8,0) {\small $m$};
\draw[color=black] (0.8,0) circle (0.2);
\end{tikzpicture}},
\end{equation*}
where  
$\begin{tikzpicture}
\tikzset{->-/.style=
{decoration={markings,mark=at position #1 with
{\arrow{latex}}},postaction={decorate}}}
\draw [color = black, line width =1pt](0.8,0) --(1.2,0);
\draw [color = black, line width =1pt](0,0) --(0.4,0);
%
\node at(0.6,0) {\small $m$};
\draw[color=black] (0.6,0) circle(0.2);
\end{tikzpicture}$ is part of  $m$ parallel copies of some stated framed oriented arc. 
\end{lem}

\begin{proof}[Proof of Proposition \ref{Fro}] Suppose $B=\{b_1,\cdots,b_r\}$. For each $1\leq t\leq r$, we regard $b_t$ as a proper embedding from $[0,1]$ to $\Sigma$. For $i,j\in\{1,2,\cdots,n\}$, we use $(b_t)_{ij}$ to denote the stated framed oriented arc obtained from $b_t$ such that $s(b_t(0)) = i,\, s(b_t(1)) = j$. From the definition of $\cF_B$  and Lemma \ref{lem_bi}, we have 
\begin{equation}
\cF_B((b_t)_{ij}) = (b_t)_{ij}^{(m)},\quad\cF_B((\cev{b_t})_{ij}) = (\cev{b_t})_{ij}^{(m)}.
\end{equation}

(a)  One can obtain $\alpha$ from $(b_t)_{ij}$ and $(\cev{b_t})_{ij}$, $1\leq t\leq r,\ 1\leq i,j\leq n$ using the relations \eqref{w.cross}, \eqref{w.twist}, \eqref{wzh.seven}, and \eqref{wzh.eight} (maybe we need to reverse the orientations of some stated $n$-webs in these relations). Lemmas \ref{lem_cross}-\ref{lem_twist} imply these relations for stated framed oriented arcs in $\cS_n(\Sigma,\mathbbm{u})$ coincide with the relations for the corresponding $m$-parallel copies of these stated framed oriented arcs in $\cS_n(\Sigma,\mathbbm{v})$. Thus we have $\cF_B(\alpha) = \alpha^{(m)}$.

(b) The claim follows from (a).

(c) Lemma \ref{lem_splitting} implies the claim.
\end{proof}

\paragraph{\textbf{The Frobenius map for reduced stated $\SL(n)$-skein algebras}}
For any essentially bordered pb surface $\Sigma$,
clearly $\cF\colon \cS_n(\Sigma,\mathbbm{u})\rightarrow \cS_n(\Sigma,\mathbbm{v})$ sends the bad arc to $m$ parallel copies of the bad arc. Then 
$\cF$ induces an algebra homomorphism $\overline{\cF}\colon
\barS_n(\Sigma,\mathbbm{u}) \rightarrow \barS_n(\Sigma,\mathbbm{v})$.

\section{Center of stated $\SL(n)$-skein algebras}\label{sec-centers}
In this section,
we formulate the center of the stated $\SL(n)$-skein algebra when $\hat q$ is a root of unity. Studying the center of the stated $\SL(n)$-skein algebra is crucial to understand the representation theory of the stated $\SL(n)$-skein algebras. 
As we shall see in Section~\ref{sec-Unicity-Theorem}, the Unicity theorem holds for stated $\SL(n)$-skein algebras, which classifies a particular family of irreducible representations of the stated $\SL(n)$-skein algebra with the `maximal' dimension. 
In Section~\ref{sec-Unicity-Theorem}, we will use the center
of the stated $\SL(n)$-skein algebra to precisely calculate this maximal dimension.

Thanks to
\eqref{eq-sandwitch-property},
to understand the center of the stated $\SL(n)$-skein algebra, it suffices to understand the center of the corresponding $\mathcal A$-quantum torus.

Throughout the section, we assume $\Sigma$ is an essentially bordered pb surface and contains no interior punctures.

\subsection{Ordered vertex sets}\label{sec:vertex}
For a triangulable pb surface $\Sigma$ with a triangulation $\lambda$, we introduced the small vertex set $\overline V_{\lambda}$ in Section \ref{subsec:FGalg}. Define the subset $\obV_{\lambda}\subset \overline{V}_\lambda$ consisting of all the small vertices contained in the interior of $\Sigma$.

Fix an ideal triangulation $\lambda$ of an essentially bordered pb surface $\Sigma$ and consider the extended triangulation $\lambda^\ast$. An advantage to consider such an extended triangulation, one can construct an embedding of stated $\SL(n)$-skein algebra into the quantum torus associated with $\lambda^\ast$.
Let $v_i,u_i,w_i\ (i=1,\dots,n-1)$ denote the vertices on the boundary of an attached triangle as in Figure \ref{Fig;coord_uvw}. 
Note that $v_i$'s  are on the attached edge $e_1$, and 
$$v_i=(n-i,i,0),\quad u_i=(i,0,n-i),\quad w_i=(0,i,n-i).$$ 
For two attached triangles $\tau$ and $\tau'$, we will use 
$v_i,u_i,w_i$ (resp. $v_i',u_i',w_i'$) to denote $(n-i,i,0), (i,0,n-i), (0,i,n-i)$ of $\tau$ (resp. $\tau'$).
\begin{figure}[h]
    \centering
    \includegraphics[width=140pt]{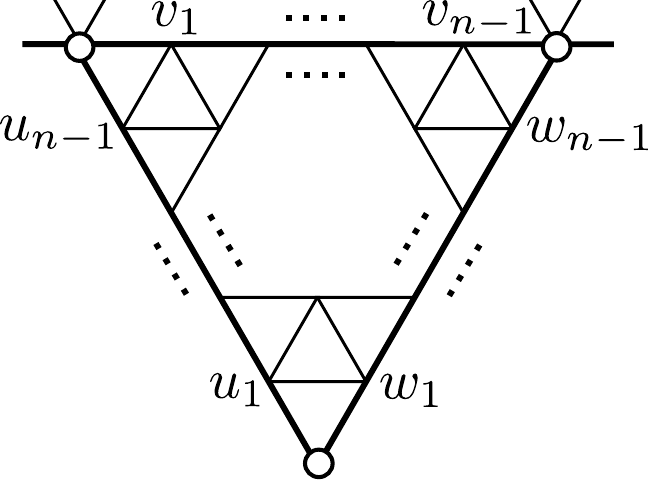}
    \caption{The vertices $v_i$ are on the attached edge.}\label{Fig;coord_uvw}
\end{figure}

Suppose $\overline{\Sigma}$ has boundary components $\partial_1,\cdots,\partial_b$, and $\partial_i$ contains $r_i$ punctures for each $1\leq i\leq b$.
For each $1\leq i\leq b$, we label the boundary edges contained in $\partial_i$

Let $W$ (resp. $U$) be the ordered set of the small vertices on $V_\lambda\setminus\obVlast$ (resp.$V'_\lambda\setminus\obVlast$) 
with the order defined as follows. 
On the boundary component $\partial_i$ of $\overline{\Sigma}$, we label the connected components of $\partial_i\cap \Sigma$ by $e_1,e_2, \dots,e_{r_i}$ consecutively following the positive orientation of $\partial \overline{\Sigma}$. 
For each $e_i$, there is a unique attached triangle $\tau_j$ in $\lambda^\ast$ containing $e_i$. 
Each $w\in V_\lambda\setminus\obVlast$ is determined by $(\partial_i,\tau_j,w_k)$, where $w$ is the small vertex with  the coordinate $w_k$ in the attached triangle $\tau_j$, which is attached to $\partial_i$. 
We identify $(\partial_i,\tau_j,w_k)$ with $(b-i,r_i-j,n-k)$.
Consider the lexicographic order on $W$ with respect to $(b-i,r_i-j,n-k)$ and suppose that $W$ is equipped with the order. 
In the same manner, we define an order on $U$ and suppose $U$ is  also equipped with the order.

\subsection{Block decomposition of matrices}

The bilinear form $\barH_\lambda\colon \barV_\lambda\times\barV_\lambda\to\bZ$ is defined as
\begin{itemize}
\item for $v, v'\in \barV_\lambda$ not on the same boundary edge, $\barH_\lambda(v,v')=-\frac{1}{2}\barQ_\lambda(v,v')\in\bZ$,
\item for $v,v'\in \barV_\lambda$ on the same boundary edge, 
$$\barH_\lambda(v,v') = \begin{cases} -1 \qquad &\text{if $v=v'$},\\
1 &\text{if there is an arrow from $v$ to $v'$}, \\
0 &\text{otherwise}.
\end{cases}$$
\end{itemize}
Define $\sfH_\lambda$ as the restriction of $\barH_{\lambda^\ast}$ to $V_\lambda\times V'_\lambda$. 
\begin{lem}[{\cite[A part of Lemma 11.9]{LY23}}]\label{lem:invertible_KH}
The following matrix identities holds.
\begin{enumerate}
    \item $n(\mathsf{K}_\lambda -\mathsf{K}_\lambda^T) =\mathsf{P}_\lambda$.
    \item $\barK_\lambda\barH_\lambda=nI$ and $\sfK_\lambda\sfH_\lambda=nI$.
\end{enumerate} 
\end{lem}

\begin{prop}[{\cite[Propositon 11.10]{LY23}}]\label{prop:LY23_11.10}
Let $\bk$ be a vector in $\bZ^{\barV_\lambda}$. Then the following are equivalent.
\begin{enumerate}
\item $\bk$ is balanced.
\item $\bk\barH_\lambda\in(n\bZ)^{\barV_\lambda}$.
\item There exists a vector $\mathbf{c}\in\bZ^{\barV_\lambda}$ such that $\bk=\mathbf{c}\barK_\lambda$.
\end{enumerate}
The same results hold for the non-reduced case, i.e., when $\barV_\lambda, \barH_\lambda, \barK_\lambda$ are replaced with ${V}_\lambda, \sfH_\lambda, \sfK_\lambda$ respectively.
\end{prop}

When we regard $\sfC$ as a ($\obVlast, U)\times (\obVlast,W, U)$-matrix, 
we have the following form; 
$$\sfC=\begin{pmatrix}
I&C_1&O\\
O&-I&I
\end{pmatrix}.$$

\begin{lem}\label{lemKQ}
    We have $\barK_{\lambda} \barQ_{\lambda} =
    \begin{pmatrix}
-2nI &  \ast \\
O    &  \ast \\
\end{pmatrix}$,
where the rows and columns are divided in 
$(\obV_{\lambda}, \overline{V}_{\lambda}\setminus \obV_{\lambda})$.
\end{lem}
\begin{proof}
In the proof, all divisions for rows or columns are in 
$(\obV_{\lambda}, \overline{V}_{\lambda}\setminus \obV_{\lambda})$.
Suppose $\barH_{\lambda} = 
\begin{pmatrix}
H_1 &H_2
\end{pmatrix}$ and $\overline{\mathsf Q}_\lambda= 
\begin{pmatrix}
    Q_1 & Q_2
\end{pmatrix}$. From the definition of $\barH_{\lambda}$, we have $H_1 = -\frac{1}{2}Q_1$.
Lemma~\ref{lem:invertible_KH} implies $\barK_{\lambda}
\begin{pmatrix}
H_1 &H_2
\end{pmatrix}
= \begin{pmatrix}
\barK_{\lambda}H_1 &\barK_{\lambda}H_2
\end{pmatrix} = nI$.
Then $\barK_{\lambda} H_1 =   \begin{pmatrix}
nI  \\
O     \\
\end{pmatrix}$. We have $\barK_{\lambda} Q_1 = -2\barK_{\lambda} H_1 =   \begin{pmatrix}
-2nI  \\
O     \\
\end{pmatrix}$.
\end{proof}

From Lemma \ref{lemKQ}, $\barKl\barQl$ has the following form; 
\begin{align}
\barKl\barQl=
\begin{pmatrix}
-2nI & D & \ast \\
O    & A & \ast \\
O    & B & \ast \\
\end{pmatrix},   \label{eq:matrix_barKQast}
\end{align}
where the rows and columns are divided in $(\obVlast, W, U)$. 

We have the $(\obVlast,U)\times(\obVlast,W, U)$-matrix
\begin{equation}\label{CKQ1}
\sfC\barKl\barQl=
\begin{pmatrix}
-2nI & D+C_1A & \ast\\
O    & B-A & \ast
\end{pmatrix}
\end{equation}

Suppose 
\begin{align}
\barKl = \begin{pmatrix}
K_{11} & K_{12} & K_{13} \\
K_{21}    & K_{22} & K_{23} \\
K_{31}    & K_{32} & K_{33} \\
\end{pmatrix},\label{eq:matrix_K_docomp3}    
\end{align}
where the rows and columns are divided in $(\obVlast, W, U)$.

\begin{lem}[{\cite[Lemma 11.6]{LY23}}]\label{Lem:LY23_11.6}
The restriction of $\sfC \barKl$ to $V_\lambda'\times(\barVlast \setminus V_\lambda)$ is $O$.
\end{lem}

Lemma \ref{Lem:LY23_11.6} implies 
$$\sfC\barKl = \begin{pmatrix}K_{11}+C_1 K_{21} & K_{12}+C_1K_{22} & O\\
K_{31} - K_{21}    & K_{32} - K_{22} & O
\end{pmatrix} = \begin{pmatrix}
\sfK_{\lambda} & O
\end{pmatrix},$$
where 
\begin{equation}\label{eq_K}
\sfK_{\lambda}=(\sfC\barKl)|_{V_\lambda'\times V_\lambda} = \begin{pmatrix}
K_{11}+C_1 K_{21} & K_{12}+C_1 K_{22} \\
K_{31} - K_{21}    & K_{32} - K_{22}
\end{pmatrix}.
\end{equation}

Suppose $\barQl = \begin{pmatrix}
\sfQ_{\lambda} & Q_{12}\\
Q_{21}    & Q_{22} 
\end{pmatrix},$
where the rows and columns are divided in $(V_{\lambda},U)$.
Then we have 
\begin{equation}\label{CKQ2}
    \sfC\barKl\barQl = \begin{pmatrix}
\sfK_{\lambda} & O
\end{pmatrix}
\begin{pmatrix}
\sfQ_{\lambda} & Q_{12}\\
Q_{21}    & Q_{22} 
\end{pmatrix}
=\begin{pmatrix}
\sfK_{\lambda} \sfQ_{\lambda} & \ast
\end{pmatrix}.
\end{equation}

Compare equations \eqref{CKQ1} and \eqref{CKQ2}, we have
\begin{equation}\label{KQ}
\sfK_{\lambda}\sfQ_{\lambda} = 
\begin{pmatrix}
-2nI & D+C_1 A\\
   O & B-A
\end{pmatrix}.
\end{equation}

\subsection{Explicit presentations of matrices}

We will explicitly give the submatrix $B-A$ of $\sfK_\lambda \sfQ_{\lambda}$ in \eqref{KQ} and show some property of $D+C_1A$. 
\begin{lem}\label{matrixA}
    We have $A = -n I$. 
\end{lem}
\begin{proof} 
For any $w\in W$, $\widetilde{Y}_w$ is an arc segment in a triangle and $Y(\widetilde{Y}_w)=w$. By combining with the definitions of $\barKl$ and $\barQl$, this implies that $\barKl\barQl(w,w')=0$ if $w$ and $w'$ are not in a same attached triangle. For $w_i,w_j\in W$ in a same attached triangle $\tau$, we have 
\begin{eqnarray*}
\barKl\barQl(w_i,w_j)&=&\sum_{v\in \overline{V}_{\lambda^{*}}}\barKl(w_i,v)\barQl(v,w_j)\\
&=&2\barK_{\tau}(w_i,(1,j,n-j-1))-2\barK_{\tau}(w_i,(1,j-1,n-j))\\
&\ & +\barK_{\tau}(w_i,w_{j-1})
-\barK_{\tau}(w_i,w_{j+1})=-n \delta_{ij}, 
\end{eqnarray*}
where we formally set $\barK_{\tau}(w_i,w_{0})=0=\barK_{\tau}(w_i,w_{n})$ and 
the last equality follows by concrete computations with $w_i=(0, i, n-i)$ and the definition of $\barK_{\tau}=\barK_{\mathbb P_3}$. 
\end{proof}

\begin{lem}\label{matrixB}
    Suppose $\overline{\Sigma}$ has boundary components $\partial_1,\cdots,\partial_b$, and $\partial_i$ contains $r_i$ punctures for each $1\leq i\leq b$.
    We have $B = diag\{B_{1},B_2,\cdots,B_b\}$, where $B_i$ is the matrix associated to $\partial_i$ for $1\leq i\leq b$. Furthermore, for each $1\leq i\leq b$, we have 
    \begin{equation}\label{Bi}
    B_i = \begin{pmatrix}
     O & O &  \cdots &O & nI \\
     nI & O & \cdots &O & O \\
     O  & nI& \cdots & O & O\\
     \vdots & \vdots &  & \vdots & \vdots \\
     O & O & \cdots  & nI &  O \\
    \end{pmatrix} \text{\ if $r_i>1$, and\ } 
    B_i = nI \text{\ if $r_i=1$}, 
    \end{equation}
    where $B_i$ is of size $r_i(n-1)$ and  every block matrix is of size $n-1$.
\end{lem}
\begin{proof}
From the definition of skeletons, $\barKl\barQl(u,w)=0$ if $u$ and $w$ are in triangles attached to different boundary components of $\overline{\Sigma}$. 
Hence, it suffices to consider $u,w$ are in triangles attached to the same boundary component of $\overline{\Sigma}$. 
Fix $w_j\in W$ and let $\tau$ be the triangle containing $w_j$. Let $r$ be the number of punctures on the boundary component of $\overline{\Sigma}$ intersecting $\tau$.

\noindent\textbf{Case of $r>1$.} 
Let $\tau'$ be the attached triangle sharing a vertex with $\tau$ such that the attached edge of $\tau'$ follows that of $\tau$ in clockwise with respect to the vertex. See Figure \ref{Fig;tau_tau'}. 
\begin{figure}[h]
    \centering
    \includegraphics[width=170pt]{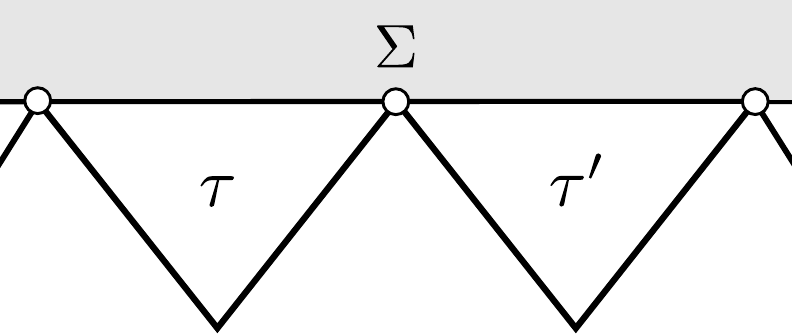}
    \caption{A local picture of $\Sigma^\ast$}\label{Fig;tau_tau'}
\end{figure}

For $u_i\in (U\cap \tau)$, $\skeleton_\tau(u_i)=u_i.$ 
One can easily show  
\begin{eqnarray*}
\barKl\barQl(u_i,w_j)&=&
2\barKt(u_i,(1,j,n-j-1))-2\barKt(u_i,(1,j-1,n-j))\\
&\ &+\barKt(u_i,w_{j-1})-\barKt(u_i,w_{j+1})=0, 
\end{eqnarray*} 
where we formally set $\barK_{\tau}(\ast,w_{0})=0=\barK_{\tau}(\ast,w_{n})$. 

For $u'_i\in (U\cap \tau')$, $\skeleton_\tau(u'_i)=v_i.$ 
One can easily show  
\begin{eqnarray*}
\barKl\barQl(u_i',w_j)&=&
2\barKt(v_i,(1,j,n-j-1))-2\barKt(v_i,(1,j-1,n-j))\\
&\ &+\barKt(v_i,w_{j-1})-\barKt(v_i,w_{j+1})=n\delta_{ij}.
\end{eqnarray*} 
In particular, for $u\not\in (U\cap (\tau\cup \tau'))$, 
$\barKl\barQl(u,w_j)=0$ from the definition of skeletons. 

\noindent\textbf{Case of $r=1$.} 
For $u_i\in (U\cap \tau)$, $\skeleton_\tau(u_i)=u_i+v_i.$
From the computations in Case of $r>1$, we have
\begin{eqnarray*}
\barKl\barQl(u_i,w_j)&=&\sum_{v\in \overline V_{\lambda^{*}}}\barKl(u_i,v)\barQl(v,w_j)\\
&=&2\barKl(u_i,(1,j,n-j-1))-2\barKl(u_i,(1,j-1,n-j))\\
&\ &+\barKl(u_i,w_{j-1})
-\barKl(u_i,w_{j+1})\\
&=&2\{\barK_{\tau}(u_i,(1,j,n-j-1))+\barK_{\tau}(v_i,(1,j,n-j-1))\}\\
&\ &-2\{\barK_{\tau}(u_i,(1,j-1,n-j))+\barK_{\tau}(v_i,(1,j-1,n-j))\}\\
&\ &+\{\barK_{\tau}(u_i,w_{j-1})+\barK_{\tau}(v_i,w_{j-1})\}
-\{\barK_{\tau}(u_i,w_{j+1})+\barK_{\tau}(v_i,w_{j+1})\}=n\delta_{ij}
\end{eqnarray*}
and $\barKl\barQl(u,w_j)=0$ for any $u\not\in U\cap \tau$.
\end{proof}

From Lemmas \ref{matrixA} and \ref{matrixB}, we have the following. 
\begin{cor}\label{invertibility} 
Suppose that the connected components of $\partial\overline{\Sigma}$ are labeled as $1,2,\dots, b$ and the $i$-th boundary component has $r_i$ punctures ($r_i$ is odd) for any $i=1,\dots,b$. 
Then, 
$\det(B_i-A_i)=2^{n-1}n^{r_i(n-1)}$ and 
$\det(B-A)=2^{b(n-1)}n^{p(n-1)}$, 
where $p$ denotes the total number of the punctures of $\overline\Sigma$.
\end{cor}

\begin{lem}\label{matrixDC}
    The matrix $D+C_1A$ is over $n\bZ$. 
\end{lem}
\begin{proof}
From Lemma \ref{matrixA}, it suffices to show that $D$ is a matrix over $n\bZ$. 
For $w_j$ in a triangle $\tau$ and $v\in \obVlast$, $$\barKl\barQl(v,w_j)=2\barKl(v,(1,j,n-j-1))-2\barKl(v,(1,j-1,n-j))+\barKl(v,w_{j-1})-\barKl(v,w_{j+1}).$$
From the definition of skeletons, 
the right-hand side is a sum of 
\begin{align}
2\barKt(\xi,(1,j,n-j-1))-2\barKt(\xi,(1,j-1,n-j))+\barKt(\xi,w_{j-1})-\barKt(\xi,w_{j+1}),\label{eq:vw_tau}
\end{align}
where $\xi$ is of $v, v_i$ in $\tau$ $(i\in \{1,2,\cdots,n-1\})$ depending on the number of punctures on each boundary component of $\overline{\Sigma}$ and the location of $v$. 
The computation for $\xi=v_i$ was done in the proof of Lemma \ref{matrixB}. 
For $\xi=v$, one can easily show that 
$\eqref{eq:vw_tau}\in n\bZ$ by considering the three cases $j'\geq j+1$,\quad $j'=j$,\quad $j'\leq j-1$ with $v=i'j'k'$. 
\end{proof}

We will explicitly give the submatrix $K_{32}-K_{22}$ of $K_\lambda$ in \eqref{eq_K}. 
\begin{lem}\label{matrixK}
    Suppose $\overline{\Sigma}$ has boundary components $\partial_1,\cdots,\partial_b$, and $\partial_i$ contains $r_i$ punctures for each $1\leq i\leq b$.
    We have $K_{32}-K_{22} = diag\{L_{1},L_2,\cdots,L_b\}$, where $L_i$ is the matrix associated to $\partial_i$ for each $1\leq i\leq b$.
    Furthermore, for each $1\leq i\leq b$, we have 
\begin{equation}\label{eq_L}
L_i = \begin{pmatrix}
-G & O & O & \cdots &O& G \\
G & -G & O & \cdots &O& O \\
\vdots & \vdots & \vdots & \; &\vdots& \vdots \\
O & O & O & \cdots &-G& O \\
O & O & O & \cdots &G& -G \\
\end{pmatrix}\text{\ if $r_i>1$, and }
L_i=O\text{ if $r_i=1$},
\end{equation}
where $G$ is the square matrix of size $(n-1)$ with
\begin{align}\label{eq-matrix-G-def}
    G_{ij} = \begin{cases}i(n-j) & i\leq j,\\
j(n-i) & i>j.
\end{cases}
\end{align}
\end{lem}
\begin{proof}
Let $\tau$ be the triangle containing $w_i\in W$ and $r$ be the number of boundary punctures on the component of $\overline{\Sigma}$ attached by $\tau$. 
When the attaching edge of $\tau$ is $e_k$ then let $\tau'$ be the attached triangle with $e_{k+1}$.

Note that, for $u''_i\in U$, $(K_{32}-K_{22})(u''_i,w_j)=K_{32}(u''_i,w_j)-K_{22}(w''_i,w_j)$ since $\mathsf{C}\barKl(u''_i,\ast)=\barKl(u''_i,\ast)-\barKl(w''_i,\ast)$, where $w''_i\in W$ is in the attached triangle containing $u''_i$.

It is easy to see that $K_\tau(u_i,w_j)=0$ and 
$$K_\tau(v_i,w_j)=K_\tau(w_i,w_j)=\begin{cases}
i(n-j) & i\leq j,\\
j(n-i) & i>j.
\end{cases}$$

When $r>1$, we have
$K_{32}(u_i,w_j)-K_{22}(w_i,w_j)=
K_\tau(u_i,w_j)+K_\tau(v_i,w_j)-K_\tau(w_i,w_j)$. 
This implies that each matrix on the diagonal of $L$ is $-G$. 
For $u'_i\in\tau'$,  
$K_{32}(u'_i,w_j)-K_{22}(w'_i,w_j)=K_\tau(v_i,w_j)$. 
This implies that we have the presentation of $L$ given by \eqref{eq_L} except for the entries written as $O$. 
Similarly to the proofs of Lemmas \ref{matrixA} and \ref{matrixB}, 
the other entries are 0.  

When $r=1$, we have 
$$K_{32}(u_i,w_j)-K_{22}(w_i,w_j)=
K_\tau(u_i,w_j)+K_\tau(v_i,w_j)-K_\tau(w_i,w_j)=0.$$
\end{proof}

Let $E$ and $F$ be square matrices of size $(n-1)$ defined by 
\begin{align}
E_{ij}=\begin{cases}
i-j+1 & \text{if $i\geq j$}\\
0 & \text{if $i<j$}
\end{cases},\qquad 
F_{ij}=\begin{cases}
n-j & \text{if $i=1$}\\
-n &\text{if $i=j+1$}\\
0 & \text{otherwise}
\end{cases}.\label{matrixEF}
\end{align}

The following lemma holds from a direct matrix computation.
\begin{lem}\label{matrixG}
    We have $EF=G$.
\end{lem}

\subsection{When $\hat q$ is generic}

\begin{prop}[generic case]
With the assumption as Corollary \ref{invertibility},
the center of stated $\SL(n)$-skein algebra of $\Sigma$ with generic $q^{1/n}$ is trivial. 
\end{prop}
\begin{proof}
The sandwiched property \eqref{eq-sandwitch-property} implies $\mathcal Z(\Ap)\subset\mathcal Z(\Sn)\subset\mathcal Z(\A)$, especially $\mathcal Z(\Sn)=\Sn\cap \mathcal Z(\A)$.
Under the same condition, $\mathcal Z(\A)$ is trivial from (a) in Lemma \ref{quantum} and Corollary \ref{invertibility}. This implies the claim for $\mathcal Z(\Sn)$.  
\end{proof}

\subsection{When $\hat q$ is a root of unity}\label{sub_center}
In the rest of this section, we suppose Condition~$(\ast)$.

We recall the mod $n$ intersection number between negatively ordered stated $n$-web diagrams introduced in \cite{Wan24}.
Suppose $\alpha$ is a negatively ordered (stated) $n$-web diagram in an essentially bordered pb surface $\Sigma$. We define $P(\alpha)$ from $\alpha$ by smoothly moving its endpoints along $\partial \Sigma$ in the positive direction until all the endpoints of $\alpha$ lie in boundary punctures. During the procedure of moving the endpoints of $\alpha$, there is no crossing created except at boundary punctures.

For every crossing $c$, we define  $w(c)\in\{-1,1\}$ as in Figure \ref{fg1}. 
\begin{figure}[h]  
	\centering\includegraphics[width=5cm]{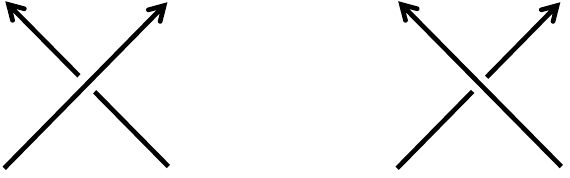} 
	\caption{Around a crossing $c$, the value $w(c)$ for the left (resp. the right) is "$1$" (resp. "$-1$").}
	\label{fg1}   
\end{figure}

 For any two (stated) $n$-web diagrams $\alpha,\beta$, we use $P(\alpha)\cup \beta$ to denote the diagram obtained from $P(\alpha)$ and $\beta$ by stacking $P(\alpha)$ above $\beta$ with only double points as transversal crossings. Define $w(P(\alpha),\beta)\in \{0,1,\dots, n-1\}$ by 
$$w(P(\alpha),\beta) \equiv \sum_{c}w(c)\mod{n},$$
where $c$ runs over the crossings in $P(\alpha)\cup\beta$ between $P(\alpha)$ and $\beta$. 
From Figure \ref{moves}, we know $w(P(\alpha),\beta)$ is well-defined.

\begin{figure}[h]  
	\centering\includegraphics[width=10cm]{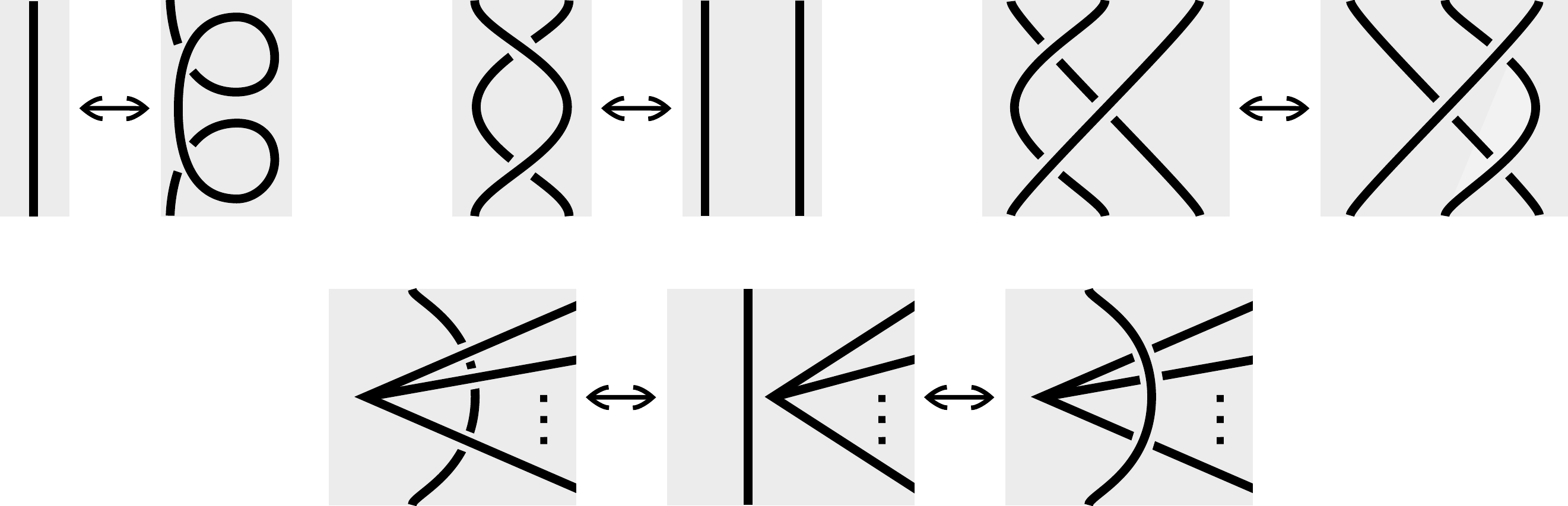}
	\caption{Four moves for the $n$-web diagrams that preserve the isotopy class of the corresponding $n$-webs.}  
	\label{moves}   
\end{figure}

Let $e$ denote a connected component of $\partial\Sigma$ and $\alpha$ be a  stated $n$-web diagram in $\Sigma$. Let $e(\alpha)$ denote the set of endpoints of $\alpha$ on $e$.

Define $\delta_e(\alpha) = \sum_{c\in e(\alpha)} sign(c)$, where $sign(c) = 1$ (resp. $sign(c) = -1$) if, at an endpoint $c$, the diagram $\alpha$ points inward (resp. outward) $e$.

For any $i\in\{1,2,\cdots,n\}$, set $1*i = i$ and $(-1)*i = \bar{i}$. Then define 
$$k_{(e,i)}(\alpha) = |\{c\in e(\alpha)\mid sign(c)*s(c) = i\}|,$$ where $s(c)$ is the state of $c$.

For any two  stated $n$-web diagrams $\alpha,\beta$, define 
$$k_e(\alpha,\beta) = \sum_{i\in\{1,2,\cdots,n\}}
k_{(e,i)}(\alpha)k_{(e,i)}(\beta).$$

The following generalizes Lemma 4.2 in \cite{Yu23} for $n=2$ to general $n$. 
\begin{lem}\label{lem_q}
 Let $\alpha\in \cS_n(\Sigma,\mathbbm{u})$ be a stated $n$-web diagram consisting of arcs, and let $\beta\in \cS_n(\Sigma,\mathbbm{v})$ be a stated $n$-web diagram. We have 
 $$\cF(\alpha)\beta =
 q^{-\frac{m}{n}\big(2 w(P(\alpha),\beta) +\sum_{e}\delta_e(\alpha)\delta_e(\beta)\big)}
 q^{m\sum_{e}k_e(\alpha,\beta)}\beta\cF(\alpha).$$
\end{lem}
\begin{proof}
Combining with Lemmas \ref{lem_cross} and \ref{lem7.5}, the proof of Lemma 4.2 in \cite{Yu23}  works here. 

\end{proof}

\begin{lem}\label{lem_central}
Let $\alpha$ be a  stated $n$-web diagram consisting of arcs. Then $\alpha^{(2md)}\in \mathcal{Z}(\cS_n(\Sigma,\mathbbm{v}))$.
If $q^m = 1$, we have $\alpha^{(dm)}\in \mathcal{Z}(\cS_n(\Sigma, \mathbbm{v}))$.
\end{lem}
\begin{proof}
    Note that $\alpha^{(2md)} = \cF(\alpha^{(2d)})$. For any stated $n$-web diagram $\beta$,  $2w(P(\alpha^{(2d)}),\beta)+\sum_{e}\delta_e(\alpha^{(2d)})\delta_e(\beta)$ is a multiple of $d$, and $\sum_{e}k_e(\alpha,\beta)$ is even. By combining $q^{\frac{md}{n}} =q^{\frac{m'}{n}} = \mathbbm{v}^{2m'} = 1$, Lemma \ref{lem_q} implies the first claim. Similarly, the second claim follows.
\end{proof}

\begin{lem}[{\cite[A part of Lemma 4.9]{LY23}}]\label{lem;LY_4.9}
For $1\leq i< j \leq n$, we have 
\begin{align*}
\begin{array}{c}\includegraphics[scale=0.27]{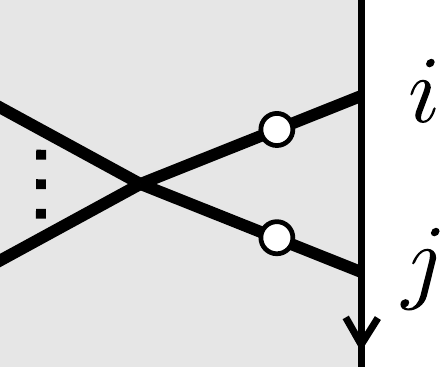}\end{array}
=(-q)\begin{array}{c}\includegraphics[scale=0.27]{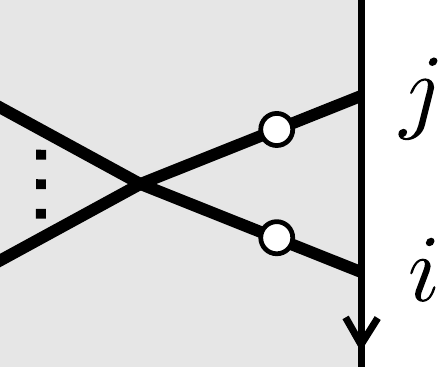}\end{array}.
\end{align*}
\end{lem}

For $\Sigma^\ast$ and $u_i=(i,0,n-i)$ in the attached triangle with the attaching edge $e$, we use $\gaa_{u_i}^e$ to denote $\gaa_{u_i}$ defined in Section~\ref{sec:quantum_trace}. 
From $u_i=(i,0,n-i)$ and  Lemma \ref{lem;LY_4.9}, 
$\gaa_{u_{i}}^{e}$ is equal to 
\begin{align}
\begin{array}{c}\includegraphics[scale=0.38]{draws/gu.pdf}\end{array}
\end{align} up to a power of $q$.

\begin{lem}\label{boundary_center}
Suppose $\overline{\Sigma}$ has a boundary component $\partial$ such that  the number of the connected components of $\partial\cap \Sigma$ is even and the connected components are labeled as $e_1,e_2,\cdots,e_r$ with respect to the orientation of $\partial$.
For every $0\leq k\leq m'$ and $1\leq i\leq n-1$, the element
\begin{align}
(\gaa_{u_{i}}^{e_1})^k (\gaa_{u_{i}}^{e_2})^{m'-k}\cdots (\gaa_{u_{i}}^{e_{r-1}})^k (\gaa_{u_{i}}^{e_r})^{m'-k}\label{central}
\end{align}
is central in $\cS_n(\Sigma,\mathbbm{v})$.
\end{lem}
\begin{proof}

It suffices to show the commutativity of the element \eqref{central} and an $n$-web diagram $\alpha$ incident to $e_l$ for some $i$, especially the commutativity of each segment of $\alpha$
incident to $e_l$ and $(\gaa_{u_{i}}^{e_{l}})^k (\gaa_{u_{i}}^{e_{l+1}})^{m'-k}$ or $(\gaa_{u_{i}}^{e_{l}})^{m'-k} (\gaa_{u_{i}}^{e_{l+1}})^k$.

\begin{lem}[{\cite[Lemma 4.10]{LY23}}]\label{lem;LY_comm}
Suppose $\mathbf{i},\mathbf{j} \subset \bJ$ are sequences of consecutive numbers, with either $\max \mathbf{i} \geq \max \mathbf{j}$ or $\min \mathbf{i} \geq \min \mathbf{j}$, then $$\begin{array}{c}\includegraphics[scale=0.38]{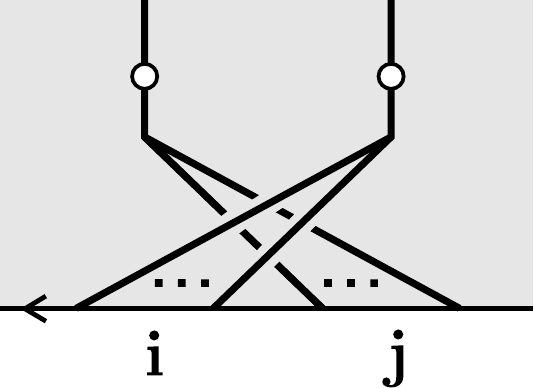}\end{array}= q^{\langle\mathbf{i},\mathbf{j}\rangle}
\begin{array}{c}\includegraphics[scale=0.38]{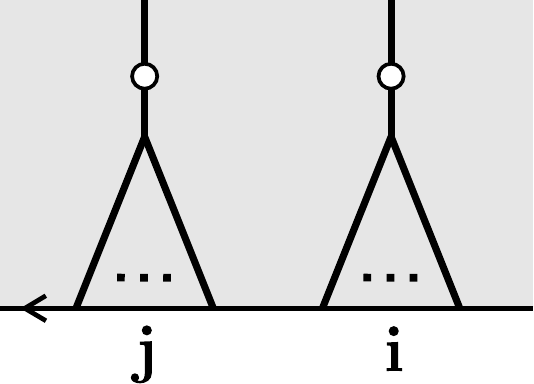}\end{array},$$ 
where the bracket was defined by \eqref{eq;bracket}, and by abuse of notations, $\mathbf{i}$ and $\mathbf{j}$ also denote the corresponding $\mathbbm{d}$-grading.
\end{lem}
Note that Lemma \ref{lem;LY_comm} works even for stated arcs. 
If the orientations of the strands of $\alpha$ and $(\gaa_{u_{i}}^{e_{j}})^k (\gaa_{u_{i}}^{e_{j+1}})^{m'-k}$ are compatible, then the exponent of $q$ is 
$$\langle w_\ell, k\mathbf{j}\rangle-\langle w_\ell, (m'-k)\mathbf{j}^c\rangle
=\langle w_\ell, k(\mathbf{j}+\mathbf{j}^c)\rangle-m'\langle w_\ell,\mathbf{j}^c\rangle=-m'\langle w_\ell,\mathbf{j}^c\rangle,$$ 
where the last equality follows from $\mathbf{j}+\mathbf{j}^c=w_1+\cdots+w_n=0$. 
From $q^{m'}=\hat{q}^{2n^2m'}=1$, each segment of $\alpha$ incident to $e_l$ and $(\gaa_{u_{i}}^{e_{j}})^k (\gaa_{u_{i}}^{e_{j+1}})^{m'-k}$ are (locally) commutative. 
With a similar computation, each segment of $\alpha$ incident to $e_l$ and $(\gaa_{u_{i}}^{e_{j}})^{m'-k} (\gaa_{u_{i}}^{e_{j+1}})^{k}$ are (locally) commutative. 

Similarly to Lemma \ref{lem;LY_comm}, we also have the following. 
\begin{lem}\label{lem;comm_op}
Suppose $\mathbf{i},\mathbf{j} \subset \bJ$ are sequences of consecutive numbers, with either $\max \mathbf{i} \geq \max \mathbf{j}$ or $\min \mathbf{i} \geq \min \mathbf{j}$, then $$\begin{array}{c}\includegraphics[scale=0.38]{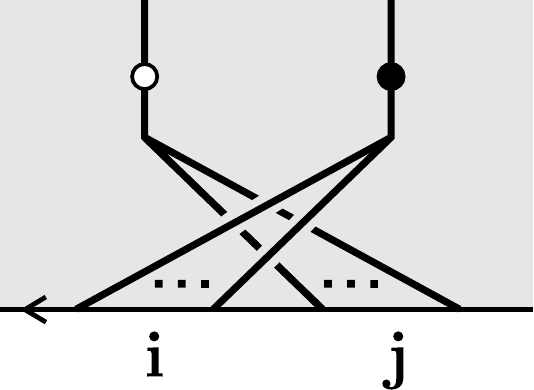}\end{array}= q^{-\langle\bar{\mathbf{i}},\mathbf{j}\rangle}
\begin{array}{c}\includegraphics[scale=0.38]{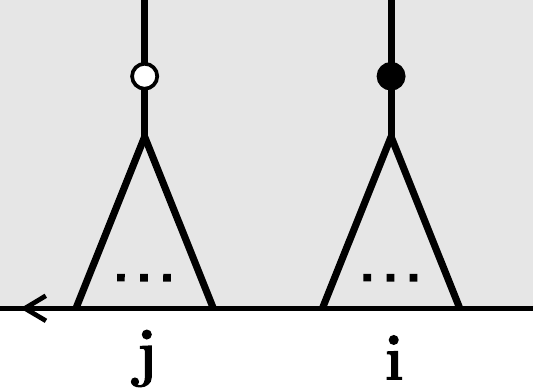}\end{array},$$ 
where the bracket was defined by \eqref{eq;bracket}, and we regard $\mathbf{i}$ and $\mathbf{j}$ as the corresponding $\mathbbm{d}$-grading. 
\end{lem}
Note that, from $\langle\bar{\mathbf{i}},\mathbf{j}\rangle=\langle\mathbf{i},\bar{\mathbf{j}}\rangle$, the coefficient in Lemma \ref{lem;comm_op} is unique. 
With Lemma \ref{lem;comm_op}, one can show the (local) commutativity of each segment of $\alpha$ incident to $e_l$ and $(\gaa_{u_{i}}^{e_{j}})^k (\gaa_{u_{i}}^{e_{j+1}})^{m'-k}$ (or $(\gaa_{u_{i}}^{e_{j}})^{m'-k} (\gaa_{u_{i}}^{e_{j+1}})^{k}$) similarly to the case when the orientations are compatible. 
\end{proof}

\begin{rem}
    Lemmas \ref{lem_q}, \ref{lem_central}, and \ref{boundary_center} hold for general essentially bordered pb surfaces. 
\end{rem}

Define \begin{align}
    \Lambda_{m'}= \{\mathbf{k}\in\mathbb Z^{V_{\lambda}'} \mid 
\mathbf{k}\sfK_{\lambda} =\bm{0} \text{ in }\mathbb Z_{m'}\},\quad \Lambda_{m'}^{+}= \{\mathbf{k}\in\mathbb N^{V_{\lambda}'} \mid 
\mathbf{k}\sfK_{\lambda} =\bm{0} \text{ in }\mathbb Z_{m'} \}.\label{eq:Lambda}
\end{align}

\begin{rem}\label{rem_coprime}
    Suppose $m'$ and $n$ are coprime. Then we will show 
$\Lambda_{m'} = (m'\mathbb Z)^{V_{\lambda}'}$ and $\Lambda_{m'}^{+} = (m' \mathbb N)^{V_{\lambda}'}$.
Obviously, we have $(m'\mathbb Z)^{V_{\lambda}'}\subset \Lambda_{m'}$.
For any $\mathbf{k}\in \Lambda_{m'}$, we have 
$\mathbf{k}\sfK_{\lambda} = \bm{0}$ in  $\bZ_m$. 
We have $\sfK_{\lambda}$ is invertible in $\mathbb Z_{m'}$ because $m$ and $2n$ are coprime and $\sfK_{\lambda}\sfH_{\lambda} = n I$.
Then we have $\mathbf{k} = \bm{0}$ in $\mathbb Z_{m'}$.
Similarly, we have $\Lambda_{m'}^{+} = (m' \mathbb N)^{V_{\lambda}'}$ if $m'$ and $n$ are coprime.
\end{rem}

\begin{lem}\label{balance}
    Suppose $\mathbf{k}\in\mathbb Z^{V_{\lambda}}$ is balanced, $e$ is a boundary component of $\Sigma$, and $w_1,\cdots,w_{n-1}$ are illustrated in Figure \ref{Fig;coord_uvw} (here $e$ is the attached edge). Then there exists $l\in\mathbb Z$ such that $(\mathbf{k}(w_1),\mathbf{k}(w_2),\cdots,\mathbf{k}(w_{n-1})) - l(1,2,\cdots,n-1) = \bm{0}$ in $\mathbb Z_n$.
\end{lem}
\begin{proof}
    We restrict our discussion to the attached triangle.
    Since $\mathbf{k}$ is balanced, then {\bf k}$=n\mathbf{b} + s_1\proj_1+ s_2\proj_2 + s_3\proj_3$, where $ \mathbf{b}\in\mathbb Z^{V_{\lambda}},s_1,s_2,s_3,\in\mathbb Z$ and $\proj_i\colon \barV \to\bZ\ (i=1,2,3)$ were defined as \eqref{def:proj}.
    For each $1\leq i\leq n-1$, in $\mathbb Z_n$, 
    \begin{align*}
        \mathbf{k}(w_i) &= s_1\proj_1(w_i)+ s_2\proj_2(w_i) + s_3\proj_3(w_i)\\&=
        s_1\proj_1(0,i,n-i)+ s_2\proj_2(0,i,n-i) + s_3\proj_3(0,i,n-i)\\
        &=s_2i+s_3(n-i) = (s_2-s_3)i.
    \end{align*}
    Thus we can set $l=s_2-s_3$.
\end{proof}
Note that, while there is ambiguity on the choice of $l$ in $\bZ$, $l$ is uniquely determined in $\bZ_n$.

Let $\mathsf{B}$ denote the set of the central elements in Lemma \ref{boundary_center}. 
\begin{thm}\label{center_torus}
Let $\Sigma$ be a triangulable essentially
bordered pb surface without interior punctures, and $\lambda$ be a triangulation of $\Sigma$.
If $m''$ is odd, the center of $\A$ is generated by $\tra(\mathsf{B})$ and $\{a^{\mathbf{k}}\mid \mathbf{k} \in \Lambda_{m'} \}$. 
\end{thm}
We will give a proof of Theorem \ref{center_torus} in Section \ref{sec;pf}.

\begin{cor}\label{coprime}
When $m''$ is odd  and $\gcd(m',n)=1$, the center of $\A$ is generated by $\tra(\mathsf{B})$ and $\im \cF$.
\end{cor}
\begin{proof}
From Theorem \ref{center_torus},
it suffices to show $\Lambda_{m'} = m \mathbb Z^{V_{\lambda}'}$. Since $m'$ and $2n$ are coprime, we have $m'=m$.
Then Remark \ref{rem_coprime} completes the proof. 
\end{proof}

When $\Sigma$ is an essentially bordered pb surface and contains no interior punctures, we know $\gaa_u\gaa_v = \hat q^{2\mathsf P_{\lambda}(u,v)}\gaa_v\gaa_u$ for $u,v\in V_{\lambda}'$ from Theorem~\ref{traceA}.
Like in \eqref{eq:Weyl}, for any ${\bf k}\in\mathbb Z^{V_{\lambda}'}$, define
\begin{align}\label{eq-Weyl-normalization-g}
\gaa^{\mathbf{k}} := [\prod_{v\in V_{\lambda}'}\gaa_v^{\mathbf{k}(v)}]_{\rm Weyl}.
\end{align}
Let $\cY_{\lambda}$ denote the subalgebra of $\cS_n(\Sigma,\mathbbm{v})$ generated by $\{\gaa^{\mathbf{k}}\mid \textbf k \in \Lambda_{m'}^{+}\}$ and $\mathsf{B}$.

\begin{thm}\label{thm-main-center-skein}
Let $\Sigma$ be a triangulable essentially
bordered pb surface without interior punctures, and $\lambda$ be a triangulation of $\Sigma$.
Assume that $m''$ is odd and $\hat{q}^2$ is a primitive $m''$-th root of unity. 

(a) We have 
$$\mathcal{Z}(\cS_n(\Sigma,\mathbbm{v})) = \{x\in \cS_n(\Sigma,\mathbbm{v})\mid \text{$\exists  \mathbf{k}\in \mathbb{N}^{V_{\lambda}'}  $  such that  $ \gaa^{m'\mathbf{k}}x\in\mathcal Y_{\lambda}$}\}.$$

(b) Let $\mathcal F_{\mathsf B}$ denote the subalgebra of $\cS_n(\Sigma,\mathbbm{v})$ generated by $\im\cF$ and $\mathsf B$.
If $m'$ and $n$ are coprime, we have 
$$\mathcal{Z}(\cS_n(\Sigma,\mathbbm{v})) = \{x\in \cS_n(\Sigma,\mathbbm{v})\mid \exists  y\in\im\cF  \text{  such that  } yx\in \mathcal F_{\mathsf B}  \}.$$

\end{thm}
\begin{proof}
(a)
    We have $\tra (\{x\in \cS_n(\Sigma,\mathbbm{v})\mid \text{$\exists  \mathbf{k}\in \mathbb{Z}^{V_{\lambda}'}  $  such that  $ \gaa^{m'\mathbf{k}}x\in\mathcal Y_{\lambda}$}\}) \subset \mathcal Z(\A)$. Thus $ \{x\in \cS_n(\Sigma,\mathbbm{v})\mid \text{$\exists  \mathbf{k}\in \mathbb{Z}^{V_{\lambda}'}  $  such that  $ \gaa^{m'\mathbf{k}}x\in\mathcal Y_{\lambda}$}\}\subset \mathcal{Z}(\cS_n(\Sigma,\mathbbm{v})) .$

From Theorem  \ref{traceA}, we can 
 identify $\cS_n(\Sigma,\mathbbm{v})$ and $\tra(\cS_n(\Sigma,\mathbbm{v}))$. 
 Then  have $\Ap\subset\Sn\subset\A$, which implies 
$$\mathcal{Z}(\cS_n(\Sigma,\mathbbm{v}))=\cS_n(\Sigma,\mathbbm{v})\cap \cZ(\A).$$ 
For $z\in \mathcal{Z}(\cS_n(\Sigma,\mathbbm{v}))$, 
Theorem \ref{center_torus} implies 
$z = \sum_{\mathbf{k}\in \Lambda} f_{\mathbf{k}} a^{\mathbf{k}}$, where $f_{\mathbf{k}}$ is a polynomial in $\mathsf{B}$ and $\Lambda$ is a finite subset of $\Lambda_{m'}$. There exists $\mathbf{k}_0\in \bN^{V_{\lambda}'}$ such that $\mathbf{k}+m'\mathbf{k}_0\in \bN^{V_{\lambda}'}$ for all $\mathbf{k}\in \Lambda$. 
We have 
$$a^{m'\mathbf{k}_0}z = \sum_{\mathbf{k}\in \Lambda} f_{\mathbf{k}} a^{\mathbf{k}+m'\mathbf{k}_0}\in \mathcal Y_{\lambda}.$$
From $\tra(\mathbf{g}^\mathbf{k})=a^\mathbf{k}$, 
we have $z\in \{x\in \cS_n(\Sigma,\mathbbm{v})\mid \text{$\exists  \mathbf{k}\in \mathbb{Z}^{V_{\lambda}'}  $  such that  $ \gaa^{m'\mathbf{k}}x\in\mathcal Y_{\lambda}$}\}$.

(b) As in (a), one can prove (b) using Corollary \ref{coprime}.
\end{proof}

\begin{conj}
Under the assumption of Theorem \ref{thm-main-center-skein} (b),
    we have $$\mathcal Z(\cS_n(\Sigma,\mathbbm{v})) = \mathcal F_{\mathsf B}.$$
\end{conj}

\subsection{Proof of Theorem \ref{center_torus}}\label{sec;pf}
\begin{proof}[Proof of Theorem \ref{center_torus}]
    Since $m''$ is odd, we have $\gcd(m'',2n) = \gcd(m'',n) = d'$.

    Let $\mathcal{A}_{m'}$ denote the subalgebra of $\A$  generated by $\tra(\mathsf{B})$ and $\{a^{\mathbf{k}}\mid \mathbf{k}\in \Lambda_{m'} \}$.
    Lemma \ref{boundary_center} implies $\mathcal{A}_{m'} $ is contained in the center of $\A$.

    Lemma \ref{quantum} implies $\mathcal Z(\A) = \cR\text{-span}\{a^{\tft}\mid  \langle {\bf t},{\bf t}'\rangle_{\sfP_{\lambda}}=0\in \mathbb{Z}_{m''},\forall {\bf t}'\in\mathbb{Z}^{V_{\lambda}'}  \}$.
    Suppose $\tft_0\in\mathbb Z^{V_{\lambda}'}$ such that $\tft \sfP_{\lambda} \tft_0^{T} = 0$ {in $\mathbb Z_{m''}$} for any $\tft\in \mathbb Z^{V_{\lambda}'}$, where we regard $\tft,\tft_0$ as row vectors. 
    Lemma~\ref{lem:invertible_KH} implies $\tft \sfK_{\lambda} \sfQ_{\lambda} (\tft_0 \sfK_{\lambda})^T = 0 $ in $\mathbb Z_{m''}$ for all $\tft\in \mathbb Z^{V_{\lambda}'}.$
    Set $\tfk _0= \tft_0 \sfK_{\lambda}$. Then we have $\sfK_{\lambda} \sfQ_{\lambda} \tfk _0^{T}= \bm{0}$ in $\mathbb Z_{m''}.$
    
    We regard $\mathbf{k}_0 = (
        \tfk_1,\tfk_2
    )\in \bZ^{V_{\lambda}}$ with $\tfk_1\in \bZ^{\obVlast}$ and $\tfk_2\in \bZ^{W}$.  From equation \eqref{KQ}, we have 
    \begin{equation}\label{eq_key}
    \begin{cases}
    -2n\bk_1^T+ (D+C_1A) \bk_2^T=\bm{0},\\
    (B-A)\bk_2^T=\bm{0},
    \end{cases}\text{ in $\mathbb Z_{m''}$}.
    \end{equation}

Suppose $\overline{\Sigma}$ has boundary components $\partial_1,\cdots,\partial_b$, and each boundary component $\partial _i$ contains $r_i$ punctures. Suppose $\tfk_2 =  (
\tfk_{\partial_1},\tfk_{\partial_2},\cdots,\tfk_{\partial_b}
)$ where $\tfk_{\partial_i}\in\mathbb Z^{r_i(n-1)}$ is the row vector associated to $\partial_i$ for each $1\leq i\leq b$.
From Lemmas \ref{matrixA} and \ref{matrixB}, we know $A = diag\{A_1,A_2,\cdots,A_b\}$ and $B = diag\{B_1,B_2,\cdots,B_b\}$ with $A_i=-n I$ and $B_i$ given as \eqref{Bi}. Then the second equality in equation \eqref{eq_key} implies $(B_i - A_i) \tfk_{\partial _i}^T = \bm{0}$ in $\mathbb Z_{m''}$ for each $1\leq i\leq b$. Obviously, the matrix $\frac{1}{d'}(B_i-A_i)$ is an integer matrix.
Then we have $\frac{1}{d'}(B_i-A_i) \tfk_{\partial_{i}}^T = \bm{0}$ in $\Zmp$ (not in $\mathbb Z_{m''}$). 

Suppose $r_i$ is odd. 
From Corollary \ref{invertibility}, we have $\det(\frac{1}{d'}(B_i-A_i)) = 2^{n-1}(n/d')^{r_i(n-1)}$. Since $2^{n-1}(n/d')^{r_i(n-1)}$ is coprime with $m'$, 
$\frac{1}{d'}(B_i-A_i)$ has its inverse given as $1/\det(\frac{1}{d'}(B_i-A_i))$ times the adjoint matrix of $\frac{1}{d'}(B_i-A_i)$. 
Hence, $\frac{1}{d'}(B_i-A_i) \tfk_{\partial_{i}}^T = 0$ in $\Zmp$ implies $\tfk_{\partial_{i}}=\bm{0}$ in $\Zmp$.

Suppose $r_i$ is even and $\tfk_{\partial_i}= (
\tfb_{i1},\tfb_{i2},\cdots,\tfb_{ir_i})$, where $\tfb_{ij}\in \bZ^{n-1}$ for any $j$. 
Then $(B_i - A_i) \tfk_{\partial _i}^T = \bm{0}$ in $\mathbb Z_{m''}$ implies 
\begin{equation}\label{eq's}
\begin{cases}
n\tfb_{ir_i} + n\tfb_{i1} =\bm{0},\\
n\tfb_{i1} + n\tfb_{i2} =\bm{0},\\
 n\tfb_{i2} + n\tfb_{i3} =\bm{0},\\
\;\vdots\\
n\tfb_{ir_i-1} + n\tfb_{ir_i} =\bm{0},\\
\end{cases}\text{ in $\mathbb Z_{m''}$}.
\end{equation}
Equation \eqref{eq's} implies 
$\tfb_{ij}  = (-1)^{j+1}\tfb_{i1}$ in $\mathbb Z_{m'}$.
Since $\mathbf{b}_{i1}$ is balanced, Lemma \ref{balance} implies there exists $l_i\in\mathbb Z$ such that 
\begin{align}
\mathbf{b}_{i1} - l_i(1,2,\cdots,n-1) = \bm{0}\text{\quad in } \mathbb Z_n.\label{eq:bi1}
\end{align}
Since $m''$ is odd, there exists $\mathbf{d}_i\in \mathbb Z^{n-1}$  such that $2\mathbf{d}_i E = (-l_i,0,\cdots,0)$ in $\mathbb Z_{m''}$, where $E$ was given as \eqref{matrixEF}.

Define $\mathbf{d}= (\mathbf{d}_{\partial_1},\mathbf{d}_{\partial_2},\cdots,\mathbf{d}_{\partial_b})\in \mathbb Z^{U}$, 
where $\mathbf{d}_{\partial_i}\in\mathbb Z^{r_i(n-1)}$ is the vector associated to $\partial_i$,  such that 
\begin{equation}\label{zero}
    \mathbf{d}_{\partial_i} = 
        \begin{cases}
        \bm{0}& \text{if $r_i$ is odd},\\
        (\mathbf{d}_i,-\mathbf{d}_i,\cdots,\mathbf{d}_i,-\mathbf{d}_i) & \text{if $r_i$ is even}.\\
        \end{cases}
    \end{equation}
    Set $\mathbf{d}' := (\mathbf{0},\mathbf{d})\in\mathbb Z^{V_{\lambda}'}$ and $\mathbf{f} := \mathbf{d}'\sfK_{\lambda}$.
    From the definition of $\mathbf{d}$, we know $\sfK_{\lambda}\sfQ_{\lambda}\mathbf{f}^{T} = \bm{0}$ in $\mathbb Z_{m''}$.
    We regard $\mathbf{f} = (
        \mathbf{f}_1,\mathbf{f}_2
    )\in \mathbb Z^{V_{\lambda}}$ with $\mathbf{f}_1\in \bZ^{\obVlast}$ and $\mathbf{f}_2\in \bZ^{W}$. 
    By replacing $\mathbf{k}_i$ with $\mathbf{f}_i$, $\mathbf{f}_1$ and $\mathbf{f}_2$ satisfy equation \eqref{eq_key}, i.e., $\mathbf{f}$ is a concrete solution. From \eqref{eq_K} and Lemma \ref{matrixK}, we have 
    $\mathbf{f}_2 = \mathbf{d}(K_{32}-K_{22}) = (\mathbf{d}_{\partial_1}L_1,\mathbf{d}_{\partial_2}L_2,\cdots,\mathbf{d}_{\partial_b}L_b)$.

    When $r_i$ is odd, we have $\mathbf{d}_{\partial_i}L_i = \bm{0}$.

    When $r_i$ is even, Lemma \ref{matrixK} implies $\mathbf{d}_{\partial_i}L_i = (-2\mathbf{d}_{i}G,2\mathbf{d}_{i}G,\cdots,-2\mathbf{d}_{i}G,2\mathbf{d}_{i}G)$. Then we have 
    \begin{align}\label{proof-center-mdouble}
        -2\mathbf{d}_{i}G = -2\mathbf{d}_{i}EF=(l_i,0,\cdots,0)F
         = l_i(1,2,\cdots,n-1) \text{\quad in $\mathbb Z_{m''}$}.
    \end{align}
    By combining with \eqref{eq:bi1}, we have $-2\mathbf{d}_{i}G = \mathbf{b}_{i1}$ in $\mathbb Z_{d'}.$ This implies $\mathbf{d}_{\partial_i}L_i =\tfk_{\partial_i}$
    in $\mathbb Z_{d'}$, especially $\mathbf{f}_2=\tfk_2$ in 
    $\mathbb Z_{d'}$.

Set $\mathbf{h} := \mathbf{f} - \mathbf{k}_0$. Then $\mathbf{h}  =
(\mathbf{h} _1,\mathbf{h} _2)$, where $\mathbf{h}_1 = \mathbf{f}_1 - \mathbf{k}_1$ and $\mathbf{h}_2 = \mathbf{f}_2 - \mathbf{k}_2$. Note that $\mathbf{h}_1$ and $\mathbf{h}_2$ satisfy equation \eqref{eq_key}.
We regard $\mathbf{h}_2 =  (\tff_{\partial_1},\tff_{\partial_2},\cdots,\tff_{\partial_b})$ where $\tff_{\partial_i}\in\mathbb Z^{r_i(n-1)}$ is the row vector associated to $\partial_i$ for each $1\leq i\leq b$.
From the above discussion, we know $\tff_{\partial_i} = \bm{0}$ in $\mathbb Z_{m'}$ if $r_i$ is odd. Also, we have 
$\tff_{\partial_i} = (\tff'_{i},-\tff'_{i},\cdots,\tff'_{i},-\tff'_{i})$ in $\mathbb Z_{m'}$ if $r_i$ is even, where 
$\tff'_{i} =(h_{i,1},h_{i,1},\cdots,h_{i,n-1}) \in\mathbb Z^{n-1}$.
Since $\tff_2 =\bm{0}$ in $\mathbb Z_{d'}$, each $h_{i,j}$ is a multiple of $d'$ for $1\leq j \leq n-1$. The equation 
$(x_1,x_2,\cdots,x_{n-1})F= \tff'_{i}$ in $\mathbb Z_{m''}$ implies
\begin{equation}\label{equation}
\begin{cases}
(n-1)x_1 - nx_2 =h_{i,1},\\
(n-2)x_1 - nx_3 =h_{i,2},\\
\;\vdots\\
2x_1 - nx_{n-2} =h_{i,n-2},\\
x_1 = h_{i,n-1},\\
\end{cases}\text{ in $\mathbb Z_{m''}$}.
\end{equation}
Equation \eqref{equation} is equivalent to the following equations:
\begin{equation}\label{equation2}
\begin{cases}
\dfrac{n}{d'} x_2 =(n-1)\dfrac{h_{i,n-1}}{d'}-\dfrac{h_{i,1}}{d'} \text{ in $\mathbb Z_{m'}$},\medskip\\
\dfrac{n}{d'}x_3 =(n-2)\dfrac{h_{i,n-1}}{d'}-\dfrac{h_{i,2}}{d'}\text{ in $\mathbb Z_{m'}$},\\
\;\vdots\\
\dfrac{n}{d'}x_{n-2} =2\dfrac{h_{i,n-1}}{d'}-\dfrac{h_{i,n-2}}{d'}\text{ in $\mathbb Z_{m'}$},\\
x_1 = h_{i,n-1}\text{ in $\mathbb Z_{m''}$},\\
\end{cases}
\end{equation}
Since $n/d'$ and $m'$ are coprime, Equation \eqref{equation2} has a solution, and so does Equation \eqref{equation}.
Fix a solution $\mathbf{x}_i'\in \mathbb Z^{n-1}$ of Equation \eqref{equation}. Since $m''$ is odd, there exists 
$\mathbf{x}_i''\in \mathbb Z^{n-1}$ such that $2\mathbf{x}_i ''E=\mathbf{x}_i'$ in $\mathbb Z_{m''}$. Thus we have 
$2\mathbf{x}_i''G = 2\mathbf{x}_i''EF=\mathbf{x}_i'F=\tff_i'$ in $\mathbb Z_{m''}$.

Define $\mathbf{x}_2= (\mathbf{x}_{\partial_1},\mathbf{x}_{\partial_2},\cdots,\mathbf{x}_{\partial_b})\in \mathbb Z^{U}$, where $\mathbf{x}_{\partial_i}\in\mathbb Z^{r_i(n-1)}$ is the vector associated to $\partial_i$,  such that 
\begin{equation}
\mathbf{x}_{\partial_i} = 
\begin{cases}
\bm{0}& r_i \text{ is odd},\\
(-\mathbf{x}_i'',\mathbf{x}_i'',\cdots,-\mathbf{x}_i'',\mathbf{x}_i'') & r_i \text{ is even}.\\
\end{cases}
\end{equation}
Define $\mathbf{x}\in\mathbb Z^{V_{\lambda}'}$ with $\mathbf{x} = (0,\mathbf{x}_2)$, and define $\mathbf{y} = \mathbf{x}\sfK_{\lambda}$.

We regard $\mathbf{y} = (\mathbf{y}_1,\mathbf{y}_2)\in \mathbb Z^{V_{\lambda}}$ with $\mathbf{y}_1\in \bZ^{\obVlast}$ and $\mathbf{y}_2\in \bZ^{W}$.
By the definition of $\mathbf{x}$, $\mathbf{y}_1$ and $\mathbf{y}_2$ satisfy Equation \eqref{eq_key}, and $\mathbf{y}_2-\mathbf{h}_2 =\bm{0}$ in $\mathbb Z_{m'}$. We have $\mathbf{y} - \mathbf{h} = (\mathbf{y}_1 - \mathbf{h}_1,\mathbf{y}_2 - \mathbf{h}_2)$ satisfies Equation \eqref{eq_key}, especially $-2n(\mathbf{y}_1 - \mathbf{h}_1)^T+ (D+C_1A) (\mathbf{y}_2 - \mathbf{h}_2)^T=\bm{0}\text{ in }\mathbb Z_{m''}$.
Lemma \ref{matrixDC} and $\mathbf{y}_2-\mathbf{h}_2 =\bm{0}$ in $\mathbb Z_{m'}$ imply 
$-2n(\mathbf{y}_1 - \mathbf{h}_1)^T=\bm{0}\text{ in }\mathbb Z_{m''}$.
Then we have $\mathbf{y}_1 - \mathbf{h}_1 = \bm{0}$ in $\mathbb Z_{m'}$. 
This implies $\mathbf{y} - \mathbf{h} =\bm{0}$ in $\mathbb Z_{m'}$.
Since both of $\mathbf{y}$ and $\mathbf{h}$ are balanced, then 
$\mathbf{y} - \mathbf{h} = \mathbf{z}\sfK_{\lambda}$ for some 
$\mathbf{z}\in \mathbb Z^{V_{\lambda}'}$. We have $\mathbf{z}\in\Lambda_{m'}$ since $\mathbf{y} - \mathbf{h} = \bm{0}$ in $\mathbb Z_{m'}$.

We have $\mathbf{h} = \mathbf{f}-\mathbf{k}_0 = \mathbf{d}'\sfK_{\lambda}- \mathbf{t}_0\sfK_{\lambda}$, and 
$\mathbf{h} = \mathbf{y}-\mathbf{z}\sfK_{\lambda} = \mathbf{x}\sfK_{\lambda}- \mathbf{z}\sfK_{\lambda}$. We have $\mathbf{t}_0=
\mathbf{d}'-\mathbf{x}+\mathbf{z}$.
Then $a^{\mathbf{t}_0} = a^{\mathbf{d}'}a^{-\mathbf{x}}a^{\mathbf{z}}\in\mathcal A_{m'}$ because of the definitions of $\mathbf{d}',\mathbf{x}$ and $\mathbf{z}\in \Lambda_{m'}$.
\end{proof}

\section{The Unicity Theorem and the PI-degree of the stated $\SL(n)$-skein algebra}\label{sec-Unicity-Theorem}
In this section, we suppose that $\cR=\bC$ and Condition~$(\ast)$ as Section~\ref{sub_center} when we consider root of unity case.
In the rest of the paper, all the representations are assumed to be finite-dimensional.
We will prove the Unicity Theorem for stated $\SL(n)$-skein algebras. This classifies the irreducible representations of the stated $\SL(n)$-skein algebra. This extends the works \cite{FKBL19,KW24} for $n=2,3$ and \cite{Wan23} for general $n$ and a special root of unity. 
We will also formulate the PI-degree of the stated $\SL(n)$-skein algebra when the essentially bordered pb surface has no interior punctures. This PI-degree equals the 'maximal' dimension of 
irreducible representations. 
Irreducible representations whose dimensions are the PI-degree are called the Azumaya representations.  
See \cite{FKBL21,Yu23,KW24} for related works for $n=2,3$.

\subsection{Almost Azumaya algebra and Unicity theorem}\label{sub-Almost-Azumaya-algebra}
While the following is a sufficient condition to be almost Azumaya, we will use the following definition to avoid the precise definition of almost Azumaya. See e.g. \cite{FKBL19} for details.
\begin{dfn}\label{def:almost_Azumaya}
A $\bC$-algebra $A$ is \textbf{almost Azumaya} if $A$ satisfies the following conditions; 
\begin{enumerate}
    \item $A$ is finitely generated as a $\bC$-algebra, 
    \item $A$ has no zero-divisors, 
    \item $A$ is finitely generated as a module over its center. 
\end{enumerate}
\end{dfn}
Note that the original second condition is that $A$ is prime, which is a weaker condition than $A$ has no zero-divisors. In the paper, we use the different definition suitable for skein algebras.

Suppose $A$ is  almost Azumaya. We use $\text{Frac}( {\mathcal Z(A)})$ to denote the fractional field of $\mathcal Z(A)$, and use 
$\widetilde A$ to denote $A\otimes_{\cZ(A)} \Frac(\cZ(A))$.
Define $\rankZ A$ to be the dimension of $\widetilde{A}$ over $\Frac(\cZ(A))$.

\begin{rem}
If $A$ is free over its center $\cZ(A)$ then $\rankZ A$ is equal to the usual rank. 
\end{rem}

Let $V(A)$ denote the maximal spectrum $\Specm(\cZ(A))$. 
Wedderburn’s theorem implies that $\widetilde{A}$ is a division algebra, and has dimension $D^2$ over $\Frac(\cZ(A))$ with a positive integer $D$, called the \textbf{PI-degree} of $A$.
A point $\mathfrak{m} \in  V(A)$ is \textbf{Azumaya} if $A/\mathfrak{m}A \cong M_D(\bC)$, the algebra of $D \times D$-matrices over $\bC$. Let $\Azumaya(A)\subset V(A)$ denote the subset of all the Azumaya points, called the \textbf{Azumaya locus}. 
It is well-known that $D^2 = \rankZ A$ \cite{FKBL19}.

Let $\Irrep(A)$ denote the set of conjugacy classes of finite dimensional irreducible representations of $A$ and ${\rm Hom}_{\text{\rm $\bC$-alg}}(\cZ(A), \bC)$ denote the set of $\bC$-algebra homomorphisms from $\cZ(A)$ to $\bC$.  

\def\csh{{\rm csh}}
For a non-zero irreducible representation $\rho\colon A \to {\rm End}(V)$, Schur’s lemma implies that, for any $z \in  \cZ(A)$, there is a scalar $\chi_\rho(z) \in \bC$ such that $\rho(z) = \chi_\rho(z) \Id$.
Since $\chi_\rho\colon \cZ(A) \to \bC$ is a surjective $\bC$-algebra homomorphism, the kernel $\kernel(\chi_\rho(z))$ is a maximal ideal. 
By regarding ${\rm Hom}_{\text{\rm $\bC$-alg}}(\cZ(A), \bC)$ and $V(A)$, $\chi_\rho$ is an element of $V(A)$, called the \textbf{central character} (or \textbf{classical shadow}) of $\rho$. 

\begin{thm}[Unicity theorem \cite{FKBL19, BG02}]\label{thm:Unicity-alomost}
Let $A$ be an almost Azumaya $\bC$-algebra. 
\begin{enumerate}
    \item The central character map 
    $$\chi\colon \Irrep(A) \to {\rm Hom}_{\text{\rm $\bC$-alg}}(\cZ(A), \bC)\equiv V(A),\quad [\rho] \mapsto \chi_\rho$$ is surjective. 
    \item The Azumaya locus $\Azumaya(A)$ is a Zariski dense open subset of $V(A)$ and satisfies, for any $\tau \in \Azumaya (A)$, the preimage $\chi^{-1}(\tau)$ consists of exactly one element. In other words, $\tau$ is the central character of a unique (up to equivalence) irreducible representation $\rho_\tau$. In particular, for all $\tau \in \Azumaya (A)$, the dimension of $\rho_\tau$ equals the PI-degree $D$ of $A$. On the other hand, for $\tau \notin \Azumaya (A)$, the central character $\tau$ corresponds to at most $r$ non-equivalent irreducible representations, and the dimension of each is less than $D$, where $r$ depends only on $\Sigma$ and $\hat{q}$.
\end{enumerate}
\end{thm}
For $\tau \in \Azumaya (A)$, the unique element in $\chi^{-1}(\tau)$
(or each of its representatives) is called an \textbf{Azumaya representation} of $A$ associated to $\tau$.

\begin{prop}\label{thm;azumaya}
When $\mathbbm{v}$ is a root of unity and $\Sigma$ is an essentially bordered pb surface, $\cS_n(\Sigma,\mathbbm{v})$ is  almost Azumaya. 
\end{prop}
Note that Theorem 6.1 in \cite{LY23} shows that $\cS_n(\Sigma,\mathbbm{v})$ satisfies conditions (1) and (2) in Definition \ref{def:almost_Azumaya}. 
The remaining part is to show (3). 
Since the proving technique of Proposition 8.10 in \cite{Wan23} works here, we do not give a full proof. We will give an outline of a proof of Proposition \ref{thm;azumaya} in Appendix \ref{appendix}.

\subsection{Dimension of quantum $A$-torus over its center}
In the rest of this section, we suppose that $\Sigma$ is a connected triangulable essentially bordered pb surface without interior punctures and $\lambda$ is a triangulation of $\Sigma$. Suppose $\overline{\Sigma}$ has boundary components $\partial_1,\cdots,\partial_b$, and each $\partial_i$ contains $r_i$ punctures.

We will define a subset $\Lambda_{\partial}\subset \bZ^{V_{\lambda}'}$.
Define $\Lp=\emptyset$ if all $r_i$ ($1\leq i\leq b$) are odd.

Suppose at least one of $r_i$ is even. 
For each $i$ such that $r_i$ is even, we label the boundary components of $\Sigma$ contained in $\partial_i$ consecutively from
$1$ to $r_i$ following the positive orientation of $\partial_i$. For $1\leq j\leq n-1$, define $\mathbf{k}_{j,\partial_i}\in \mathbb Z^{V_{\lambda}'}$ such that 
$$\mathbf{k}_{j,\partial_i}(v) =\begin{cases}
        (-1)^{k-1}& \text{if } v= u_j^{e_k}\text{ for }1\leq k\leq r_i, \\
        0 &  \text{ otherwise}.\\
        \end{cases}$$
Then, define 
\begin{align}\label{eq-balanced-boundary-A-version}
    \Lambda_{\partial} =\{\mathbf{k}_{j,\partial_i}\mid r_i \text{ is even},\; 1\leq j\leq n-1\}.
\end{align}

Let $\Lambda_z$ denote the subgroup of $\mathbb Z^{V_{\lambda}'}$ generated by  $\Lambda_{\partial}$ and $\Lambda_{m'}$ defined as \eqref{eq:Lambda}.

\begin{lem}\label{lem5.1}
Suppose $m''$ is odd. \\
(a) The center of $\A$ is $\mathbb C\text{-span}\{a^{\mathbf{k}}\mid \mathbf{k}\in \Lambda_z\}$.\\
(b) We have $\rankZ\A=\left|\dfrac{\mathbb Z^{V_{\lambda}'}}{\Lambda_z}\right|$. 
\end{lem}
\begin{proof}
(a) follows from Theorem \ref{center_torus}.
(b) follows from (a) and Lemma \ref{PI}.
\end{proof}

The following proposition is proved in \cite{FKBL21,Wan23a,Yu23} when $n=2$. We will prove the proposition using techniques in \cite{Wan23a}.

\begin{prop}\label{rank_eq}
We have $\rankZ\A = \rankZ\cS_n(\Sigma,\mathbbm{v})$.
\end{prop}
\begin{proof}
From Theorem \ref{traceA} (a), we can identify $\Sn$ and $\tra(\Sn)$. 
Then we  have $\Ap\subset\Sn\subset\A,$ which implies 
$$\mathcal Z(\Ap)\subset\mathcal Z(\Sn)\subset\mathcal Z(\A).$$
To show the claim, it suffices to show $\tZS=\tZA$ and $\tSn =\tA$, 
where $\widetilde A=A\otimes_{\mathcal{Z}(A)} \Frac({\mathcal Z(A)})$.

For any ${\bf k}\in\Lambda_z$, there exists ${\bf t}\in m'\mathbb N^{V_{\lambda}'}\subset\Lambda_{m'}$ such that 
${\bf t}+{\bf k}\in \mathbb N^{V_{\lambda}'}.$
We have $a^{{\bf t}}, a^{{\bf t}+{\bf k}}\in \mathcal Z(\Ap)\subset\mathcal Z(\Sn)$.
Then $a^{{\bf k}}=(a^{{\bf t}})^{-1}a^{{\bf t}+{\bf k}}\in\tZS$.
Lemma \ref{lem5.1} implies $\ZA\subset\tZS$. Since $\tZS$ is a field, then $\tZA\subset\tZS$. Obviously, we also have $\tZS\subset\tZA$.
This shows $\tZA=\tZS$.

From \cite[Proposition 2.1 (a)]{FKBL21}, both of $\tSn$ and $\tA$ are division algebras. From this fact, $\Ap\subset \tSn$ implies $\A\subset \tSn$,  
and $\A\subset \tSn$ implies $\tA\subset \tSn$.
Obviously, we also have $\tSn\subset\tA$.
This shows $\tA= \tSn$.
\end{proof}

For any abelian group $G$,
we take the cohomology of a CW complex of $\overline\Sigma$ over $G$
$$0\rightarrow C^0(\overline \Sigma,G)\xrightarrow{\delta_0} C^1(\overline \Sigma,G)\xrightarrow{\delta_1} C^2(\overline \Sigma,G)\rightarrow 0.$$
Let $Z^{i}(\overline\Sigma,G),\ B^{i}(\overline\Sigma,G)$, and $H^{i}(\overline\Sigma,G)$ denote the $i$-th cocycle, the $i$-th coboundary, and the $i$-th cohomology group of $\overline\Sigma$ with coefficient $G$ respectively.

\begin{lem}\label{lem5.3}
Suppose $k$ is a positive integer and $\gcd(k,n) = l$. Set $N = kn/l$.
   Then there is a short exact sequence 
$$0\rightarrow N\mathbb Z^{V_{\lambda}}\xrightarrow{L}\Lambda_{\lambda}\cap k\mathbb Z^{V_{\lambda}} \xrightarrow{J} Z^1(\overline \Sigma,\mathbb Z_n)_l \rightarrow 0,$$
where $L$ is the natural embedding and $Z^1(\overline \Sigma,\mathbb Z_n)_l=l(C^1(\overline\Sigma,\mathbb Z_n))\cap Z^1(\overline \Sigma,\mathbb Z_n)$ and $\Lambda_\lambda$ is the balanced part.
\end{lem}
We will prove Lemma \ref{lem5.3} in Section \ref{sec;pf_lem}

Recall $d=\gcd(m',n)$ defined in Section~\ref{notation}. 
\begin{prop}\label{prop5.4}
    We have $\left|\displaystyle \frac{\Lambda_{\lambda}}{\Lambda_{\lambda}\cap m'\mathbb Z^{V_{\lambda}}}\right|=m^{|V_{\lambda}|} \left|\frac{Z^1(\overline \Sigma,\mathbb Z_n)}{Z^{1}(\overline \Sigma,\mathbb Z_n)_d}\right|$.
\end{prop}
\begin{proof}Set $N' = m'n/d$.
Under $k=m'$, Lemma \ref{lem5.3} implies
$$\left|\frac{\Lambda_{\lambda}\cap m'\mathbb Z^{V_{\lambda}}}{n\mathbb Z^{V_{\lambda}}\cap m'\mathbb Z^{V_{\lambda}}}\right| = \left|\frac{\Lambda_{\lambda}\cap m'\mathbb Z^{V_{\lambda}}}{ N'\mathbb Z^{V_{\lambda}}}\right| = | Z^1(\overline \Sigma,\mathbb Z_n)_d|,$$ and 
$$\left|\frac{\Lambda_{\lambda}}{n\mathbb Z^{V_{\lambda}}\cap m'\mathbb Z^{V_{\lambda}}}\right| = \left|\frac{\Lambda_{\lambda}}{n\mathbb Z^{V_{\lambda}}}\right|
\left|\frac{n\mathbb Z^{V_{\lambda}}}{N'\mathbb Z^{V_{\lambda}}}\right| = |Z^1(\overline\Sigma,\mathbb Z_n)|
\Big(\frac{m'}{d}\Big)^{|V_{\lambda}|}.$$
Note that 
\begin{align}\label{eq_pi}
    \left|\frac{\Lambda_{\lambda}}{n\mathbb Z^{V_{\lambda}}\cap m'\mathbb Z^{V_{\lambda}}}\right|=\left|\frac{\Lambda_{\lambda}}{\Lambda_{\lambda}\cap m'\mathbb Z^{V_{\lambda}}}\right| \left|\frac{\Lambda_{\lambda}\cap m'\mathbb Z^{V_{\lambda}}}{n\mathbb Z^{V_{\lambda}}\cap m'\mathbb Z^{V_{\lambda}}}\right|.
\end{align}
Then Equation \eqref{eq_pi} completes the proof. 
\end{proof}

Obviously, we have the following short exact sequence
\begin{align}\label{short}
    0\rightarrow d\mathbb Z_n\xrightarrow{f} \mathbb Z_n\xrightarrow{g} \mathbb Z_d\rightarrow 0,
\end{align}
where $f$ is the embedding and $g$ is the projection.

\begin{lem}\label{lem_short}
We have the following short exact sequence
    \begin{align*}
    0\rightarrow Z^1(\overline\Sigma,\mathbb Z_n)_d\xrightarrow{L} Z^1(\overline\Sigma,\mathbb Z_n)\xrightarrow{g_*} Z^1(\overline\Sigma,\mathbb Z_d)\rightarrow 0,
\end{align*}
where $L$ is the embedding and $g_*$ is induced by $g$.
\end{lem}
\begin{proof}
Obviously, we have $g_\ast \circ L=0$. For any $c\in\kernel{g_*}$, we can regard $c$ as a map from the set of all the 1-simplices to $\mathbb Z_n$. Then $g_\ast(c) = g\circ c = 0\in Z^1(\overline\Sigma,\mathbb Z_d)$, i.e., 
$(g\circ c)(e) = 0\in\mathbb Z_d$ for each 1-simplex $e$. Then $c(e)\in d\mathbb Z_n$. This shows $c\in d C^1(\overline\Sigma,\mathbb Z_n)$, and $c\in\im L$.  

The remaining part is to show the surjectivity of $g_\ast$. Consider the following commutative diagram. 
    \begin{equation}
    \begin{tikzcd}
0  \arrow[r, " "] &  B^1(\overline\Sigma,\mathbb Z_n)  \arrow[r, " "]
\arrow[d, "g^b_{\ast}"] 
&  Z^1(\overline\Sigma,\mathbb Z_n)  \arrow[r, " "]
\arrow[d, "g^z_{\ast}"] 
&  H^1(\overline\Sigma,\mathbb Z_n)  \arrow[r, " "]
\arrow[d, "g^h_{\ast}"]& 0\\
  0\arrow[r, " "]  &  B^1(\overline\Sigma,\mathbb Z_d)  \arrow[r, " "]
&  Z^1(\overline\Sigma,\mathbb Z_d)  \arrow[r, " "]
&  H^1(\overline\Sigma,\mathbb Z_d)  \arrow[r, " "]
&  0\\
\end{tikzcd}
\end{equation}
Here the two rows are short exact sequences and $g^b_\ast,g^z_\ast,g^h_\ast$ are induced by $g$. 
It is easy to show $g^b_\ast\colon B^1(\overline\Sigma,\mathbb{Z}_n)\rightarrow  B^1(\overline\Sigma,\mathbb{Z}_d)$ is surjective, i.e., $\coker{g^b_\ast}=0$. Then the snake lemma implies 
\begin{align}
\frac{Z^1(\overline\Sigma,\mathbb Z_d)}{\im g_\ast}= \coker{g^z_\ast}\cong \coker{g^h_\ast}=\frac{H^1(\overline\Sigma,\mathbb Z_d)}{\im g_\ast}.\label{eq:coker}\end{align}

Since the free abelian group is a projective module over $\mathbb Z$,
the short exact sequence in \eqref{short} induces the following short exact sequence for cochain complexes
\begin{align*}
    0\rightarrow C^{\ast}(\overline\Sigma,d \mathbb{Z}_n)\xrightarrow{f_{\ast}} C^{\ast}(\overline\Sigma,\mathbb Z_n)\xrightarrow{g_\ast} C^{\ast}(\overline\Sigma,\mathbb Z_d)\rightarrow 0.
\end{align*}
The above short exact sequence induces the following exact sequence
$$H^1(\overline\Sigma,\mathbb Z_n)\xrightarrow{g_\ast} H^1(\overline\Sigma,\mathbb Z_d)\rightarrow H^2(\overline\Sigma, d\mathbb Z_n).$$
According to Poincar\'e–Lefschetz duality, we have 
$H^2(\overline\Sigma, d\mathbb Z_n) \simeq H_0(\overline\Sigma,\partial \overline\Sigma; d\mathbb Z_n) = 0$, where $H_0(\overline\Sigma,\partial \overline\Sigma; d\mathbb Z_n)$ is the relative homology group.
Since $g^h_\ast\colon H^1(\overline\Sigma,\mathbb Z_n)\rightarrow H^1(\overline\Sigma,\mathbb Z_d)$ is surjective, so is $g^z_\ast\colon Z^1(\overline\Sigma,\mathbb Z_n)\rightarrow Z^1(\overline\Sigma,\mathbb Z_d)$ from \eqref{eq:coker}.
\end{proof}

\begin{rem}
    Lemmas \ref{lem5.3}, \ref{lem_short}, and Proposition \ref{prop5.4} work for all essentially bordered pb surfaces since there always exists a triangulation such that it has no self-folded triangle. 
\end{rem}

\begin{lem}\label{lem5.8}
    We have $$\displaystyle \left|\frac{\Lambda_{\lambda}}{\Lambda_{\lambda}\cap m'\mathbb Z^{V_{\lambda}}}\right| =d^{r(\Sigma)}m^{|V_{\lambda}|}.$$
\end{lem}
\begin{proof}
    From Proposition \ref{prop5.4}, and Lemma \ref{lem_short}, we have
$$\left|\displaystyle \frac{\Lambda_{\lambda}}{\Lambda_{\lambda}\cap m'\mathbb Z^{V_{\lambda}}}\right| = m^{|V_{\lambda}|}\left|\frac{Z^1(\overline \Sigma,\mathbb Z_n)}{Z^{1}(\overline\Sigma,\mathbb Z_n)_d}\right|=m^{|V_{\lambda}|}|Z^1(\overline\Sigma,\mathbb Z_d)|.$$
From a similar formula to Equation (72)  in \cite{Yu23}, we have $|Z^1(\overline\Sigma,\mathbb Z_d)|= d^{r(\Sigma)}$.
\end{proof}

\begin{rem}\label{rem_partial}
Proposition~\ref{prop:LY23_11.10} implies there is a group isomorphism $\varphi\colon\mathbb Z^{V_{\lambda}'}\rightarrow \Lambda_{\lambda}$, defined by $\varphi(\textbf{k}) = \textbf{k}\sfK_{\lambda}$ for $\textbf{k}\in \mathbb Z^{V_{\lambda}'}$.
Then $\left|\dfrac{\mathbb Z^{V_{\lambda}'}}{\Lambda_z}\right| =
\left|\dfrac{\Lambda_{\lambda}}{\varphi(\Lambda_z)}\right|,$
where $\Lambda_z = \Lambda_{m'}+\langle\Lambda_{\partial}\rangle$ and $\langle\Lambda_{\partial}\rangle$ is the subgroup of $\mathbb Z^{V_{\lambda}'}$ generated by $\Lambda_{\partial}$ (see \eqref{eq-balanced-boundary-A-version}). From the definition of $\Lambda_{m'}$, we have $\varphi(\Lambda_z) = (\Lambda_{\lambda}\cap m'\mathbb Z^{V_{\lambda}})+\varphi(\langle\Lambda_{\partial}\rangle)$. 

Let us look at what $\varphi(\langle\Lambda_{\partial}\rangle)$ is. If there is no even boundary component, we have $\varphi(\langle\Lambda_{\partial}\rangle)= 0$. Suppose $\partial_1,\cdots,\partial_t$ are even boundary components. For any element $\textbf{k}\in \langle\Lambda_{\partial}\rangle$,
we write $\mathbf{k} = (\tfk_1,\tfk_2)\in \bZ^{V_{\lambda}}$ with $\tfk_1\in \bZ^{\obVlast}$ and $\tfk_2\in \bZ^{U}$. 
From the definition of $\Lambda_{\partial}$, we have $\tfk_1=\bm{0}$ and
$\textbf{k}_2|_{\partial} =\bm{0}$ if $\partial$ is a boundary component with odd number of boundary punctures. For $1\leq i\leq t$, we have $\textbf{d}_i:=\textbf{k}_2|_{\partial_i} = (-\alpha_i,\alpha_i,\cdots,-\alpha_i,\alpha_i)$,  
where $\alpha_i\in\bZ^{n-1}$. Suppose $\varphi(\tfk) = (\textbf{u}_1,\textbf{u}_2)$, where $\textbf{u}_1\in \bZ^{\obVlast}$ and $\textbf{u}_2\in \bZ^{W}$. Then 
$(\textbf{u}_1,\textbf{u}_2)$ satisfies Equation \eqref{eq_key} since $a^{\tfk}$ is in the center of $\A$. See Theorem \ref{lem5.1}. 
Equation \eqref{eq_K} and Lemma \ref{matrixK} imply 
$\textbf{u}_2 = (\textbf{d}_1L_1,\cdots,\textbf{d}_tL_t,0,\cdots,0)$.
From Equation \eqref{eq_L}, we have $\textbf{d}_iL_i= (2\alpha_i G,-2\alpha_i G,\cdots,-2\alpha_i G,2\alpha_i G)$. 
\end{rem}

\begin{lem}\label{lem5.10}
Suppose $\overline\Sigma$ contains $t$ boundary components with even number of boundary punctures. Then we have 
$$\left|\dfrac{\varphi(\Lambda_z)}{\Lambda_{\lambda}\cap m'\mathbb Z^{V_{\lambda}}}\right|=\left| \frac{(\Lambda_{\lambda}\cap m'\mathbb Z^{V_{\lambda}})+\varphi(\langle\Lambda_{\partial}\rangle)}{\Lambda_{\lambda}\cap m'\mathbb Z^{V_{\lambda}}}\right|=(m')^tm^{t(n-2)}.$$
\end{lem}
\begin{proof}
    If $t=0$, it is obvious. Suppose $t>0$,  these even boundary components are $\partial_1,\partial_2,\cdots,\partial_t$. 

Obviously, we have 
$$\left| \frac{(\Lambda_{\lambda}\cap m'\mathbb Z^{V_{\lambda}})+\varphi(\langle\Lambda_{\partial}\rangle)}{\Lambda_{\lambda}\cap m'\mathbb Z^{V_{\lambda}}}\right| =
\left| \{x+m' \mathbb Z^{V_{\lambda}}\mid x\in \varphi(\langle\Lambda_{\partial}\rangle)\}\right|,$$
where $\{x+m' \mathbb Z^{V_{\lambda}}\mid x\in \varphi(\langle\Lambda_{\partial}\rangle)\}$ is a subset of $\frac{\mathbb Z^{V_{\lambda}}}{m'\mathbb Z^{V_{\lambda}}}$ (it is actually a subgroup of $\frac{\mathbb Z^{V_{\lambda}}}{m'\mathbb Z^{V_{\lambda}}}$). 

In the rest of the proof, we follow the notations in Remark \ref{rem_partial}.
Define $$\mu\colon\mathbb Z_{m'}^{n-1}\rightarrow \mathbb Z_{m'}^{n-1},\;\mu(\textbf{p}) = 2\textbf{p} G.$$ For any element $\textbf{u}=(\textbf{u}_1,\textbf{u}_2)\in \varphi(\langle\Lambda_{\partial}\rangle)$, we have $\textbf{u}_2 = (\textbf{d}_1L_1,\cdots,\textbf{d}_tL_t,0,\cdots,0)$, where $\textbf{d}_iL_i= (2\alpha_i G,-2\alpha_i G,\cdots,-2\alpha_i G,2\alpha_i G)$.
Define $\theta(\textbf{u}) = (2\alpha_1 G,\cdots,2\alpha_t G)\in (\im\mu)^t$. It is easy to see $\theta$ is a well-defined surjective group homomorphism from $\{x+m' \mathbb Z^{V_{\lambda}}\mid x\in \varphi(\langle\Lambda_{\partial}\rangle)\}$ to $(\im\mu)^t$. Suppose $\theta(\textbf{u}) = 0$. Then $2\alpha_i G=\bm{0}$ in $\mathbb Z_{m'}$. Thus we have $\textbf{u}_2=\bm{0}$ in $\mathbb Z_{m'}$. Since $(\textbf{u}_1,\textbf{u}_2)$ satisfies equation \eqref{eq_key}, we have $\textbf{u}_1=\bm{0}$ in $\mathbb Z_{m'}$. Thus $\theta$ is a group isomorphism. This means 
$\left| \{x+m' \mathbb Z^{V_{\lambda}}\mid x\in \varphi(\langle\Lambda_{\partial}\rangle)\}\right| = |(\im\mu)^t| = |\im\mu|^t$. Then Lemma \ref{lem;Im_mu} completes the proof.
\end{proof}

\begin{lem}\label{lem;Im_mu}
For $\mu\colon\mathbb Z_{m'}^{n-1}\rightarrow \mathbb Z_{m'}^{n-1},\ \mathbf{p}\mapsto 2\mathbf{p} G$, we have $|\im\mu| = m'm^{n-2}$.
\end{lem}
\begin{proof}
Recall $G=EF$ (Lemma \ref{matrixG}). Since $2$ and $\det{E}=1$ are invertible in $\mathbb{Z}_{m'}$, we can regard $\mu$ as a map defined by $\mu(\textbf{p}) =\textbf{p} F$ up to isomorphism. Suppose $\textbf{p}=(p_1,\cdots,p_{n-1})$. Then $$\mu(\textbf{p})  = ((n-1)p_1-np_2,(n-2)p_1-np_3,\cdots,2p_1-np_{n-1},p_1) =\bm{0}$$ implies 
$p_1= 0$ and $ np_{i} = 0$, $2\leq i\leq n-1$. Since $\gcd(n,m')=d$, we have $|\kernel \mu| = d^{n-2}$. Then 
$|\im\mu|=\dfrac{(m')^{n-1}}{|\kernel \mu|} = m'(\frac{m'}{d})^{n-2}=m'm^{n-2}$.
\end{proof}

\begin{thm}\label{thm:rank}
Let $\Sigma$ be a triangulable essentially
bordered pb surface without interior punctures, $\lambda$ be a triangulation of $\Sigma$, and $r(\Sigma) := \# (\partial \Sigma) - \chi(\Sigma)$, where $\chi(\Sigma)$ denotes the Euler characteristic of $\Sigma$.
Suppose $\overline\Sigma$ contains $t$ even boundary components. 
We have $$\rankZ \cS_n(\Sigma,\mathbbm{v})=\rankZ \A= d^{r(\Sigma)-t}m^{(n^2-1)r(\Sigma)-t(n-1)}, $$
where $d,m$ are defined in Section~\ref{notation}.
\end{thm}
\begin{proof}
Proposition \ref{rank_eq} claims $\rankZ \cS_n(\Sigma,\mathbbm{v})=\rankZ \A$.

From Lemma \ref{lem5.1} and Remark \ref{rem_partial}, we have 
$$\rankZ\A=\left|\dfrac{\mathbb Z^{V_{\lambda}'}}{\Lambda_z}\right|=\left|\dfrac{\Lambda_{\lambda}}{\varphi(\Lambda_z)}\right|.$$
We have 
$$\left|\frac{\Lambda_{\lambda}}{\Lambda_{\lambda}\cap m'\mathbb Z^{V_{\lambda}}}\right| = \left|\dfrac{\Lambda_{\lambda}}{\varphi(\Lambda_z)}\right| \left|\dfrac{\varphi(\Lambda_z)}{\Lambda_{\lambda}\cap m'\mathbb Z^{V_{\lambda}}}\right|$$
Lemmas \ref{lem5.8} and \ref{lem5.10} imply 
$\rankZ \A= d^{r(\Sigma)-t}m^{|V_{\lambda}|-t(n-1)}$, 
where $|V_{\lambda}|=(n^2-1)r(\Sigma)$ shown in \cite[Lemma 11.2]{LY23}.  

\end{proof}

\begin{rem}
When $d=1$, the rank in Theorem \ref{thm:rank} recovers Theorem 5.3 in \cite{Yu23}.
\end{rem}

\subsection{Proof of Lemma \ref{lem5.3}}\label{sec;pf_lem}
\begin{proof}
The triangulation $\lambda$ gives a cell decomposition of $\overline\Sigma$. We orient the $1$-simplices so that the orientations of the $1$-simplices contained in $\partial \overline \Sigma$ match with that of $\overline\Sigma$. The orientation of $\overline\Sigma$ gives the orientations of the $2$-simplices. 
For every $1$-simplex $e$, we label the vertices of the $n$-triangulation in $e$ as $v_1^{e},v_2^{e},\cdots,v_{n-1}^{e}$  consecutively using the orientation of $e$ such that the orientation of $e$ is given from $v_1^{e}$ to $v_{n-1}^e$. 

First, we will construct a well-defined group homomorphism $J$.
For any $\mathbf{k}\in \Lambda_{\lambda}\cap k\mathbb Z^{V_{\lambda}}$, Lemma \ref{balance} implies there exists $s^{\mathbf{k}}_e\in\mathbb Z$ such that
$s^{\mathbf{k}}_e(1,2,\cdots,n-1) = (\mathbf{k}(v_1^e),\cdots, \mathbf{k}(v_{n-1}^e)) \in \mathbb Z_n$ (note that $s^{\mathbf{k}}_e$ is unique as an element in $\mathbb Z_n$).
Since $\mathbf{k}\in k\mathbb Z^{V_{\lambda}}$ and $\gcd(k,n)=l$, $s^{\mathbf{k}}_e$ is a multiple of $l$, especially $s^{\mathbf{k}}_e\in l\mathbb Z_n\subset \bZ_n$.  
Let $s^{\mathbf{k}} \in C^1(\overline\Sigma,\mathbb Z_n)$ be an element such that every 1-simplex $e$ is assigned with $s^{\mathbf{k}}_e\in l\mathbb Z_n$. From the above discussion, we know $s^{\mathbf{k}}$ is a well-defined element in $lC^1(\overline\Sigma,\mathbb Z_n)$. Then the remaining part is to show $s^{\mathbf{k}}\in Z^1(\overline\Sigma,\mathbb Z_n)$. Suppose $\tau$ is a 2-simplex, and $e_1,e_2,e_3$ are the three 1-simplices which bound $\tau$. See the left picture in Figure \ref{Fig;coord_ijk}. Assume the orientation $e_2$ matches with that of $\tau$ and the orientations of $e_1,e_3$ do not match with that of $\tau$.
In the $n$-triangulation of $\tau$, we have $v_i^{e_1} = (n-i, i,0),\; v_i^{e_2} = (0,i,n-i),\; v_i^{e_3} = (i,0,n-i)$. 
Since $\mathbf{k}$ is balanced, the restriction $\mathbf{k}|_{\tau}$ of $\mathbf{k}$ to $\tau$ equals $x_1\mathbf{pr}_1+x_2\mathbf{pr}_2+x_3\mathbf{pr}_3$ in $\mathbb Z_n$. Then 
\begin{align*}
    \mathbf{k}(v_i^{e_1}) = (x_2-x_1)i\in \mathbb Z_n,\;
    \mathbf{k}(v_i^{e_2}) = (x_2-x_3)i\in \mathbb Z_n,\;
    \mathbf{k}(v_i^{e_3}) = (x_1-x_3)i\in \mathbb Z_n.
\end{align*}
Thus 
$$s^{\mathbf{k}}_{e_1} = x_2-x_1\in \mathbb Z_n,\; s^{\mathbf{k}}_{e_2} = x_2-x_3\in \mathbb Z_n,\; s^{\mathbf{k}}_{e_3} = x_1-x_3\in \mathbb Z_n.$$
Since $-s^{\mathbf{k}}_{e_1} +s^{\mathbf{k}}_{e_2}-s^{\mathbf{k}}_{e_3} = 0\in \mathbb Z_n$, we have $s^{\mathbf{k}}\in Z^1(\overline\Sigma,\mathbb Z_n)$. Then define
$$J\colon\Lambda_{\lambda}\cap k\mathbb Z^{V_{\lambda}} \rightarrow Z^1(\overline \Sigma,\mathbb Z_n)_l,\quad \mathbf{k}\mapsto s^{\mathbf{k}}.$$
Thus $J$ is a well-defined group homomorphism.

Next, we will show the surjectivity of $J$.
Note that every element in $C^1(\overline\Sigma,\mathbb Z_n)$ is represented by a map from the set of all the 1-simplices to $\mathbb Z_n$. Suppose $c\in Z^1(\overline\Sigma,\mathbb Z_n)_l$. For any 1-simplex $e$, we choose $t_e\in \mathbb Z$ such that $c(e) = t_e\in\mathbb Z_n$. Since $c\in lC^1(\overline\Sigma,\mathbb Z_n)$, we have $t_e$ is a multiple of $l$ for each 1-simplex $e$. For each 2-simplex $\tau$, suppose $\tau$, $e_1,\;e_2,\;e_3$ look like in the left picture in Figure \ref{Fig;coord_ijk}. Assume the orientation $e_2$ is the one induced from $\tau$ and the orientation of $e_1,e_3$ is the one opposite to the orientation induced from $\tau$. 
    Since $c\in Z^1(\overline\Sigma,\mathbb Z_n)$, we have $-t_{e_1}+t_{e_2}-t_{e_3} = 0\in \mathbb Z_n$. 
    Set $y_1 = -t_{e_1}$, $y_2 = 0$, $y_3 = -t_{e_2}$. 
    Then we have the following equations in $\mathbb Z_n$:
    \begin{align}\label{eqy}
        y_2-y_1 = t_{e_1},\; y_2-y_3 = t_{e_2},\; y_1-y_3=t_{e_3}.
    \end{align}
    Note that each $y_i$ is a multiple of $l$. Since $\gcd(n/l,k/l) = 1$, the equation $\frac{y_1}{l}\mathbf{pr}_1+\frac{y_2}{l}\mathbf{pr}_2+\frac{y_3}{l}\mathbf{pr}_3+\frac{n}{l}\mathbf{x} = 0\in \mathbb Z_{\frac{k}{l}}$ has a unique solution in $\mathbb Z_{\frac{k}{l}}$. Then there exists $\mathbf{a}_{\tau}\in\mathbb Z^{\overline V_{\tau}}$ such that 
    $\frac{y_1}{l}\mathbf{pr}_1+\frac{y_2}{l}\mathbf{pr}_2+\frac{y_3}{l}\mathbf{pr}_3+\frac{n}{l}\mathbf{a}_{\tau} = 0 \in \mathbb Z_{\frac{k}{l}}$. Then we have     $y_1\mathbf{pr}_1+y_2\mathbf{pr}_2+y_3\mathbf{pr}_3+n\mathbf{a}_{\tau} = 0 \in \mathbb Z_{k}$. For each 2-simplex $\tau$, define 
    $\mathbf{k}_{\tau} = y_1\mathbf{pr}_1+y_2\mathbf{pr}_2+y_3\mathbf{pr}_3+n\mathbf{a}_{\tau}\in k\mathbb Z^{\overline{V}_{\tau}}$. Then equation \eqref{eqy} implies we have the following equations in $\mathbb Z_n$:
\begin{align}\label{eq_edge}
    \mathbf{k}_{\tau}(v_i^{e_j}) = (-1)^j(y_j-y_{j+1})i=c(e_j)i\ (j=1,2,3)
\end{align}
where the indices are $\bZ_3$-cyclic. 

For each attached triangle $\tau$, we label some vertices in $\tau$ as Figure \ref{Fig;coord_uvw}. 
Let us recall the coordinates for these vertices
$$u_i=(i,0,n-i),\quad v_i=(n-i,i,0),\quad w_i=(0,i,n-i).$$
Suppose $\tau$ is attached to the 1-simplex $e$.
Since $\gcd(\frac{n}{l},\frac{k}{l}) = 1$, there exist $r,s\in\mathbb Z$ such that $r\frac{n}{l}+s\frac{k}{l} = 1$.
Define $\mathbf{f} = -r\frac{t_e}{l}\mathbf{pr}_2$. 
Then 
$\frac{t_e}{l}\mathbf{pr}_2+\frac{n}{l} \mathbf{f} = 0\in \mathbb{Z}_{\frac{k}{l}}$. Set $\mathbf{k}_{\tau} = t_e\mathbf{pr}_2+ n\mathbf{f}$. Then $\mathbf{k}_{\tau} =\bm{0}\in \mathbb Z_{k}$,  $\mathbf{k}_{\tau}(u_i) = 0$ and $\mathbf{k}_{\tau}(v_i) = \mathbf{k}_{\tau}(v_i^e) = t_e i = c(e) i$ for 
    $1\leq i\leq n-1$. 

    We will construct $\mathbf{g}\in \mathbb Z^{\overline{V}_{\lambda}}$ whose restriction on each triangle $\tau$ is $\mathbf{k}_{\tau}$. For each 1-simplex $e$, suppose the triangle $\tau$ (resp. $\tau'$) is on your right (resp. left) if you walk along $e$ (we consider the triangulation $\lambda^\ast$ of $\Sigma^\ast$ when $e$ is contained in $\partial\overline\Sigma$).
    Note that $\tau'$ can never be the attached triangle.
   Equation \eqref{eq_edge} implies 
    $$(\mathbf{k}_{\tau}(v_1^{e}),\cdots,\mathbf{k}_{\tau}(v_{n-1}^{e}))=(\mathbf{k}_{\tau'}(v_1^{e}),\cdots,\mathbf{k}_{\tau'}(v_{n-1}^{e}))\in \mathbb Z_n.$$ We also have 
    $$(\mathbf{k}_{\tau}(v_1^{e}),\cdots,\mathbf{k}_{\tau}(v_{n-1}^{e}))=(\mathbf{k}_{\tau'}(v_1^{e}),\cdots,\mathbf{k}_{\tau'}(v_{n-1}^{e}))=\bm{0}\in \mathbb Z_k.$$
    Then we have 
    $$(\mathbf{k}_{\tau}(v_1^{e}),\cdots,\mathbf{k}_{\tau}(v_{n-1}^{e}))=(\mathbf{k}_{\tau'}(v_1^{e}),\cdots,\mathbf{k}_{\tau'}(v_{n-1}^{e}))\in \mathbb Z_N.$$
    Thus there exists $\mathbf{b}_e\in N\mathbb Z^{n-1}$ such that 
    $$(\mathbf{k}_{\tau}(v_1^{e}),\cdots,\mathbf{k}_{\tau}(v_{n-1}^{e}))=\mathbf{b}_e+(\mathbf{k}_{\tau'}(v_1^{e}),\cdots,\mathbf{k}_{\tau'}(v_{n-1}^{e})).$$
    Define $\mathbf{b}_{\tau'}\in \mathbb Z^{\overline{V}_{\tau'}}$ as $$\mathbf{b}_{\tau'}(v) = \begin{cases}
        \mathbf{b}_e(v_i^e)& v= v_i^e\\
        \bm{0} &\text{otherwise}.
    \end{cases}$$
    Then we define a new $\mathbf{k}_{\tau'}$ as $\mathbf{k}_{\tau'}+\mathbf{b}_{\tau'}$.
    We apply the above procedure consecutively to every  1-simplex. Then for each  1-simplex $e$, we have $\mathbf{k}_{\tau}$ and $\mathbf{k}_{\tau'}$ agree with each on $e$, where $\tau$ and $\tau'$ are the two triangles containing $e$. 
    Thus there exists $\mathbf{g}\in \mathbb Z^{\overline{V}_{\lambda}}$ such that $\mathbf{g}|_{\tau} = \mathbf{k}_{\tau}$ in $\mathbb Z_N$ for every triangle $\tau$ in $\lambda^{\ast}$. 
    From the construction of $\mathbf{g}$, we have $\mathbf{g}\in \Lambda_{\lambda}\cap k\mathbb Z^{V_{\lambda}}$ and $J(\mathbf{g}) = c$. Thus $J$ is surjective. 

Obviously, we have $J\circ L = 0$. Then, to show the exactness, it suffices to show $\kernel J\subset \im L$. 
Suppose $\mathbf{h}\in\kernel J$.  
For each 2-simplex $\tau$ in $\overline{\Sigma}$, suppose $\tau$, $e_1,\;e_2,\;e_3$ look like the left picture in Figure \ref{Fig;coord_ijk}. Assume the orientation $e_2$ is the one induced from $\tau$ and the orientation of $e_1,e_3$ is the one opposite to the orientation induced from $\tau$.
Since $\mathbf{h}$ is balanced, suppose $\mathbf{h}|_{\tau} = 
z_1\mathbf{pr}_1+z_2\mathbf{pr}_2+z_3\mathbf{pr}_3+n\mathbf{p}$. 
Since $J(\mathbf{h}) =\bm{0}$, we have $z_2-z_1 = z_2-z_3=z_1-z_3=0\in\mathbb Z_n$. Then for every $(ijk)\in\overline{V}_{\tau}$, we have 
$$\mathbf{h}|_{\tau}((ijk)) = 
z_1 i+z_2j+z_3k = (z_1-z_3) i+(z_2-z_3)j = 0 \in \mathbb Z_n.$$
Thus $\mathbf{h}|_{\tau}=\bm{0}\in\mathbb Z_n$.

Suppose $\tau$ is an attached triangle whose attaching edge is  $e$. Since $\mathbf{h}|_{\tau} = t_1\mathbf{pr}_1+t_2\mathbf{pr}_2+t_3\mathbf{pr}_3$ in   $\mathbb Z_n$, we have $\mathbf{h}(u_i) = t_1-t_3 = 0\in\mathbb Z_n$. 
Then 
$$\mathbf{h}|_{\tau} = t_1\mathbf{pr}_1+t_2\mathbf{pr}_2+t_1\mathbf{pr}_3 = 
t_1\mathbf{pr}_1+t_2\mathbf{pr}_2-t_1(\mathbf{pr}_1+\mathbf{pr}_2)
= (t_2-t_1)\mathbf{pr}_2\text{ in }\mathbb Z_n.$$
This implies that $\mathbf{h}(v_i^e) = \mathbf{h}(v_i) = (t_2-t_1)i\in\mathbb Z_n$ for $1\leq i\leq n-1$. Since $J(\mathbf{h}) =\bm{0}$, we have $t_2-t_1-0\in\mathbb Z_n$. Thus $\mathbf{h}|_{\tau} =\bm{0}$ in $\mathbb Z_n$.

From the above discussion, we know $\mathbf{h} =\bm{0}$ in $\mathbb Z_n$. We also have $\mathbf{h}\in k\mathbb Z^{V_{\lambda}}$. Then 
$\mathbf{h}\in N\mathbb Z^{V_{\lambda}}$. This shows $\mathbf{h}\in\im L$. 
\end{proof}

\section{On reduced stated $\SL(n)$-skein algebras}
In this section, we will formulate the center and PI-degree of the reduced stated $\SL(n)$-skein algebra when $\hat q$ is a root of unity. Although the strategy is similar to the non-reduced case, the boundary central elements need some "symmetric property", Lemma \ref{reduced-boundary_center} (b). So we have to deal with this symmetric property when we prove the center of the reduced stated $\SL(n)$-skein algebra. Not like the non-reduced case, the boundary central elements in Lemma \ref{reduced-boundary_center} do not need the assumption that $\hat q$ is a root of unity. 
Reduced stated $\SL(n)$-skein algebras are (potentially) related to quantum (higher) cluster algebras and understanding representation theory of reduced stated $\SL(n)$-skein algebras would be helpful that of quantum (higher) cluster algebras. See \cite{Mul16, LY22, IY23} for more details. 
Toward deeper understanding of representations of reduced stated $\SL(n)$-skein algebras related to Azumaya loci, see \cite{KK23} for example.

\subsection{Quantum torus matrices}\label{sub:quantum-torus-reduced}
Suppose that $\partial\overline{\Sigma}$ is labeled as $\partial_1,\partial_2,\dots,\partial_b$. 
For each boundary component of $\overline{\Sigma}$, label boundary punctures and boundary edges of $\Sigma$ on $\partial_i$ by $p_1, p_2,\dots, p_{r_i}$ and
$e_1, e_2,\dots, e_{r_i}$ in the orientation of $\partial \overline{\Sigma}$. 
Here the endpoints of $e_k$ $(k=1,2,\dots, r_i)$ are $p_k$ and $p_{k+1}$ with indices in $\bZ_{r_i}$. 

We consider an ideal triangulation of $\Sigma$ defined as follows. 
Suppose $r_i\geq 2$. 
We take ideal arcs $e_{2k-1,2k+1}$ $(k=1,2,\dots,\lfloor r_i/2\rfloor)$ connecting $p_{2k-1}$ and $p_{2k}$ which are relatively homotopic to a part of the boundary component, where $\lfloor \,\cdot\,\rfloor$ denotes the floor function.
Then $e_{2k-1},e_{2k},e_{2k-1,2k+1}$ forms an ideal triangle. 
We will regard the obtained triangles as attached triangles and use the coordinates defined in Section~\ref{sec:vertex}. 

When $r_i>1$ is odd, we additionally take an ideal arc $e_{r_i-2,1}$ connecting $p_{1}$ and $p_{r_i-2}$ so that $e_{r_i-2,1}, e_{r_i-1}, e_{r_i-2,r_i-1}$ forms an ideal triangle. 
Note that each boundary edge is a boundary of an ideal triangle bounded by some of the above ideal arcs and boundary edges. 
See, e.g. Figure \ref{Fig;arcs_bdr_odd}.

After the above procedure, we take additional ideal arcs to get an ideal triangulation of $\Sigma$, denoted by $\mu$. 
By regarding the triangle bounded by $e_{2k-1},e_{2k},e_{2k-1,2k+1}$ as an attached triangle, let $W$ (resp. $U$) denote the set of all the small vertices $w_j$ (resp. $u_j$) on $\partial\overline{\Sigma}$ with the same indices in Figure \ref{Fig;coord_uvw}. 
When $r_i$ is odd, we label the small vertices on $e_{r_i}$ by $a_1,a_2,\dots,a_{n-1}$ as in Figures~\ref{Fig:r=odd}, \ref{Fig:r=1}.

When $r_i>1$,
we will define the sets $W_i$ and $U_i$. 
Let $W_i$ denote the set of the small vertices in $W$ which are on $\partial_i$. 
Each $w\in W_i$ is determined by $(e_j,w_k)$, where $w$ is the small vertex with the coordinate $w_k$ on the edge $e_j$ in $\partial_i$. 
We identify $(e_j,w_k)$ with $(r_i-j,n-k)$. 
Consider the lexicographic order on $W_i$ with respect to $(r_i-j,n-k)$ and suppose that $W$ is equipped with the order when we regard $W$ is the union of $W_i$'s. 
In the same manner, we define $U_i$ and an order on $U_i$, and suppose $U_i$ is also equipped with the order. 

Only in the case when $r_i\geq 1$ is odd,
set $V_{i}=\{a_1,a_2,\dots, a_{r_i}\}$ and suppose it is equipped with the order defined by $a_j>a_k$ if $j<k$.

In the rest of this subsection, we assume $\Sigma$ is an essentially bordered pb surface and contains no interior punctures.

\begin{figure}[h]  
	\centering\includegraphics[width=7cm]{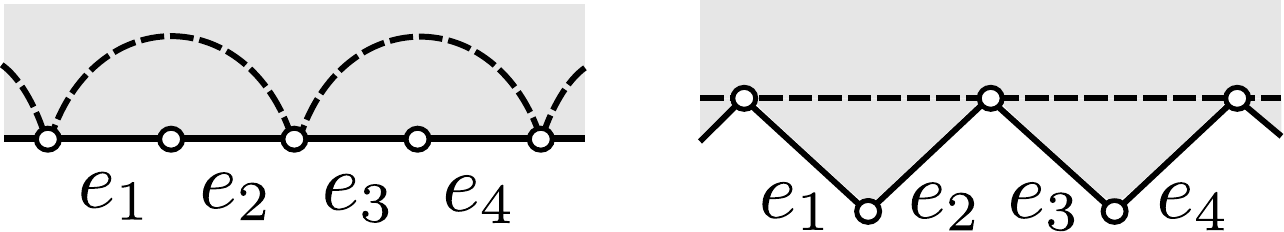}
	\caption{The case when $r_i$ is even. Left: an ideal triangulation $\mu$,\ Right: $\mu$ as an extended triangulation $\lambda^\ast$.}  
	\label{Fig:r=even}   
\end{figure}

\begin{figure}[h]  
	\centering\includegraphics[width=13.5cm]{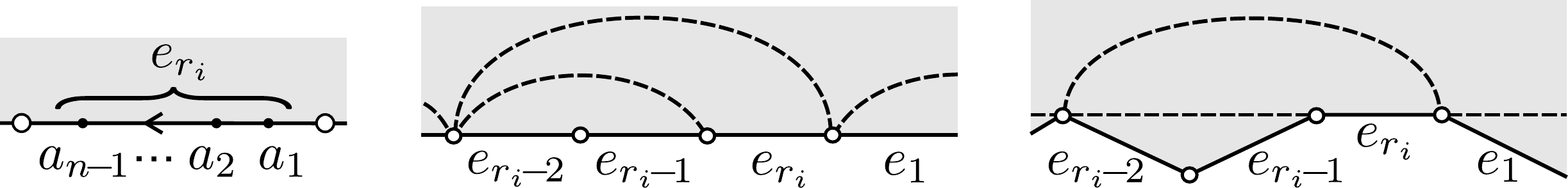}
	\caption{The case when $r_i$ is odd. Left: the small vertices $a_1,a_2,\dots, a_{n-1}$ on $e_{r_i}$,\ Middle: an ideal triangulation $\mu$,\ Right: $\mu$ as a partially extended triangulation.}  
	\label{Fig:r=odd}   
\end{figure}

\begin{figure}[h]
    \centering
    \includegraphics[width=100pt]{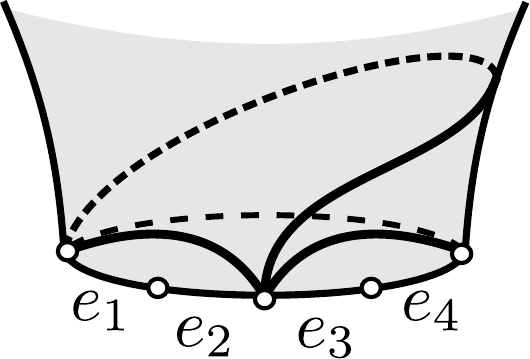}
    \caption{Case of $r_i=5$. The edges $e_{3,1}, e_5, e_{3,4}$ bound a triangle.}\label{Fig;arcs_bdr_odd}
\end{figure}

\begin{lem}\label{lem:barKQ_even}
Let $\Sigma$ be an essentially bordered pb surface without interior punctures such that $\overline{\Sigma}$ has only even boundary components.  
Let $\lambda$ be an ideal triangulation whose extended ideal triangulation $\lambda^\ast$ is $\mu$. 
Then we have 
\begin{align}
\barK_{\mu}\barQ_{\mu}=\barKl\barQl=
\begin{pmatrix}
-2nI & \ast & \ast\\
O    & A    & -A \\
O    & B    & A
\end{pmatrix}, \label{eq:barKQ_even}
\end{align}
where the columns and the rows are divided into $(\mathring{\barV}_{\lambda^\ast}, W, U)$ with respect to $\lambda^\ast$, where $W,U$ are defined in Section~\ref{sec:vertex}. 
\end{lem}
\begin{proof}
From \eqref{eq:matrix_barKQast} and Lemmas \ref{matrixA} and \ref{matrixB}, it suffices to show the part corresponding to $(W,U)\times U$ in \eqref{eq:barKQ_even}. 

First, we consider the restriction of $\barK_{\mu}\barQ_{\mu}$ on $W\times U$. 
From the definitions of skeletons, $\barK_{\mu}$, and $\barQ_{\mu}$, we have 
\begin{eqnarray*}
&{}&\barK_{\mu}\barQ_{\mu}(w_i, u_j)=\sum_{v\in \barV_\mu} \barK_{\mu}(w_i,v)\barQ_{\mu}(v,u_j)\\
&{}&=2\barK_{\mu}(w_i,(j-1,1,n-j))-2\barK_{\mu}(w_i,(j,1,n-j-1))
+\barK_{\mu}(w_i,u_{j+1})-\barK_{\mu}(w_i,u_{j-1}). 
\end{eqnarray*}

Next, we consider the restriction of $\barK_{\mu}\barQ_{\mu}$ on $W\times U$. 
From the definitions of skeletons, $\barK_{\mu}$, and $\barQ_{\mu}$, we have 
\begin{eqnarray*}
&{}&\barK_{\mu}\barQ_{\mu}(u_i, u_j)=\sum_{v\in \barV_\mu} \barK_{\mu}(w_i,v)\barQ_{\mu}(v,u_j)\\
&{}&=2\barK_{\tau}(u_i,(j-1,1,n-j))-2\barK_{\tau}(u_i,(j,1,n-j-1))
+\barK_{\tau}(u_i,u_{j+1})-\barK_{\tau}(u_i,u_{j-1})\\
&{}&\ +2\barK_{\tau}(v_i,(j-1,1,n-j))-2\barK_{\tau}(v_i,(j,1,n-j-1))
+\barK_{\tau}(v_i,u_{j+1})-\barK_{\tau}(v_i,u_{j-1}).  
\end{eqnarray*}

In both cases, we have the claim by concrete computations.
\end{proof}

Let $I'$  denote the anti-diagonal matrix whose anti-diagonal entries are 1. 
Let $G'$ denote the matrix obtained from $G$, defined in \eqref{eq-matrix-G-def}, by reversing the rows. 
For $i\in\bZ_{\geq1}$, define the following matrices:
\begin{align*}
    (B;i)&=\begin{pmatrix}
     O & O &  \cdots &O & nI \\
     nI & O & \cdots &O & O \\
     O  & nI& \cdots & O & O\\
     \vdots & \vdots &  & \vdots & \vdots \\
     O & O & \cdots  & nI &  O \\
    \end{pmatrix},\qquad
    (B_O;i)=\begin{pmatrix}
     O & O &  \cdots &O & O \\
     nI & O & \cdots &O & O \\
     O  & nI& \cdots & O & O\\
     \vdots & \vdots &  & \vdots & \vdots \\
     O & O & \cdots  & nI &  O \\
    \end{pmatrix},\;\\
    (A;i) &= \text{diag}\{-nI,\cdots,-nI\},\quad
    (E;i) = (O,\cdots,O, nI),\quad
    (E^T; i) = (nI',O,\cdots,O)^T,\\
      (\widetilde{G};i)&=\begin{pmatrix}
     O & O &  \cdots &O & G \\
     G & O & \cdots &O & O \\
     O  & G& \cdots & O & O\\
     \vdots & \vdots &  & \vdots & \vdots \\
     O & O & \cdots  & G &  O \\
    \end{pmatrix},\qquad
    (\widetilde{G}_O;i)=\begin{pmatrix}
     O & O &  \cdots &O & O \\
     G & O & \cdots &O & O \\
     O  & G& \cdots & O & O\\
     \vdots & \vdots &  & \vdots & \vdots \\
     O & O & \cdots  & G &  O \\
    \end{pmatrix},\\
    (G;i) &= \text{diag}\{G,\cdots,G\},\qquad
    (G_O;i) = \text{diag}\{O,G,\cdots,G\},\\
    (E_G;i) &= (O,\cdots,O, G),\qquad
    (E_G^T; i) = (G',O,\cdots,O)^T,
\end{align*}
where the block matrices $O,I, G$ are of size $n-1$, and the number of the block matrices in the rows or columns is $i$ in each matrix.

From Lemma \ref{lemKQ}, we can
write the matrices 
in the following forms
\begin{align}
\barK_{\mu}\barQ_{\mu}=
\begin{pmatrix}
-2nI & P' \\
O    &  P \\
\end{pmatrix}, \quad
\barK_{\mu}=
\begin{pmatrix}
\ast & \ast \\
\ast    &  K_{\partial} \\
\end{pmatrix},\label{eq;matrix_P}
\end{align}
where the rows and the columns are divided in $(\mathring{\barV}_\mu, \barV_\mu\setminus \mathring{\barV}_\mu)$.

Suppose that the $i$-th boundary component $\partial_i$ of $\overline{\Sigma}$ contains $r_i$ punctures. 
\begin{lem}\label{lem-reduced-P}
The matrix $P$ has the form
    \begin{align}\label{eq-matrix-P-reduced}
        P = \text{diag}\{P_1,\cdots, P_b\},
    \end{align}
 where the matrix $P_i$ is associated to the $i$-th boundary component with the following descriptions:
$$P_i = \begin{cases}
    
    \begin{pmatrix}
       (A,{\frac{r_i}{2}}) &  -(A,{\frac{r_i}{2}})\\
       (B,{\frac{r_i}{2}})&  (A,{\frac{r_i}{2}}) \\
    \end{pmatrix}
&\text{ if } r_i\text{ is even,}\medskip\\
   
    \begin{pmatrix}
       (A,{\frac{r_i-1}{2}})&  -(A,{\frac{r_i-1}{2}}) & O \\
       (B_O,{\frac{r_i-1}{2}})  &  (A,{\frac{r_i-1}{2}}) & (E^T,{\frac{r_i-1}{2}}) \\
      (E,{\frac{r_i-1}{2}})  &  O & -nI \\
    \end{pmatrix}
&\text{ if } r_i\text{ is odd and } r_i>1,\medskip\\
     -nI+nI'
&\text{ if } r_i=1,
\end{cases}
$$
with the decomposition of rows and columns in $(W_i,U_i), (W_i,U_i,V_i), V_i$ respectively.
\end{lem}
\begin{proof}
From the definition of $\mu$, it is easy to see that $P$ has a block matrix decomposition like in \eqref{eq-matrix-P-reduced}.

When $r_i$ is even, the claim holds from Lemma \ref{lem:barKQ_even}. 

When $r_i$ is odd and $r_i>1$, 
we use $a_1,a_2,\dots,a_{n-1}$ on $e_{r_i}$ given in Figure~\ref{Fig:r=odd}. 
The restriction of $\barK_{\mu}\barQ_{\mu}$ to $(W_i,U_i)\times (W_i,U_i)$ is obtained from the case when $r_i$ is even except for $(B_O,\frac{r_i-1}{2})$. 
Since there is no contribution of small vertices on $e_{r_i}$ in the restriction, we have $(B_O,\frac{r_i-1}{2})$ instead of $(B,\frac{r_i-1}{2})$.

For any $v\in W_i\cup U_i$ not on $e_{r_i-2}$ nor $e_{r_i-1}$, $\barK_{\mu}\barQ_{\mu}(a_j,v)=0$ from the definition of skeletons.
For $w_k \in e_{r_i-1}$, 
The computation of $\barK_{\mu}\barQ_{\mu}(a_j,w_k)$ is the same with that of $\barK_{\lambda^\ast}\barQ_{\lambda^\ast}(u_j,w_k)\ (u_j\in \tau'\cap U_i,\ w_k\in \tau\cap W_i)$ in non-reduced case, where $\tau,\tau'$ are located as in Figure \ref{Fig;tau_tau'}.
Hence, we have $\barK_{\mu}\barQ_{\mu}(a_j,w_k)=n\delta_{jk}$. 
For $u_k \in e_{r_i-2}\cap U_i$, the computation of 
$\barK_{\mu}\barQ_{\mu}(a_j,u_k)$ is the same with that of $\barK_{\lambda^\ast}\barQ_{\lambda^\ast}(u_j,u_k)\ (u_j\in \tau'\cap U_i,\ u_k\in \tau\cap W_i)$ in non-reduced case. 
Hence, we have $\barK_{\mu}\barQ_{\mu}(a_j,u_k)=0$. See Lemma \ref{KQ}. 

Similarly, for any $v\in U_i\cup W_i$ not on $e_{1}$ nor $e_{r_i-2}$, 
$\barK_{\mu}\barQ_{\mu}(v,a_k)=0$. 
For $u_j \in e_{r_i-2}$, the computation of 
$\barK_{\mu}\barQ_{\mu}(u_j,a_k)$ is the same with that of $\barK_{\lambda^\ast}\barQ_{\lambda^\ast}(u_j,u_k)\ (u_j\in \tau'\cap U_i,\ u_k\in \tau\cap U_i)$ in non-reduced case. 
Hence, we have $\barK_{\mu}\barQ_{\mu}(u_j,a_k)=0$. 
For $u_j \in e_{1}$, the computation of 
$\barK_{\mu}\barQ_{\mu}(u_j,a_k)$ is the same with that of $\barK_{\lambda^\ast}\barQ_{\lambda^\ast}(u_j,w_{n-k})\ (u_j\in \tau'\cap U_i,\ w_{n-k}\in \tau\cap W_i)$ in non-reduced case. 
Hence, we have $\barK_{\mu}\barQ_{\mu}(u_j,a_k)=n\delta_{j,n-k}$.  
See Lemma \ref{KQ}.  
Thus we have the second case of $P_i$. 

In the rest of the proof, we consider the case of $r_i=1$. 
Let $e_1$ be the unique edge on $\partial_i$, and we use  $a_1,a_2,\dots,a_{n-1}$ on $e$ given in Figure~\ref{Fig:r=1}. 
There is a unique triangle $\bar{\tau}$ whose boundary has $e_1$, and we assign $\alpha,\beta,\gamma$ to the corners of the triangle. 
Let $p$ be the boundary puncture of the corner $\gamma$ and $e$ be a boundary edge of $\Sigma$ located in clockwise relative to $\gamma$ at $p$. See Figure~\ref{Fig:r=1}. 
For $u_j \in e\neq e_1$, the computation of 
$\barK_{\mu}\barQ_{\mu}(u_j,a_k)$  is the same with that of $\barK_{\lambda^\ast}\barQ_{\lambda^\ast}(u_j,u_k)\ (u_j\in \tau'\cap U_i,\ w_{n-k}\in \tau\cap U_i)$ in non-reduced case. 
See Lemma \ref{KQ}.  
Hence, we have $\barK_{\mu}\barQ_{\mu}(u_j,a_k)=0$. 
This implies that, for any small vertex $v$ not on $e_1$, 
$\barK_{\mu}\barQ_{\mu}(v,a_k)=0$. 
The remaining part is $\barK_{\mu}\barQ_{\mu}(a_j,a_k)$, and we have 
\begin{eqnarray*}
\barK_{\mu}\barQ_{\mu}(a_j,a_k)&=&\sum_{v\in \barV_\mu}\barK_{\mu}(a_j,v)\barQ_{\mu}(v,a_k)\\
&=&\sum_{v\in \barV_\mu}\barK_{\bar{\tau}}(a_j,v)\barQ_{\bar{\tau}}(v,a_k)+\sum_{v\in \barV_\mu}\barK_{\bar{\tau}}(b_j,v)\barQ_{\bar{\tau}}(v,a_k)
+\sum_{v\in \barV_\mu}\barK_{\bar{\tau}}(c_j,v)\barQ_{\bar{\tau}}(v,a_k). 
\end{eqnarray*}
It is easy to see that $\sum_{v\in \barV_\mu}\barK_{\bar{\tau}}(a_j,v)\barQ_{\bar{\tau}}(v,a_k)=n\delta_{jk}$, $\sum_{v\in \barV_\mu}\barK_{\bar{\tau}}(b_j,v)\barQ_{\bar{\tau}}(v,a_k)=n\delta_{j,n-k}$ and
$\sum_{v\in \barV_\mu}\barK_{\bar{\tau}}(c_j,v)\barQ_{\bar{\tau}}(v,a_k)=0$. 
Thus we have the third case of $P_i$. 
\end{proof}

\begin{figure}[h]  
	\centering\includegraphics[width=10cm]{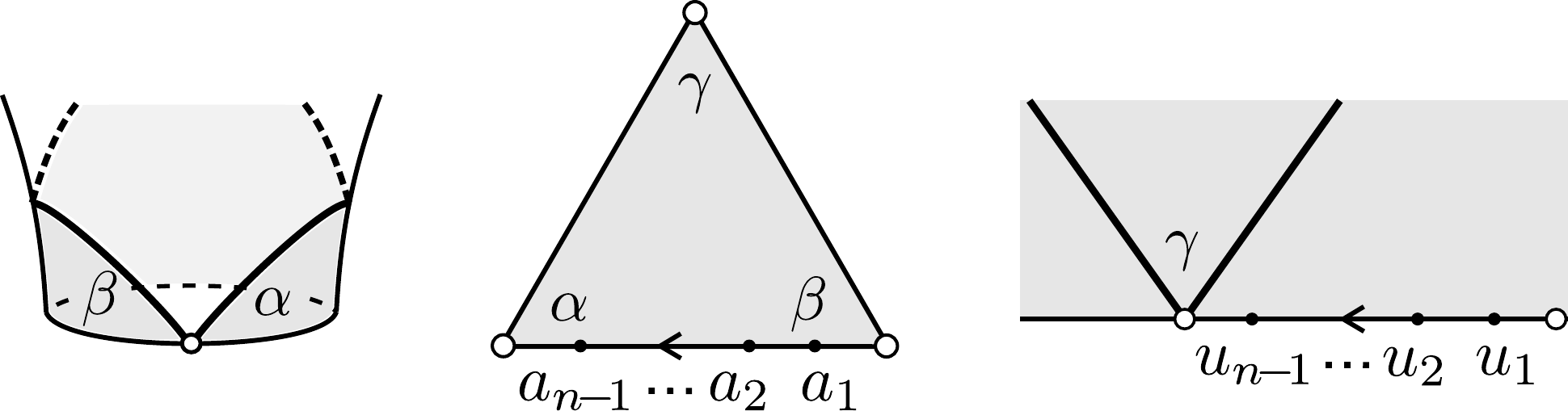}
	\caption{The case of $r_i=1$. Left: the unique triangle on $\partial_i$,\ Middle: the same triangle with the left one with corners $\alpha,\beta,\gamma$ and the small vertices $a_1,a_2,\dots, a_{n-1}$ on $\partial_i$,\ Right: the corner $\gamma$ and the small vertices $u_1,\dots,u_{n-1}$ on $e$.}  
	\label{Fig:r=1}   
\end{figure}

\begin{lem}\label{lem:reduced_P_n}
Each entry of $P'$, given in \eqref{eq;matrix_P}, is $0$ or $n$. 
\end{lem}
\begin{proof}
Note that the claim does not depend on the order of vertex sets. 
Hence, to show the claim, we can identify each boundary edge $e$ with $e_3$ of $\bP_3$ and we will use the coordinate $u_j=(j,0,n-j)$ for small vertices on $e$. 
Let  $\tau$ be the triangle of $\mu$ with $e$. 
Then we formally consider the small vertices in $\tau$
$$u_{j1}=(j-1,1,n-j),\ u_{j2}=(j,1,n-j-1),\ u_{j3}=(j+1,0,n-j-1),\  u_{j4}=(j-1,0,n-j+1).$$ 
While $u_{14}$ (resp. $u_{(n-1)3}$) does not exist for $u_1$ (resp. $u_{n-1}$), the absence does not affect the following computations.

By definition of $\barQ_\mu$, 
\begin{eqnarray*}
\barK_\mu\barQ_\mu(v,u_j)&=&\sum_{u\in\barVl}\barK_\mu(v,u)\barQ_\mu(u,u_j)\\
&=&\sum_{v'}\Big(2\barK_\tau(v',u_{j1})-2\barK_\tau(v',u_{j2})+\barK_\tau(v',u_{j3})-\barK_\tau(v',u_{j4})\Big),
\end{eqnarray*}
where $v'$ depends on the skeleton $\skeleton_\tau(v)$. 
We will check 
\begin{align}
2\barK_\tau(v',u_{j1})-2\barK_\tau(v',u_{j2})+\barK_\tau(v',u_{j3})-\barK_\tau(v',u_{j4})\label{eq:KQ_tau}
\end{align}
is $0$ or $n$ for any $v'\in\barV_\tau\cap \obV_{\lambda}$. 

For $v'\in\barV_\tau\cap \obV_{\lambda}\cap \partial \tau$, from the computations in the proof of Lemma~\ref{lem:barKQ_even}, we know the value of \eqref{eq:KQ_tau} is $0$ or $n$. 

Suppose $v'\in\barV_\tau\cap \obV_{\lambda}\cap \Int\tau$ and use $v'=(a,b,c)\ (a,b,c>0)$. 
First, we consider the case when $a\geq j$, or $a<j$ and $c<n-j$. 
By the definition of $\barK_\tau$, 
we have 
\begin{align*}
&\barK_\tau(v',u_{j1})=a+cj, 
&&\barK_\tau(v',u_{j2})=a+c(j+1),\\ 
&\barK_\tau(v',u_{j3})=c(j+1), 
&&\barK_\tau(v',u_{j4})=c(j-1).
\end{align*}
Hence, \eqref{eq:KQ_tau} is equal to $0$. 

Next, we consider the case when $a<j$ and $c=n-j$. 
Similarly, we have 
\begin{align*}
&\barK_\tau(v',u_{j1})=a+cj, 
&&\barK_\tau(v',u_{j2})=a+(a+b)(n-j-1),\\
&\barK_\tau(v',u_{j3})=(a+b)(n-j-1),
&&\barK_\tau(v',u_{j4})=c(j-1).
\end{align*}
Hence, \eqref{eq:KQ_tau} is equal to $n$. 

Finally, we consider the case when $a<j$ and $c>n-j$. 
Similarly, we have 
\begin{align*}
&\barK_\tau(v',u_{j1})=a+(a+b)(n-j), 
&&\barK_\tau(v',u_{j2})=a+(a+b)(n-j-1),\\
&\barK_\tau(v',u_{j3})=(a+b)(n-j-1),
&&\barK_\tau(v',u_{j4})=(a+b)(n-j+1).
\end{align*}
Hence, \eqref{eq:KQ_tau} is equal to $0$. 

\end{proof}

\begin{lem}\label{reduced-K}
The matrix $K_\partial$ has the form
\begin{align}\label{eq-matrix-K-reduced}
    K_{\partial} = \text{diag}\{S_1,\cdots, S_b\},
\end{align}
     where the matrix $S_i$ is associated to the $i$-th boundary component with the following descriptions:
\begin{align*}
S_i=
\begin{cases} 
    \begin{pmatrix}
       (G;{\frac{r_i}{2}}) & (G;{\frac{r_i}{2}})\\
       (\widetilde{G};{\frac{r_i}{2}})&  (G;{\frac{r_i}{2}}) \\
    \end{pmatrix}
&\text{if  $r_i$ is even,}\medskip\\
\begin{pmatrix}
       (G;{\frac{r_i-1}{2}})&  (G;{\frac{r_i-1}{2}}) & O \\
       (\widetilde{G}_O;{\frac{r_i-1}{2}})  &  (G;{\frac{r_i-1}{2}}) & (E_G^T;{\frac{r_i-1}{2}}) \\
      (E_G;{\frac{r_i-1}{2}})  &  O & G \\
    \end{pmatrix}
&\text{if $r_i$ is odd and  $r_i>1$,}\medskip\\
    G+G'
&\text{if  $r_i=1$},
\end{cases}
\end{align*}
with rows and columns are divided in $(W_i,U_i),\ (W_i,U_i,V_i),\ V_i$ respectively. 
\end{lem}
\begin{proof}
From the definition of $\mu$, it is easy to see that $K_{\partial}$ has a block matrix decomposition like in \eqref{eq-matrix-K-reduced}.

First, we consider the case when $r_i$ is even and use the decomposition \eqref{eq:matrix_K_docomp3} by regarding $\mu$ as an extended ideal triangulation. 
Since the graph $\widetilde{Y}_w$ of $w\in W_i$ consists of the main segment, $K_{22}(w,w')=K_{23}(w,w')=0$ if $w$ and $w'$ are not in the same triangle. 
For $w_j,w_k,u_k$ in a same triangle $\tau$, since $K_{22}(w_j,w_k)=\barK_\tau(w_j,w_k)$ and $K_{23}(w_j,u_k)=\barK_\tau(w_j,u_k)$, we have
$$K_{22}(w_j,w_k)=K_{23}(w_j,u_k)=\begin{cases}
    (n-j)k & \text{if $j\geq k$}\\
    (n-k)j & \text{if $j < k$}
\end{cases}.$$ 

From $K_{22}=(G;{\frac{r_i}{2}})$ and Lemma \ref{matrixK}, we have $K_{32}=(\widetilde{G};{\frac{r_i}{2}})$. 

Since $\barK_\tau(v_j,u_k)=0$, 
$$K_{33}(u_j,u_k)=\begin{cases}
    \barK_\tau(u_j,u_k) & \text{if $u_j, u_k$ are in $\tau$} \\
    0 & \text{otherwise}
\end{cases},\qquad 
\barK_\tau(u_j,u_k)=\begin{cases}
    (n-j)k & \text{if $j\geq k$}\\
    (n-k)j & \text{if $j < k$}
\end{cases}.$$
Hence, we have the first case of $S_i$.

Next we consider the case when $r_i>1$ is odd and use the decomposition
$$S_i=\begin{pmatrix}
K_{22} & K_{23} & K_{24} \\
K_{32} & K_{33} & K_{34} \\
K_{42} & K_{43} & K_{44} 
\end{pmatrix},$$ 
where the rows and columns are divided in $(W_i,U_i, V_i)$. 
Since the graph $\widetilde{Y}_w$ of $w\in W_i$ consists of the main segment, 
the computations of $K_{22}$ and $K_{33}$ are the same with even case. 
Moreover, $K_{24}=O$. 

The arc segments of the graph $\widetilde{Y}_u$ of $u\in U_i\cap e_1$ do not affect the computation of $K_{32}$. 
The computation of the remaining part is the same with even case. 
Since $\barK_{\tau}(v_j,v_k)=0$, we have $K_{33}(u_j,u_k)=\barK_\tau(u_j,u_k)$. We also have
$$K_{34}(u_j,a_k)=\begin{cases}
    \barK_{\tau}(w_{n-j},u_k) & \text{if $u_j\in e_1$}\\
    0 & \text{otherwise}
\end{cases}.$$

Note that the skeleton of any small vertex in $e_{r_i}$ affects only the triangle $\tau$ whose boundary contains $e_{r_i-1}$. This implies that it suffices to check
$K_{42}(a_j,w_k)=\barK_\tau(v_j,w_k), K_{43}(a_j,u_k)=0, K_{44}(a_j,a_k)=\barK_{\bar\tau}(u_j,u_k)$, 
where $\bar\tau$ is the triangle containing $e_{r_i}$. 
Hence, we have the second case of $S_i$.

Finally, we consider the case of $r_i=1$. We have 
$$K_\partial(a_j,a_k)=\barK_{\tau}(a_j,a_k)+\barK_{\tau}(b_j,a_k)+\barK_{\tau}(c_j,a_k)=\barK_\tau(u_j,u_k)+\barK_\tau(w_{n-j},u_k)+0.$$
Hence, we have the third case of $S_i$. 
\end{proof}

\subsection{Central elements for generic $\hat q$}\label{sec:central_generic}

\begin{lem}\label{reduced-height}
For any pb surface $\Sigma$, in $\overline \cS_n(\Sigma,\mathbbm{v})$ we have  the following relations
\begin{align*}
    \begin{array}{c}\includegraphics[scale=0.6]{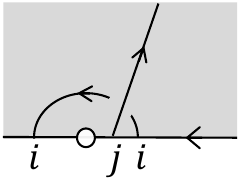}\end{array}
    =q^{-\frac{1}{n}+\delta_{i,j}}
        \begin{array}{c}\includegraphics[scale=0.6]{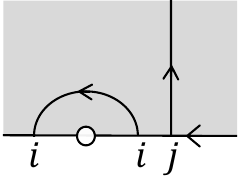}\end{array},\qquad
    \begin{array}{c}\includegraphics[scale=0.6]{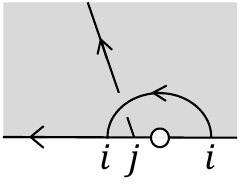}\end{array}
    =q^{\frac{1}{n}-\delta_{n+1-i,j}}
        \begin{array}{c}\includegraphics[scale=0.6]{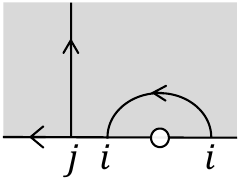}\end{array},\\
            \begin{array}{c}\includegraphics[scale=0.6]{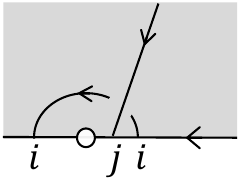}\end{array}
    =q^{\frac{1}{n}-\delta_{n+1-i,j}}
        \begin{array}{c}\includegraphics[scale=0.6]{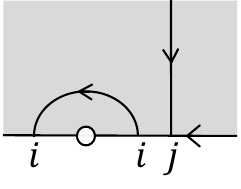}\end{array},\qquad
            \begin{array}{c}\includegraphics[scale=0.6]{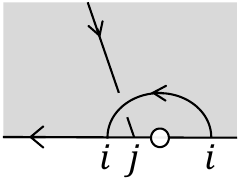}\end{array}
    =q^{-\frac{1}{n}+\delta_{i,j}}
        \begin{array}{c}\includegraphics[scale=0.6]{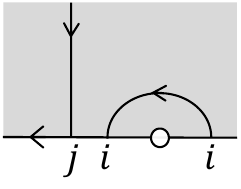}\end{array}.
\end{align*}

\end{lem}
\begin{proof}
Using Relation \eqref{wzh.eight} and \cite[Proposition 3.2]{Wan23}, we have the claim.
\end{proof}

\begin{lem}\label{reduced-boundary_center}
Suppose $\overline{\Sigma}$ has a boundary component $\partial$ such that  
the connected components of $\partial\cap \Sigma$ are labeled as $e_1,e_2,\cdots,e_r$ with respect to the orientation of $\partial$.
For each $1\leq t\leq r$, suppose the small vertices in $e_t$ are labeled as $u_{t,1},\cdots,u_{t,n-1}$ with respect to the orientation of $\partial$. 

(a) Suppose $r$ is even. For $k\in\mathbb N$ and $1\leq i\leq n-1$, the element
\begin{align}\label{eq-reduced-central1}
\bar \gaa_{u_{1,i}}^k \bar \gaa_{u_{2,n-i}}^k\cdots \bar \gaa_{u_{r-1,i}}^k \bar \gaa_{u_{r,n-i}}^k
\end{align}
is central in $\overline \cS_n(\Sigma,\mathbbm{v})$, where $\bar\gaa_v$ with a small vertex $v$ is defined in Section \ref{sub_center}.

(b) Suppose  $r$ is odd. For  $k\in\mathbb N$ and $1\leq i\leq \lfloor\frac{n}{2}\rfloor$, the element
\begin{align}\label{eq-reduced-central2}
\bar \gaa_{u_{1,i}}^k \bar \gaa_{u_{1,n-i}}^k \bar \gaa_{u_{2,i}}^k \bar \gaa_{u_{2,n-i}}^k\cdots \bar \gaa_{u_{r,i}}^k \bar \gaa_{u_{r,n-i}}^k
\end{align}
is central in $\overline \cS_n(\Sigma,\mathbbm{v})$.

\end{lem}
\begin{proof}
    \cite[Lemma 7.5 (b)]{LY23} implies elements $\bar\gaa_{u_{t,i}}$ are $q$-commuting with each other. 
    This means $\bar\gaa_{u_{t,i}} \bar\gaa_{u_{t',i'}} = q^{l}
    \bar\gaa_{u_{t',i'}} \bar\gaa_{u_{t,i}} $ for some integer $l$.
    Thus it suffices to show the case when the integer $k$ in (a) and (b) is $1$.

    \cite[Lemma 7.5 (c)]{LY23} implies
    \begin{align}\label{eq-barguti}
    \bar\gaa_{u_{t,i}}=c_{t,i}\begin{array}{c}\includegraphics[scale=0.38]{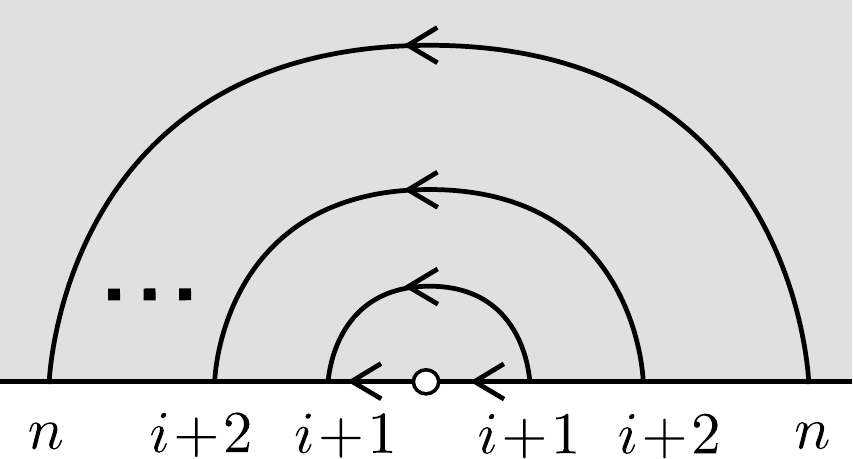}\end{array}.
    \end{align}
    Here $c_{t,i}\in\mathcal R\setminus\{0\}$ and the edge on the right of the puncture is $e_t$. 
    We will use $\begin{array}{c}\includegraphics[scale=0.7]{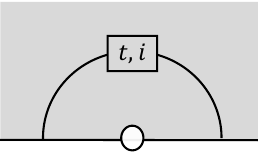}\end{array}$ to denote the stated $n$-web diagram in \eqref{eq-barguti}.

    Lemma \ref{reduced-height} shows 
    \begin{equation}\label{eq-gaa}
    \begin{split}
        &\bar \gaa_{u_{1,i}} \bar \gaa_{u_{2,n-i}}\cdots \bar \gaa_{u_{r-1,i}} \bar \gaa_{u_{r,n-i}}\\=& c
        \begin{array}{c}\includegraphics[scale=0.7]{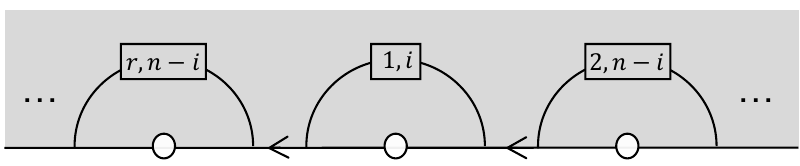}\end{array}
        \end{split}
    \end{equation}
    for some $c\in\mathcal R\setminus\{0\}$ 
    implies 
    \begin{equation}\label{eq-gaa1}
    \begin{split}
        \begin{array}{c}\includegraphics[scale=0.7]{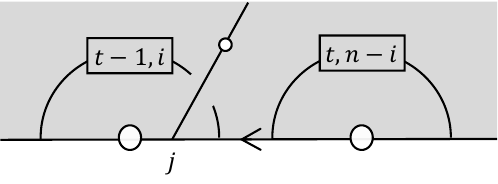}\end{array}=
        \begin{array}{c}\includegraphics[scale=0.7]{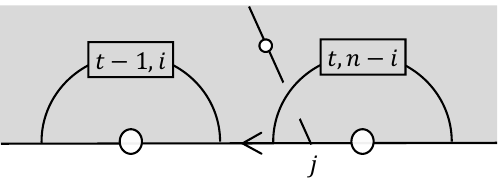}\end{array}.
        \end{split}
    \end{equation}
    Equations \eqref{eq-gaa} and \eqref{eq-gaa1} imply $\bar \gaa_{u_{1,i}} \bar \gaa_{u_{2,n-i}}\cdots \bar \gaa_{u_{r-1,i}} \bar \gaa_{u_{r,n-i}}$ is central.

    Similarly, as above, we can show elements in (b) are central. 
    
\end{proof}

We use $\bar {\mathsf B}$ to denote the subalgebra of $\overline \cS_n(\Sigma,\mathbbm{v})$ generated by central elements in Lemma \ref{reduced-boundary_center}.

In the rest of this section, we always assume $\Sigma$ is a connected triangulable essentially pb surface without interior punctures. 
Fix a triangulation $\lambda$ of $\Sigma$. 

Suppose $\overline{\Sigma}$ has boundary components $\partial_1,\cdots,\partial_b$, and each $\partial_i$ contains $r_i$ punctures.
For each $i$  we label the boundary components of $\Sigma$ contained in $\partial_i$ consecutively from
$1$ to $r_i$ following the positive orientation of $\partial_i$.

For any positive integer $t$ and ${\bf k}=(k_1,\cdots,k_t)\in\mathbb Z^t$, define $\cev{{\bf k}} = (k_t,\cdots,k_1)$.

Let $\overline \Lambda_{\partial}$ be the set consisting of ${\bf k}\in \mathbb Z^{\overline V_{\lambda}}$ with the following properties with the order of the small vertices on $\partial_i$ defined by following the positive orientation of $\partial_i$:
\begin{enumerate}
    \item ${\bf k}|_{\obVl} = {\bf 0}$, 
    \item for each even $r_i$, we have ${\bf k}|_{\partial_i} = ({\bf a},\cdots,{\bf a})$, where ${\bf a}\in\mathbb Z^n$, 
    \item for each odd $r_i$, we have ${\bf k}|_{\partial_i} = ({\bf b},\cdots,{\bf b})$, where ${\bf b}\in\mathbb Z^n$ and $\cev{{\bf b}} = {\bf b}.$ 
\end{enumerate}
It is easy to show that $\overline\Lambda_{\partial}$ is a subgroup of $\mathbb Z^{\overline V_{\lambda}}$ and 
$$\overline{{\rm tr}}_{\lambda}^A(\bar{\mathsf{B}}) = 
\{a^{\bf k}\mid \mathbf{k}\in \overline\Lambda_{\partial}\}.$$
Then Lemma \ref{reduced-boundary_center} implies 
$\{a^{\bf k}\mid k\in \overline\Lambda_{\partial}\}\subset\mathcal Z(\overline{\mathcal A}_v(\Sigma,\lambda))$.

\subsection{The center when $\hat q$ is a root of unity}

\begin{lem}\label{reduced-balance}
Let $\Si$ be a triangulable essentially bordered pb surface with a triangulation $\lambda$.
    Suppose $\mathbf{k}\in\mathbb Z^{\bar V_{\lambda}}$ 
    is balanced, and $e$ is a boundary component of $\Sigma$ containing vertices $x_1,\cdots,x_{n-1}$ (the orientation of $e$ gives the labeling of these vertices). Then there exists $l\in\mathbb Z$ such that $(\mathbf{k}(x_1),\mathbf{k}(x_2),\cdots,\mathbf{k}(x_{n-1})) - l(1,2,\cdots,n-1) =\bm{0}$ in $\mathbb Z_n$.
\end{lem}
\begin{proof}
    The proof is the same with Lemma \ref{balance}.
\end{proof}

\begin{lem}\label{lem-reverse}
    Suppose ${\bf a}=(a_1,\cdots,a_{n-1})\in\mathbb Z^{n-1}$ such that ${\bf a} = \cev{{\bf a}}\text{ in }\mathbb Z_{m'}$. Then there exists ${\bf b}\in\mathbb Z^{n-1}$ such that ${\bf b}+\cev{{\bf b}}= {\bf a}\text{ in }\mathbb Z_{m'}.$
\end{lem}
\begin{proof}
    If $n$ is odd, we set ${\bf b} = (a_1,\cdots,a_{\frac{n-1}{2}},0,\cdots,0)$.

    Suppose $n$ is even. Since $m'$ is odd, there exists $k\in\mathbb Z$ such that $2k = a_{\frac{n}{2}}\in\mathbb Z_{m'}$. Set 
    ${\bf b} = (a_1,\cdots,a_{\frac{n}{2}-1},k,0,\cdots,0)$.
\end{proof}

\begin{lem}\label{lem-reverse-d}
    Suppose ${\bf a}=(a_1,\cdots,a_{n-1})\in\mathbb Z^{n-1}$ such that ${\bf a} = \cev{{\bf a}}\text{ in }\mathbb Z_{m'}$ and ${\bf a} = \bm{0}\text{ in }\mathbb Z_{d}$. Then there exists ${\bf b}\in\mathbb Z^{n-1}$ such that ${\bf b}+\cev{{\bf b}}= {\bf a}\text{ in }\mathbb Z_{m'}$ and 
    ${\bf b} = \bm{0}\in\mathbb Z_{d}$.
\end{lem}
\begin{proof}
    The proof is the same with Lemma \ref{lem-reverse}. Note that  $2k=0\in\mathbb Z_{d}$ implies $k=0\in\mathbb Z_{d}$ since $d$ is odd. 
\end{proof}

We write $G = (G_1,\cdots,G_{n-1})^{T}$. It is easy to show $\overleftarrow{G_i} = G_{n-i}$ for $1\leq i\leq n-1$.
\begin{lem}\label{lem-G-reverse}
    (a) For any ${\bf k}=(k_1,\cdots,k_{n-1})\in\mathbb Z^{n-1}$, we have
    $\cev{{\bf k}}G={\bf k}G'=\overleftarrow{{\bf k} G}$. 

    (b) For any ${\bf k}\in\mathbb Z^{n-1}$, 
    ${\bf b}:={\bf k}(G+G')$ satisfies ${\bf b} = \cev{{\bf b}}$.

    (c) For any ${\bf c}\in\mathbb Z^{n-1}$, we have
    $({\bf c}+\cev{{\bf c}})(G+G') = 2{\bf c}G + 2\overleftarrow{{\bf c}G}$.
\end{lem}
\begin{proof}
(a) We have 
$$\cev{{\bf k}}G =  k_{n-1} G_1+\cdots+k_1G_{n-1} = k_1\overleftarrow{G_1}+\cdots+  k_{n-1}\overleftarrow{G_{n-1}} = \overleftarrow{{\bf k} G}$$
and 
$${\bf k}G' =  k_{n-1} G_1+\cdots+k_1G_{n-1} = k_1\overleftarrow{G_1}+\cdots+  k_{n-1}\overleftarrow{G_{n-1}} = \overleftarrow{{\bf k} G}.$$
(b) follows from (a).\\
(c) From (a), we have 
$$({\bf c}+\cev{{\bf c}})(G+G')=( {\bf c} G+ \cev{{\bf c}} G')+(\cev{{\bf c}} G+{\bf c} G') = 2{\bf c}G + 2\overleftarrow{{\bf c}G}.$$
\end{proof}

Let $\Si$ be an essentially bordered pb surface.
Recall that we introduced a triangulation $\mu$
in Section~\ref{sub:quantum-torus-reduced}. Actually, $\mu$ belongs to a special family of triangulations, i.e., we made specific choices for some ideal arcs near $\partial\Si$.
In the rest of this section, we will assume the triangulation $\lambda$ of $\Si$ is $\mu$.

Define
$$\overline \Lambda_{m'}= \{\mathbf{k}\in\mathbb Z^{\overline V_{\lambda}} \mid 
\mathbf{k}\barK_{\lambda} =\bm{0} \text{ in }\mathbb Z_{m'}\}
,\quad \overline\Lambda_{m'}^{+}= \{\mathbf{k}\in\mathbb N^{\overline V_{\lambda}} \mid 
\mathbf{k} \barK_{\lambda} =\bm{0} \text{ in }\mathbb Z_{m'} \}.$$

Let $\overline \Lambda_z$ denote the subgroup of $\mathbb Z^{\overline V_{\lambda}}$ generated by $\overline \Lambda_{m'}$ and $\overline \Lambda_{\partial}$.

\begin{thm}\label{center_torus-reduced}
Let $\Sigma$ be a triangulable essentially
bordered pb surface without interior punctures, and $\lambda=\mu$ be a triangulation of $\Sigma$ introduced in Section~\ref{sub:quantum-torus-reduced}.
If $m''$ is odd, we have 
$$\mathcal Z(\rA) = \mathbb C\text{-span}\{a^{\bf k}\mid \mathbf{k}\in \overline\Lambda_{z}\}.$$
\end{thm}
While the following proof is similar to the proof of Theorem~\ref{center_torus}, we will use different matrices and discuss more to show Theorem~\ref{center_torus-reduced}. 
Hence, we will give a full proof here. 
\begin{proof}
    It is obvious that $\mathbb C\text{-span}\{a^{\bf k}\mid \mathbf{k}\in \overline\Lambda_{m'}\}\subset \mathcal Z(\rA).$
    Lemma \ref{reduced-boundary_center} implies 
    $\mathbb C\text{-span}\{a^{\bf k}\mid \mathbf{k}\in \overline\Lambda_{\partial}\}\subset \mathcal Z(\rA).$
    Then we have $\mathbb C\text{-span}\{a^{\bf k}\mid \mathbf{k}\in \overline\Lambda_{z}\}\subset \mathcal Z(\rA).$

     Since $m''$ is odd, we have $\gcd(m'',2n) = \gcd(m'',n) = d'$.

    Lemma \ref{quantum} implies $\mathcal Z(\rA) = \mathbb C\text{-span}\{a^{\tft}\mid  \langle {\bf t},{\bf t}'\rangle_{\sfP_{\lambda}}=0\in \mathbb{Z}_{m''},\forall {\bf t}'\in\mathbb{Z}^{\overline V_{\lambda}}  \}$.
    Suppose $\tft_0\in\mathbb Z^{\overline V_{\lambda}}$ such that $\tft \sfP_{\lambda} \tft_0^{T} = 0$ {in $\mathbb Z_{m''}$} for any $\tft\in \mathbb Z^{V_{\lambda}'}$, where we regard $\tft,\tft_0$ as row vectors.  
    Lemma~\ref{lem:invertible_KH} implies $\tft \barK_{\lambda} \barQ_{\lambda} (\tft_0\barK_{\lambda})^T = 0 $ in $\mathbb Z_{m''}$ for all $\tft\in \mathbb Z^{\overline V_{\lambda}}.$
    Set $\tfk _0= \tft_0\barK_{\lambda}$. Then we have $\barK_{\lambda}\barQ_{\lambda} \tfk _0^{T}=\bm{0}$ in $\mathbb Z_{m''}.$
    
    We regard $\mathbf{k}_0 = (
        \tfk_1,\tfk_2
    )\in \bZ^{\overline{V}_{\lambda}}$ with $\tfk_1\in \bZ^{\obVl}$ and $\tfk_2\in \bZ^{\overline V_{\lambda}\setminus \obVl}$.  From 
    \eqref{eq;matrix_P}, we have 
    \begin{equation}\label{eq-reduced-key}
    \begin{cases}
    -2n\bk_1^T+ P' \bk_2^T=\bm{0},\\
    P \bk_2^T=\bm{0},
    \end{cases}\text{ in $\mathbb Z_{m''}$}.
    \end{equation}
     Equation \eqref{eq-reduced-key} is equivalent to the following one
\begin{equation}\label{eq-reduced-key1}
    \begin{cases}
    -2\bk_1^T+ \frac{1}{n}P' \bk_2^T=\bm{0},\\
    \frac{1}{n}P \bk_2^T=\bm{0},
    \end{cases}\text{ in $\mathbb Z_{m'}$}.
    \end{equation}

Suppose $\overline{\Sigma}$ has boundary components $\partial_1,\cdots,\partial_b$, and each boundary component $\partial _i$ contains $r_i$ punctures.
Suppose $\tfk_2 =  (
\tfk_{\partial_1},\tfk_{\partial_2},\cdots,\tfk_{\partial_b}
)$ where $\tfk_{\partial_i}\in\mathbb Z^{r_i(n-1)}$ is the row vector associated to $\partial_i$ for each $1\leq i\leq b$.
Then Lemma \ref{lem-reduced-P} and $P\tfk_2^T =\bm{0}\in\mathbb Z_{m''}$ implies
$P_i\tfk_{\partial_i}^T=0\in\mathbb Z_{m''}$ for $1\leq i\leq b$.

(1) Suppose $r_i$ is even. Let us use $r$ to denote $r_i$ in this case.
Write $\tfk_{\partial_i} = ({\bf k}_i',{\bf k}_i'')$, where ${\bf k}_i'$ (resp. ${\bf k}_i''$) is the part associated to $W$
(resp. $U$).
Then Lemma \ref{lem-reduced-P} and $P_i\tfk_{\partial_i}^T=\bm{0}\in\mathbb Z_{m''}$ imply
\begin{align}\label{eq-reduced-AB}
\begin{cases}
    (A,{\frac{r}{2}})({\bf k}_i')^T-(A,{\frac{r}{2}})({\bf k}_i'')^T=\bm{0},\\
    (B,{\frac{r}{2}})({\bf k}_i')^T+ (A,{\frac{r}{2}})({\bf k}_i'')^T=\bm{0},
\end{cases}\text{ in $\mathbb Z_{m''}$}.
\end{align}
The first equality in  \eqref{eq-reduced-AB} implies
$n({\bf k}_i')^T-n({\bf k}_i'')^T=\bm{0} \text{ in }\mathbb Z_{m''}$. Then we have 
${\bf k}_i'={\bf k}_i''\text{ in }\mathbb Z_{m'}$.
Then the second equality in  \eqref{eq-reduced-AB} implies
$$(B,{\frac{r}{2}})({\bf k}_i')^T+ (A,{\frac{r}{2}})({\bf k}_i'')^T = \big((B,{\frac{r}{2}})+ (A,{\frac{r}{2}})\big)({\bf k}_i')^T=\bm{0}\text{ in }\mathbb Z_{m''}.$$
Suppose $\tfk_{i}'= ({\bf b}_{i1},{\bf b}_{i2},\cdots,{\bf b}_{i\frac{r}{2}})$.
Then we have 
\begin{equation}\label{eq's-reduced}
\begin{cases}
-n\tfb_{i1} + n\tfb_{i2} =\bm{0},\\
 -n\tfb_{i2} + n\tfb_{i3} =\bm{0},\\
\;\vdots\\
-n\tfb_{i\,\frac{r}{2}-1} + n\tfb_{i\frac{r}{2}} =\bm{0},\\
-n\tfb_{i\frac{r}{2}} + n\tfb_{i1} =\bm{0},
\end{cases}\text{ in $\mathbb Z_{m''}$}.
\end{equation}
Equation \eqref{eq's-reduced} implies 
$\tfb_{ij} = \cdots=\tfb_{i\frac{r}{2}}$
in $\mathbb Z_{m'}$.
Thus there exists a vector $\tfb_i\in\mathbb Z^{n-1}$
such that $\tfk_{\partial_i} = (\tfb_i,\cdots,\tfb_i)$ in $\mathbb Z_{m'}$.

(2) Suppose $r_i$ is odd and $r_i>1$. We also use $r$ to denote $r_i$ in this case. 
Write $\tfk_{\partial_i} = ({\bf k}_i',{\bf k}_i'',{\bf k}_i)$, where ${\bf k}_i'$ (resp. ${\bf k}_i''$) is the part associated to $W_i$
(resp. $U_i$).
Then Lemma \ref{lem-reduced-P} and $P_i\tfk_{\partial_i}^T=\bm{0}\in\mathbb Z_{m''}$ imply
\begin{align}\label{eq-reduced-ABE}
\begin{cases}
    (A,\frac{r-1}{2})({\bf k}_i')^T-(A,\frac{r-1}{2})({\bf k}_i'')^T=\bm{0},\\
    (B_O,\frac{r-1}{2})({\bf k}_i')^T+ (A,\frac{r-1}{2})({\bf k}_i'')^T + (E^T,{\frac{r-1}{2}}) (\tfk_i)^T=\bm{0},\\
    (E,{\frac{r-1}{2}})(\tfk_i')^T-n\tfk_i^T=\bm{0}
\end{cases}\text{ in $\mathbb Z_{m''}$}.
\end{align}
The first equation in  \eqref{eq-reduced-ABE} implies
${\bf k}_i'={\bf k}_i'' \text{ in }\mathbb Z_{m'}$.
Suppose $\tfk_{i}'= ({\bf b}_{i1},{\bf b}_{i2},\cdots,{\bf b}_{i\frac{r-1}{2}})$. 
Then the first and second equations in  \eqref{eq-reduced-ABE} imply
\begin{equation}\label{eq's-reduced-1}
\begin{cases}
-n\tfb_{i1} + nI'\tfk_i =\bm{0},\\
-n\tfb_{i1} + n\tfb_{i2} =\bm{0},\\
 -n\tfb_{i2} + n\tfb_{i3} =\bm{0},\\
\;\vdots\\
-n\tfb_{i\frac{r-3}{2}} + n\tfb_{i\frac{r-1}{2}} =\bm{0},
\end{cases}\text{ in $\mathbb Z_{m''}$}.
\end{equation}
Equation \eqref{eq's-reduced-1} implies 
$\tfb_{ij}  = \cdots=\tfb_{i\frac{r-1}{2}}
=\overleftarrow{{\bf k}_i}$ in $\mathbb Z_{m'}$.
The third equation in \eqref{eq-reduced-ABE} implies
$\tfb_{i\frac{r-1}{2}}={\bf k}_i$ 
in $\mathbb Z_{m'}$.
Thus there exists a vector $\tfb_i\in\mathbb Z^{n-1}$
such that $\tfb_i=\cev{\tfb}_i$ in $\mathbb Z_{m'}$ and $\tfk_{\partial_i} = (\tfb_i,\cdots,\tfb_i)$ in $\mathbb Z_{m'}$.

(3) Suppose $r_i=1$. Then $P_i\tfk_{\partial i} =\bm{0}$ in $\mathbb Z_{m''}$ implies $\overleftarrow{\tfk_{\partial _i}} = \tfk_{\partial_i}$ in 
$\mathbb Z_{m'}$. We set $\tfb_i = \tfk_{\partial_i}$.

Now we have vectors ${\bf b}_i$ in all the cases (1)--(3).
When $r_i$ is even, 
using Lemma \ref{reduced-balance} and the technique in the proof of Theorem
\ref{center_torus}, one can show there exists ${\bf d}_i\in\mathbb Z^{n-1}$ such that $2{\bf d}_i G = {\bf b}_i$ in $\mathbb Z_{d}$ for each $1\leq i\leq b$.
When $r_i$ is odd, we have that $\tfb_i=\overleftarrow{\tfb_i}$ in $\mathbb Z_{m'}$. Lemma \ref{reduced-balance} shows that $\tfb_i=l_i(1,2,\cdots,n-1)$ in $\mathbb Z_{n}$. 
When $n$ is odd, we have $l_i(\frac{n-1}{2}) = l_i(\frac{n+1}{2})\in\mathbb Z_d$, which shows $l_i=0\in\mathbb Z_d$.
When $n$ is even, we have $l_i(\frac{n}{2}-1) = l_i(\frac{n}{2}+1)\in\mathbb Z_d$, which shows $2l_i=0\in\mathbb Z_d$.
 Since $d$ is odd, then $l_i=0\in\mathbb Z_d$ and $\tfb_i={\bf 0}$ in $\mathbb Z_{d}$.

Define $\mathbf{d}= (\mathbf{d}_{\partial_1},\mathbf{d}_{\partial_2},\cdots,\mathbf{d}_{\partial_b})\in \mathbb Z^{\overline{V}_{\lambda}\setminus\obVl}$, where $\mathbf{d}_{\partial_i}\in\mathbb Z^{r_i(n-1)}$ is the vector associated to $\partial_i$,  such that 
\begin{equation}
    \mathbf{d}_{\partial_i} = 
        \begin{cases}    ({\bf 0},\cdots,{\bf 0})& \text{if $r_i$ is odd},\\
        (\mathbf{d}_i,\cdots,\mathbf{d}_i) & \text{if $r_i$ is even}.\\
        \end{cases}
    \end{equation}
    Set $\mathbf{d}' := (\mathbf{0},\mathbf{d})\in\mathbb Z^{\overline V_{\lambda}}$ and $\mathbf{f} := \mathbf{d}'\barK_{\lambda}$.
    From the definition of $\mathbf{d}$, we know $\barK_{\lambda}\barQ_{\lambda}\mathbf{f}^{T} = \bm{0}$ in $\mathbb Z_{m''}$
        We regard $\mathbf{f} = (
        \mathbf{f}_1,\mathbf{f}_2
    )\in \mathbb Z^{V_{\lambda}}$ with $\mathbf{f}_1\in \bZ^{\obVl}$ and $\mathbf{f}_2\in \bZ^{\overline V_{\lambda}\setminus\obVl}$. 
    By replacing $\mathbf{k}_i$ with $\mathbf{f}_i$, $\mathbf{f}_1$ and $\mathbf{f}_2$ satisfy Equation \eqref{eq-reduced-key}. From  Lemma \ref{reduced-K}, we have 
    $\mathbf{f}_2 = \mathbf{d}K_{\partial} = (\mathbf{d}_{\partial_1}S_1,\mathbf{d}_{\partial_2}S_2,\cdots,\mathbf{d}_{\partial_b}S_b)$.

  When $r_i$ is even, we have $\mathbf{d}_{\partial_i}S_i = (2\mathbf{d}_{i}G,\cdots,2\mathbf{d}_{i}G) = ({\bf b}_i,\cdots,{\bf b}_i)$ in $\mathbb Z_{d}$.

Set $\mathbf{h} := \mathbf{f} - \mathbf{k}_0$. Then $\mathbf{h}  =
(\mathbf{h} _1,\mathbf{h} _2)$, where $\mathbf{h}_1 = \mathbf{f}_1 - \mathbf{k}_1$ and $\mathbf{h}_2 = \mathbf{f}_2 - \mathbf{k}_2$. Note that $\mathbf{h}_1$ and $\mathbf{h}_2$ satisfy Equation \eqref{eq-reduced-key}.
We regard $\mathbf{h}_2 =  (\tff_{\partial_1},\tff_{\partial_2},\cdots,\tff_{\partial_b})$,  where $\tff_{\partial_i}\in\mathbb Z^{r_i(n-1)}$ is the row vector associated to $\partial_i$ for each $1\leq i\leq b$.
From the above discussion, we know  
$\tff_{\partial_i} = (\tff'_{i},\cdots,\tff'_{i})$ in $\mathbb Z_{m'}$, where 
$\tff'_{i} \in\mathbb Z^{n-1}$ and $\tff'_{i}=\bm{0}$ in $\mathbb Z_{d}$.
When $r_i$ is odd, we also have $\tff'_i = \overleftarrow{\tff'_i}$ in $\mathbb Z_{m'}$, and  
Lemma \ref{lem-reverse-d} implies there exists
${\bf e}_i\in\mathbb Z^{n-1}$ such that ${\bf e}_i=\bm{0}\text{ in }\mathbb Z_{d}$ and ${\bf e}_i+\overleftarrow{{\bf e}_i} = \tff'_i\text{ in }\mathbb Z_{m'}$.
When $r_i$ is even, define ${\bf e}_i$ to be $\tff'_i$.

By definition, ${\bf e}_i=\bm{0}\text{ in }\mathbb Z_{d}$ for $1\leq i\leq b$.
Using the technique in the proof of Theorem \ref{center_torus}, one can show there exists 
$\mathbf{x}_i''\in \mathbb Z^{n-1}$ such that
$2\mathbf{x}_i''G = {\bf e}_i\text{ in }\mathbb Z_{m'}$ for $1\leq i\leq b$.

Set $\mathbf{x}_2= (\mathbf{x}_{\partial_1},\mathbf{x}_{\partial_2},\cdots,\mathbf{x}_{\partial_b})\in \mathbb Z^{\overline V_{\lambda}\setminus\obVl}$, where $\mathbf{x}_{\partial_i}\in\mathbb Z^{r_i(n-1)}$ is the vector associated to $\partial_i$,  such that 
\begin{equation}
    \mathbf{x}_{\partial_i} = 
        \begin{cases}    (\mathbf{x}_i''+\overleftarrow{\mathbf{x}_i''},\cdots,\mathbf{x}_i''+\overleftarrow{\mathbf{x}_i''})& \text{if $r_i$ is odd},\\
        (\mathbf{x}_i'',\cdots,\mathbf{x}_i'') & \text{if $r_i$ is even}.\\
        \end{cases}
    \end{equation}
Set $\mathbf{x} = (\bm{0},\mathbf{x}_2)\in\mathbb Z^{\overline V_{\lambda}}$, and define $\mathbf{y} = \mathbf{x}\barK_{\lambda}$.

We regard $\mathbf{y} = (\mathbf{y}_1,\mathbf{y}_2)\in \mathbb Z^{\overline V_{\lambda}}$ with $\mathbf{y}_1\in \bZ^{\obVl}$ and $\mathbf{y}_2\in \bZ^{\overline V_{\lambda}\setminus\obVl}$.
By the definition of $\mathbf{x}$, $\mathbf{y}_1$ and $\mathbf{y}_2$ satisfy Equation \eqref{eq-reduced-key}, and $\mathbf{y}_2-\mathbf{h}_2 =\bm{0}$ in $\mathbb Z_{m'}$. We have $\mathbf{y} - \mathbf{h} = (\mathbf{y}_1 - \mathbf{h}_1,\mathbf{y}_2 - \mathbf{h}_2)$ satisfies Equation \eqref{eq-reduced-key1}, especially $-2(\mathbf{y}_1 - \mathbf{h}_1)^T+ \frac{1}{n}P' (\mathbf{y}_2 - \mathbf{h}_2)^T=\bm{0}\text{ in }\mathbb Z_{m'}$.
Lemma \ref{lem:reduced_P_n} and $\mathbf{y}_2-\mathbf{h}_2 =\bm{0}$ in $\mathbb Z_{m'}$ imply 
$-2(\mathbf{y}_1 - \mathbf{h}_1)^T=\bm{0}\text{ in }\mathbb Z_{m'}$.
Then we have $\mathbf{y}_1 - \mathbf{h}_1 =\bm{0}$ in $\mathbb Z_{m'}$ because $m'$ is odd. 
This implies $\mathbf{y} - \mathbf{h} =\bm{0}$ in $\mathbb Z_{m'}$.
Since both of $\mathbf{y}$ and $\mathbf{h}$ are balanced, then 
$\mathbf{y} - \mathbf{h} = \mathbf{z}\barK_{\lambda}$ for some 
$\mathbf{z}\in \mathbb Z^{\overline V_{\lambda}}$. We have $\mathbf{z}\in\overline \Lambda_{m'}$ since $\mathbf{y} - \mathbf{h} =\bm{0}$ in $\mathbb Z_{m'}$.

We have $\mathbf{h} = \mathbf{f}-\mathbf{k}_0 = \mathbf{d}'\barK_{\lambda}- \mathbf{t}_0\barK_{\lambda}$, and 
$\mathbf{h} = \mathbf{y}-\mathbf{z}\barK_{\lambda} = \mathbf{x}\barK_{\lambda}- \mathbf{z}\barK_{\lambda}$. We also have $\mathbf{t}_0=
\mathbf{d}'-\mathbf{x}+\mathbf{z}$ since $\barK_{\lambda}$ is invertible shown in Lemma~\ref{lem:invertible_KH}.
Then $\mathbf{t}_0\in \overline\Lambda_z$ from $\mathbf{d}',\mathbf{x}\in\overline\Lambda_{\partial}$ and $\mathbf{z}\in\overline  \Lambda_{m'}$.
\end{proof}

Like in \eqref{eq-Weyl-normalization-g}, we can define $\bar\gaa^{{\bf k}}$ for ${\bf k}\in V_\lambda$ when $\overline{\text{tr}}_\lambda^{A}$ is injective.
Let $\overline{\cY}_{\lambda}$ denote the subalgebra of $\overline\cS_n(\Sigma,\mathbbm{v})$ generated by $\{\bar\gaa^{\mathbf{k}}\mid \textbf k \in \overline\Lambda_{m'}^{+}\}$ and 
the central elements in Lemma \ref{reduced-boundary_center}.

\begin{thm}\label{main-thm-reduced-center}
Assume that $m''$ is odd and $\hat{q}^2$ is a primitive $m''$-th root of unity. Suppose $\Sigma$ is an essentially bordered pb surface and contains no interior punctures, and $\lambda=\mu$ is a triangulation of $\Sigma$ introduced in Section~\ref{sub:quantum-torus-reduced}. We require $\Sigma$ to be a polygon if $n>3$.
 We have 
$$\mathcal{Z}(\overline\cS_n(\Sigma,\mathbbm{v})) = \{x\in \overline\cS_n(\Sigma,\mathbbm{v})\mid \text{$\exists  \mathbf{k}\in \mathbb{N}^{\overline V_{\lambda}}  $  such that  $ \bar\gaa^{m'\mathbf{k}}x\in\overline{\mathcal Y}_{\lambda}$}\}.$$

\end{thm}
\begin{proof}
    The proof for Theorem \ref{thm-main-center-skein} works here.
\end{proof}

\subsection{PI-degree}

\begin{prop}\label{thm-Unicity-reduced}
Suppose $\mathbbm{v}$ is a root unity and $\Sigma$ is an essentially bordered pb surface.
If $n=2,3$ (possibly with interior punctures), or $n>3$ and $\Sigma$ is a polygon, then
$\overline\cS_n(\Sigma,\mathbbm{v})$ is almost Azumaya.
\end{prop}
\begin{proof}
    Proposition \ref{thm;azumaya} implies $\cS_n(\Sigma,\mathbbm{v})$ satisfies conditions (1) and (3) in Definition \ref{def:almost_Azumaya}, which implies $\overline\cS_n(\Sigma,\mathbbm{v})$ also satisfies these two conditions. It suffices to show condition (2). Then Theorem \ref{traceA} and \cite[Theorem 15.5]{LY23} complete the proof.
\end{proof}

\begin{prop}\label{reduced-rank_eq}
If the algebra homomorphism $\overline{tr}_{\lambda}^A$ in Theorem \ref{traceA} is injective, we have $\rankZ\rA = \rankZ\overline\cS_n(\Sigma,\mathbbm{v})$. 
\end{prop}
\begin{proof}
The proof of Proposition \ref{rank_eq} also works here. 
\end{proof}

The aim of the remaining of this subsection is to compute $\rankZ\rA$. We will prepare some lemmas first.
\begin{lem}\label{reduced-exact}
Suppose $k$ is a positive integer and $\gcd(k,n) = l$. Set $N = kn/l$.
   Then there is a short exact sequence 
$$0\rightarrow N\mathbb Z^{\overline V_{\lambda}}\xrightarrow{ 
 L}\overline \Lambda_{\lambda}\cap k\mathbb Z^{\barV_{\lambda}} \xrightarrow{J} Z^1(\overline \Sigma,\mathbb Z_n)_l \rightarrow 0,$$
 where $L$ is the natural embedding and $Z^1(\overline \Sigma,\mathbb Z_n)_l=l(C^1(\overline\Sigma,\mathbb Z_n))\cap Z^1(\overline \Sigma,\mathbb Z_n)$ and $\overline \Lambda_\lambda$ is the balanced part.
 \end{lem}
 \begin{proof}
     The proof for Lemma \ref{lem5.3} works here. 
 \end{proof}

 \begin{lem}\label{lem-PI-1}
    We have $$\left|\displaystyle \frac{\overline \Lambda_{\lambda}}{\overline\Lambda_{\lambda}\cap m'\mathbb Z^{\overline V_{\lambda}}}\right| =d^{r(\Sigma)}m^{|\overline V_{\lambda}|}.$$
\end{lem}
\begin{proof}
    Using Lemma \ref{reduced-exact} and the technique in Lemma \ref{lem5.3}, we can show the claim.
\end{proof}

\begin{rem}\label{rem-G}
    Suppose ${\bf k}\in\mathbb Z^{n-1}$ such that ${\bf k}=\cev{{\bf k}}$.
Lemma \ref{lem-G-reverse} implies $${\bf k}(G+G') = {\bf k}G+{\bf k}G'
= {\bf k}G+\cev{{\bf k}}G = 2{\bf k}G.$$
\end{rem}

\begin{rem}\label{reduced-rem_partial}
Proposition~\ref{prop:LY23_11.10} implies there is a group isomorphism $\overline \varphi\colon\mathbb Z^{\overline V_{\lambda}}\rightarrow \overline\Lambda_{\lambda}$, defined by $\overline\varphi(\textbf{k}) = \textbf{k}\barK_{\lambda}$ for $\textbf{k}\in \mathbb Z^{\overline V_{\lambda}}$.
Then $\left|\dfrac{\mathbb Z^{\overline V_{\lambda}}}{\overline\Lambda_z}\right| =
\left|\dfrac{\overline\Lambda_{\lambda}}{\overline\varphi(\overline\Lambda_z)}\right|.$
Recall $\overline\Lambda_z = \overline\Lambda_{m'}+\langle\overline\Lambda_{\partial}\rangle,$ where $\langle\overline\Lambda_{\partial}\rangle$ is the subgroup of $\mathbb Z^{\overline V_{\lambda}}$ generated by $\overline\Lambda_{\partial}$. From the definition of $\overline\Lambda_{m'}$, we have $\overline\varphi(\overline\Lambda_z) = (\overline\Lambda_{\lambda}\cap m'\mathbb Z^{\overline V_{\lambda}})+\overline\varphi(\langle\overline\Lambda_{\partial}\rangle)$. 

Let us understand $\varphi(\langle \overline\Lambda_{\partial}\rangle)$. 
Suppose $\partial_1,\cdots,\partial_t$ (resp. $\partial_{t+1},\cdots,\partial_b$) are boundary components with even (resp. odd )number of boundary punctures. 
For any element $\textbf{k}\in \langle\overline \Lambda_{\partial}\rangle$,
we write $\mathbf{k} = (\tfk_1,\tfk_2)\in \bZ^{\overline V_{\lambda}}$ with $\tfk_1\in \bZ^{\obVl}$ and $\tfk_2\in \bZ^{\overline V_{\lambda}\setminus\obVl}$. 
Set $\textbf{d}_i:=\textbf{k}_2|_{\partial_i}$ for $1\leq i\leq b$.
Then each $\textbf{d}_i:=\textbf{k}_2|_{\partial_i} = ({\bf c}_i,\cdots,{\bf c}_i)$  where ${\bf c}_i\in\bZ^{n-1}$.
We have ${\bf c}_i = \overleftarrow{{\bf c}_i}$ for $t+1\leq i\leq b$.

 Suppose $\overline \varphi(\tfk) = (\textbf{u}_1,\textbf{u}_2)$, where $\textbf{u}_1\in \bZ^{\obVl}$ and $\textbf{u}_2\in \bZ^{\overline V_{\lambda}\setminus\obVl}$. Then 
$(\textbf{u}_1,\textbf{u}_2)$ satisfies Equation \eqref{eq-reduced-key} because $a^{\tfk}$ is in the center of $\rA$.

Equation \eqref{eq;matrix_P} and Lemma \ref{reduced-K} imply 
$\textbf{u}_2 = (\textbf{d}_1S_1,\cdots,\textbf{d}_bS_b)$.
Lemma \ref{reduced-K} and Remark \ref{rem-G} imply $\textbf{d}_iS_i= (2{\bf e}_i G,\cdots,2{\bf e}_i G)$.
\end{rem}

For any real number $k$, we use
$\lfloor k\rfloor$ to denote $\text{max}\{t\in\mathbb Z\mid t\leq k\}$.

\begin{lem}\label{reduced-PI2}
Suppose $\overline\Sigma$ contains $t$ even boundary components. 
Then we have 
$$\left|\dfrac{\overline\varphi(\overline\Lambda_z)}{\overline\Lambda_{\lambda}\cap m'\mathbb Z^{\overline V_{\lambda}}}\right|=\left| \frac{(\overline\Lambda_{\lambda}\cap m'\mathbb Z^{\overline V_{\lambda}})+\overline\varphi(\langle\overline\Lambda_{\partial}\rangle)}{\overline\Lambda_{\lambda}\cap m'\mathbb Z^{\overline V_{\lambda}}}\right|=(m')^tm^{t(n-2)+(b-t)\lfloor \frac{n}{2}\rfloor}=d^tm^{t(n-1)+(b-t)\lfloor \frac{n}{2}\rfloor}.$$
\end{lem}
\begin{proof}

Obviously, we have 
$$\left| \frac{(\overline\Lambda_{\lambda}\cap m'\mathbb Z^{\overline V_{\lambda}})+\overline\varphi(\langle\overline\Lambda_{\partial}\rangle)}{\overline\Lambda_{\lambda}\cap m'\mathbb Z^{\overline V_{\lambda}}}\right| =
\left| \{x+m' \mathbb Z^{\overline V_{\lambda}}\mid x\in \varphi(\langle\overline\Lambda_{\partial}\rangle)\}\right|,$$
where $\{x+m' \mathbb Z^{\overline V_{\lambda}}\mid x\in \varphi(\langle\overline\Lambda_{\partial}\rangle)\}$ is a subset of $\frac{\mathbb Z^{\overline V_{\lambda}}}{m'\mathbb Z^{\overline V_{\lambda}}}$ (it is actually a subgroup of $\frac{\mathbb Z^{\overline V_{\lambda}}}{m'\mathbb Z^{\overline V_{\lambda}}}$). 

In the rest of the proof, we follow the notations in Remark \ref{reduced-rem_partial}.
Set $$\overleftarrow{\mathbb Z_{m'}^{n-1}} =
\{{\bf k}\in \mathbb Z_{m'}^{n-1}\mid {\bf k}=\cev{{\bf k}}\in \mathbb Z_{m'}^{n-1}\}.$$

Define $$\nu\colon\mathbb Z_{m'}^{n-1}\rightarrow \mathbb Z_{m'}^{n-1},\;\nu(\textbf{p}) = 2\textbf{p} G\text{ and }
\cev{\nu}\colon\overleftarrow{\mathbb Z_{m'}^{n-1}}\rightarrow \overleftarrow{\mathbb Z_{m'}^{n-1}},\;\cev{\nu}(\textbf{p}) = 2\textbf{p}G.$$ For any element $\textbf{u}=(\textbf{u}_1,\textbf{u}_2)\in \overline\varphi(\langle\overline\Lambda_{\partial}\rangle)$, we have $\textbf{u}_2 = (\textbf{d}_1S_1,\cdots,\textbf{d}_bS_b)$, where $\textbf{d}_iS_i= (2{\bf e}_i G,\cdots,2{\bf e}_i G)$.
Note that ${\bf e}_i\in\mathbb Z^{n-1}$ for $1\leq i\leq b$ and 
${\bf e}_j = \overleftarrow{{\bf e}_j}$ for $t+1\leq j\leq b$. 

Define $\overline\theta(\textbf{u}) = (2{\bf e}_1 G,\cdots,2{\bf e}_b G)\in (\im\nu)^t\times (\im\cev\nu)^{b-t}$. It is easy to see $\overline\theta$ is a well-defined surjective group homomorphism from $\{x+m' \mathbb Z^{\overline V_{\lambda}}\mid x\in \varphi(\langle\overline\Lambda_{\partial}\rangle)\}$ to $(\im\nu)^t\times (\im\cev\nu)^{b-t}$. 
As in the proof of Lemma \ref{lem5.10}, one can show $\overline\theta$ is an isomorphism.
Then Lemmas \ref{lem;Im_mu} and \ref{lem-reduced-mu} complete the proof.
\end{proof}

\begin{lem}\label{lem-reduced-mu}
    For $\cev\nu\colon\overleftarrow{\mathbb Z_{m'}^{n-1}}\rightarrow \overleftarrow{\mathbb Z_{m'}^{n-1}},\ \mathbf{p}\mapsto 2\mathbf{p} G$, we have $|\im\cev\nu| = m^{\lfloor\frac{n}{2}\rfloor}$.
\end{lem}
\begin{proof}
    Since $2$ is invertible in $\mathbb Z_{m'}$, we regard $\cev\nu$ as a map defined by $\cev{\nu}(\mathbf{p})=\mathbf{p} G$. To get $|\im\cev\nu|$, we investigate $\ker\cev\nu$.

    Suppose $\{{\bf c}_1,\cdots,{\bf c}_{n-1}\}$ is the standard basis for $\mathbb Z^{n-1}$, i.e., all the entries of each ${\bf c}_i$ are zeros except the $i$-th one being $1$.

    Suppose $n$ is even. Define 
    ${\bf a}_i = {\bf c}_i+{\bf c}_{i+1}+\cdots+{\bf c}_{n-i-1}+{\bf c}_{n-i}$ for $1\leq i\leq \frac{n-1}{2}$, and ${\bf a}_{\frac{n}{2}}
    ={\bf c}_{\frac{n}{2}}.$ 
    Recall that we write $G=(G_1^T,\cdots,G_{n-1}^T)^T$ and $\overleftarrow{G_i} = G_{n-i}$ for $1\leq i\leq n-1$.
    It is easy to check 
    $$G_{i}+G_{n-i} = n({\bf a}_1+\cdots+{\bf a}_i)\text{ for } 1\leq i\leq \frac{n}{2}-1,\text{ and } G_{\frac{n}{2}} = \frac{n}{2}
    ({\bf a}_1+\cdots+{\bf a}_{\frac{n}{2}}).$$

    For any ${\bf b} = (b_1,\cdots,b_{n-1})\in\overleftarrow{\mathbb Z_{m'}^{n-1}}$, we have 
    \begin{equation}\label{eq-Zm}
    \begin{split}
        &\cev\nu({\bf b})\\ = &b_1G_1+\cdots+b_{n-1} G_{n-1}
        =b_1(G_1+ G_{n-1}) +\cdots+ b_{\frac{n}{2}-1}(G_{\frac{n}{2}-1}+ G_{\frac{n}{2}+1})+ b_{\frac{n}{2}} G_{\frac{n}{2}}\\
        =&n b_1{\bf a}_1+ n b_2({\bf a}_1 +{\bf a}_2)+\cdots
        +nb_{\frac{n}{2}-1}({\bf a}_1 +{\bf a}_2+\cdots+{\bf a}_{\frac{n}{2}-1})+\frac{n}{2}b_{\frac{n}{2}}({\bf a}_1 +{\bf a}_2+\cdots+{\bf a}_{\frac{n}{2}})\\
        =&\frac{n}{2}\big( 2b_1{\bf a}_1+ 2b_2({\bf a}_1 +{\bf a}_2)+\cdots
        +2b_{\frac{n}{2}-1}({\bf a}_1 +{\bf a}_2+\cdots+{\bf a}_{\frac{n}{2}-1})+b_{\frac{n}{2}}({\bf a}_1 +{\bf a}_2+\cdots+{\bf a}_{\frac{n}{2}})\big).
        \end{split}
    \end{equation}
    Since $m'$ is odd, then $\gcd(n,m') = \gcd(\frac{n}{2},m')=d$.
    Note that we regard $Z_{m}^{n-1}$ as a module over $\mathbb Z_{m}$.
    Obviously ${\bf a}_1, {\bf a}_1 +{\bf a}_2,\cdots, {\bf a}_1 +{\bf a}_2+\cdots+{\bf a}_{\frac{n}{2}}$ are $\mathbb Z_m$-linearly independent over $\mathbb Z_{m}$.
    Then \eqref{eq-Zm} implies 
    \begin{equation}\label{ker-mu}
        \begin{split}
             &\cev\nu({\bf b})=\bm{0}\in\mathbb Z_{m'}^{n-1}\\
             \Leftrightarrow&
             2b_1{\bf a}_1+ 2b_2({\bf a}_1 +{\bf a}_2)+\cdots
        +2b_{\frac{n}{2}-1}({\bf a}_1 +{\bf a}_2+\cdots+{\bf a}_{\frac{n}{2}-1})+\\&b_{\frac{n}{2}}({\bf a}_1 +{\bf a}_2+\cdots+{\bf a}_{\frac{n}{2}})=\bm{0}\in\mathbb Z_{m}^{n-1}\\
         \Leftrightarrow& b_1=b_2=\cdots=b_{\frac{n}{2}} = 0\in\mathbb Z_{m}.
        \end{split}
    \end{equation}
    
    Define $f\colon\overleftarrow{\mathbb Z_{m'}^{n-1}}\rightarrow \mathbb Z_{m}^{\frac{n}{2}}$ such that $f({\bf b}) = (b_1,\cdots,b_{\frac{n}{2}})$ for any ${\bf b}=(b_1,\cdots,b_{n-1})\in \overleftarrow{\mathbb Z_{m'}^{n-1}}$.
    Then \eqref{ker-mu} implies $\ker\cev{\nu} = \ker f$.
    Define $g\colon\mathbb Z_{m'}^{\frac{n}{2}}\rightarrow \overleftarrow{\mathbb Z_{m'}^{n-1}}$ such that
    $g({\bf d}) = (d_1,\cdots,d_{\frac{n}{2}-1},d_{\frac{n}{2}},d_{\frac{n}{2}-1},\cdots,d_1)$ for any ${\bf d}=(d_1,\cdots,d_{\frac{n}{2}-1},d_{\frac{n}{2}})\in \mathbb Z_{m'}^{\frac{n}{2}}$. We have $g$ is a group isomorphism. 
    It is easy to check that $f\circ g$ is the projection from 
    $\mathbb Z_{m'}^{\frac{n}{2}}$ to $\mathbb Z_{m}^{\frac{n}{2}}$.
    Then $$|\ker\cev{\nu}|=|\ker f |=|\ker(f\circ g)| = d^{\frac{n}{2}}.$$
    Thus we have 
    $$|\im\cev{\nu}|=\left|\frac{\overleftarrow{\mathbb Z_{m'}^{n-1}}}{\ker\cev\nu}\right|=\left|\frac{\mathbb Z_{m'}^{\frac{n}{2}}}{\ker\cev\nu}\right|
    =\frac{(m')^{\frac{n}{2}}}{(d)^{\frac{n}{2}}}=(m)^{\frac{n}{2}}
    =m^{\lfloor\frac{n}{2}\rfloor}.$$

    Using the same technique as above, one can show $|\im\cev\nu| = 
    m^{\frac{n-1}{2}}=m^{\lfloor\frac{n}{2}\rfloor}$ when $n$ is odd.
    
 \end{proof}

\begin{thm}\label{thm-PI-reducedA}
Assume that $m''$ is odd and $\hat{q}^2$ is a primitive $m''$-th root of unity.
Let $\Sigma$ be a triangulable essentially
bordered pb surface without interior punctures, and $\lambda=\mu$ be a triangulation of $\Sigma$
introduced in Section~\ref{sub:quantum-torus-reduced}.
Suppose $\overline\Sigma$ has $b$ boundary components among which there are $t$ even boundary components. Then we have 
$$\rankZ \rA=d^{r(\Sigma)-t} m^{|\overline V_{\lambda}|-t(n-1)-(b-t)\lfloor\frac{n}{2}\rfloor},$$
where $d,m$ are defined in Section~\ref{notation} and $|\overline V_{\lambda}|=(n^2-1)r(\Sigma)-\binom{n}{2}(\#\partial\Sigma)$ given in \eqref{eq:cardinarity}. 
\end{thm}
\begin{proof}
    We can show the claim using Lemmas \ref{lem-PI-1} and \ref{reduced-PI2} and the same technique as Theorem \ref{thm:rank}. 
\end{proof}

\begin{thm}\label{main-thm-reduced-PI}
Assume that $m''$ is odd and $\hat{q}^2$ is a primitive $m''$-th root of unity. 
Suppose $\Sigma$ is an essentially bordered pb surface and contains no interior punctures. We require $\Sigma$ to be a polygon if $n>3$.
Assume $\overline\Sigma$ has $b$ boundary components among which there are $t$ even boundary components.
Then we have 
$$\rankZ\overline\cS_n(\Sigma,\mathbbm{v}) =\rankZ \rA= d^{r(\Sigma)-t} m^{|\overline V_{\lambda}|-t(n-1)-(b-t)\lfloor\frac{n}{2}\rfloor},$$
where $d,m$ are defined in Section~\ref{notation} and $|\overline V_{\lambda}|=(n^2-1)r(\Sigma)-\binom{n}{2}(\#\partial\Sigma)$  given in \eqref{eq:cardinarity}.  
\end{thm}
\begin{proof}
Theorems \ref{traceA} and \ref{thm-PI-reducedA} and Proposition \ref{reduced-rank_eq} concludes the claim.
\end{proof}

\section{The Azumaya locus of the stated $\SL(n)$-skein algebra}\label{sec:Azumaya}
In this section, we suppose Condition ($*$). All the algebras are assumed to be $\mathbb C$-algebras.

Recall that, for an almost  Azumaya algebra $A$, $\Azumaya (A)$ denotes the Azumaya locus of $A$.

\def\TP{\mathbb T(\mathsf{P})}

\begin{thm}\cite[Theorem 4.1]{New72}\label{thm-decom-symmetric}
    Suppose $\mathsf{P}$ is an anti-symmetric integer matrix of size $r$.
    There exists an integral matrix $\mathsf{X}$ with 
    $\det\mathsf{X} =\pm 1$ such that 
    \begin{align}\label{eq-decom-P}
        \mathsf{X}^T \mathsf{P}\mathsf{X} = 
    \text{diag}\left\{
    \begin{pmatrix}
        0   & h_1\\
        -h_1& 0
    \end{pmatrix},\cdots,
    \begin{pmatrix}
        0   & h_k\\
        -h_k& 0
    \end{pmatrix},
    0,\cdots,0
    \right\},
    \end{align}
    where $h_i\mid h_{i+1}$ for $1\leq i\leq k-1$.
\end{thm}

We refer to the decomposition in \eqref{eq-decom-P} as the {\bf anti-symmetric matrix decomposition}.

\begin{thm}\cite[Theorem 2.26]{EG11}\label{thm-tensor-ir}
    (1) Let $V$ (resp. $W$) be a finite-dimensional irreducible representation of the
    algebra $A$ (resp. $B$). Then $V\ot W$ is a finite-dimensional irreducible representation of the
    algebra $A\ot B$.

    (2) Any finite-dimensional irreducible representation $M$ of the
    algebra $A\ot B$ has the form (1) for unique $V$ and $W$.
\end{thm}

For any integer $h$, we define 
$$\mathbb C[X,Y;h] = \mathbb C\langle X^{\pm1},Y^{\pm1}\rangle/(XY=\hat q^{2h} YX).$$

Recall that $\hat q^2$ is a root of unity of order $m''$.
Define $l=\frac{m''}{\gcd(m'',h)}$. We have the following.

\begin{lem}\cite[Lemma 17]{BL07}\label{lem-ir-CXY}
    Any irreducible representation of $\mathbb C[X,Y;h]$ has dimension $l$.
\end{lem}

Although the following lemma is probably well-known, we do not know where it is written explicitly. Hence, let us give a proof here.
\begin{lem}\label{lem-quantum-torus-rep}
    Suppose $\mathsf{P}$ is an anti-symmetric integer matrix.
    Any irreducible representation of $\TP$ is an Azumaya representation.
\end{lem}
\begin{proof}
    Suppose $\mathsf P$ is of size $r$ and has the  anti-symmetric matrix decomposition
    as \eqref{eq-decom-P}. 
    Then we have
    $$\TP\simeq C[X,Y;h_1]\ot\dots \ot C[X,Y;h_k]\ot(\mathbb C[X]^{\ot(r-2k)}).$$
    For each $1\leq i\leq k$, suppose $l_i = \dfrac{m''}{\gcd(m'',h_i)}$.
    From Lemma~\ref{PI}, it is easy to show that 
    \begin{align}\label{eq-rank-torus-decom}
        \rankZ(\TP) = (l_1l_2\cdots l_k)^2.
    \end{align}
    Theorem \ref{thm-tensor-ir} and Lemma \ref{lem-ir-CXY} imply any irreducible representation of $\TP$ has dimension $l_1l_2\cdots l_k$, which equals the square root of $\rankZ(\TP)$.
    This completes the proof.
\end{proof}

Suppose $\mathsf{P}$ is an anti-symmetric integer matrix, and 
    $A$ is a $\mathbb C$-algebra such that $\mathbb T^{+}(\mathsf{P})\subset A\subset \mathbb T(\mathsf{P})$ (sandwiched property).
    Obviously, we have $\mathcal Z(A)\subset \mathcal Z(\mathbb T(\mathsf{P}))$. Lemma \ref{lem-quantum-torus-rep} implies 
    the Azumaya locus of $\TP$ is
    $\text{MaxSpec}(\mathcal Z(\TP))$.
    We use $\iota$ to denote the embedding from $\mathcal Z(A)$ to 
    $\mathcal Z(\mathbb T(\mathsf{P}))$. Then $\iota$ induces a (continuous) map $\iota^*\colon \text{MaxSpec}(\mathcal Z(\TP))\rightarrow \text{MaxSpec}(\mathcal Z(A))$ given by $\iota^*(\gamma) = \iota^{-1}(\gamma)$ for any $\gamma\in \text{MaxSpec}(\mathcal Z(A))$.

The proof of Proposition \ref{rank_eq} implies the following.
\begin{lem}\label{lem-PI-torus-equal-A}
    Suppose $\mathsf{P}$ is an anti-symmetric integer matrix, and 
    $A$ is a finitely generated $\mathbb C$-algebra such that $\mathbb T^{+}(\mathsf{P})\subset A\subset \mathbb T(\mathsf{P})$.
     Then we have $A$ is  almost Azumaya and 
     the PI-degree of $\mathbb T(\mathsf{P})$ equals that of $A$.
\end{lem}

\begin{thm}\label{thm-torus-Amz}
    Suppose $\mathsf{P}$ is an anti-symmetric integer matrix, and 
    $A$ is a finitely generated $\mathbb C$-algebra such that $\mathbb T^{+}(\mathsf{P})\subset A\subset \mathbb T(\mathsf{P})$.
    Then we have the following:
\begin{enumerate}
    \item[$(a)$] any irreducible representation of $\TP$ pulls back to an Azumaya representation of $A$,
    \item[$(b)$]  we have $\im \iota^*\subset\Azumaya (A)$. 
\end{enumerate}
\end{thm}
\begin{proof}
    (a) Suppose $\rho\colon \TP\rightarrow\text{End}_{\mathbb C}(V)$ is an irreducible representation of $\TP$. Then $\rho$ induces a $\bC$-algebra homomorphism $r_\rho\colon\mathcal Z(\TP)\rightarrow\mathbb C$.
    We use $(\ker r_\rho)$ to denote the ideal of $\TP$ generated by
    $\ker r_\rho$.  Lemma \ref{lem-ir-CXY} implies $\rho$ is an Azumaya representation of $\TP$. Then $\rho$ induces a $\bC$-algebra isomorphism
    $$\bar\rho\colon \TP/(\ker r_\rho)\rightarrow \text{End}_{\mathbb C}(V).$$
    We use $\bar L$ to denote the composition 
    $A\rightarrow\TP\rightarrow \TP/(\ker r_\rho)$, where 
    the first map is the injection and the second map is the canonical projection. 
    The pullback of $\rho$ to $A$ is the composition of the following maps
    $$A\xrightarrow{\bar L} \TP/(\ker r_\rho)\xrightarrow{\bar\rho}\text{End}_{\mathbb C}(V).$$
    Lemma \ref{lem-PI-torus-equal-A} implies $\dim_{\mathbb C} V$ equals the PI-degree of $A$. It suffices to show that $\bar L$ is surjective.

    For any element $x\in\TP$, we use $\bar x$ to the image of $x$ under the projection from $\TP$ to $\TP/(\ker r_\rho)$.
    We use $\bar B$ to denote the subalgebra of $\TP/(\ker r_\rho)$ generated by $\bar x_i$ for $1\leq i\leq r$.
    Since $\mathbb T^{+}(\mathsf{P})\subset A$, then $\bar B\subset \im\bar L$.
    Here $r$ is the size of the matrix $\mathsf P$. 
    For any $1\leq i\leq r$, we have $x_i^{m''}\in\mathcal Z(\TP)$.
    Then $r_\rho(x_i^{m''}) = c_i\in\mathbb C$. We have $c_i\neq 0$ since 
    $x_i^{m''}$ is invertible. Then $\bar x_i^{-1} = c_i^{-1} \bar x_i^{m''-1}\in\bar B$.  This shows $\TP/(\ker r_\rho)\subset\bar B$. Thus we have $\TP/(\ker r_\rho)\subset\im\bar L$.

    (b) comes from Lemma \ref{lem-quantum-torus-rep} and (a).
\end{proof}

\def\trra{\overline{\text{tr}}_{\lambda}^A}

We apply the above discussions to stated ${\rm SL}(n)$-skein algebras. 
Suppose $\Sigma$ is an essentially bordered pb surface and contains no interior punctures and let $\lambda$ be a triangulation of $\Sigma$. 
The quantum trace maps 
$$\tra\colon \cS_{n}(\Sigma,\mathbbm{v})\rightarrow \A,\qquad
 \trra \colon \overline \cS_{n}(\Sigma,\mathbbm{v})\rightarrow \rA$$
 induce algebra homomorphisms
 $$\mu\colon \mathcal Z(\cS_{n}(\Sigma,\mathbbm{v}))\rightarrow \mathcal Z(\A),\qquad
 \bar\mu \colon \mathcal Z(\overline \cS_{n}(\Sigma,\mathbbm{v}))\rightarrow \mathcal Z(\rA).$$

Theorems \ref{traceA} and \ref{thm-torus-Amz} imply the following corollaries. 

\begin{cor}
    Suppose $\Sigma$ is an essentially bordered pb surface and contains no interior punctures. Then we have the following.

    (a) any irreducible representation of $\A$ pulls back to 
    an Azumaya representation of $\cS_{n}(\Sigma,\mathbbm{v})$;

    (b) we have $\im \mu^*\subset\Azumaya (\cS_{n}(\Sigma,\mathbbm{v}))$. 
\end{cor}

\begin{cor}
    Suppose $\Sigma$ is an essentially bordered pb surface and contains no interior punctures. We require $\Sigma$ to be a polygon if $n>3$.
    Then we have the following.

    (a) any irreducible representation of $\rA$ pulls back to 
    an Azumaya representation of $\overline \cS_{n}(\Sigma,\mathbbm{v})$;

    (b) we have $\im \bar\mu^*\subset\Azumaya (\overline \cS_{n}(\Sigma,\mathbbm{v}))$. 
\end{cor}

For any two integers $n_1,n_2$, we write   
$n_1\overset{(2)}{=} n_2$ if $n_1 = 2^k n_2$ for some integer $k$. Then for each integer $l$, there exists a unique odd integer $\mathsf{Odd}(l)$ such that $l\overset{(2)}{=} \mathsf{Odd}(l).$

\begin{lem}\label{lem-2power-eq}
    Let $\{n_1,\cdots,n_k\}$ and 
    $\{m_1,\cdots,m_k\}$ be sequences of positive integers such that $n_i\mid n_{i+1}$ and $m_i\mid m_{i+1}$ for $1\leq i\leq k-1$.
    Suppose $\prod_{1\leq i\leq k}\frac{l}{\gcd(l,n_i)} = \prod_{1\leq i\leq k}\frac{l}{\gcd(l,m_i)}$ for any odd positive integer $l$. Then
    $m_i\overset{(2)}{=} n_i$ for each $1\leq i\leq k$.
\end{lem}
\begin{proof}
Note that $\prod_{1\leq i\leq k}\frac{l}{\gcd(l,n_i)} = \prod_{1\leq i\leq k}\frac{l}{\gcd(l,m_i)}$ if and only if 
$\prod_{1\leq i\leq k}\gcd(l,n_i) = \prod_{1\leq i\leq k}\gcd(l,m_i)$.

We prove the claim using induction on $k$.
When $k=1$, the claim is true by setting $l = \mathsf{Odd}(n_1)\mathsf{Odd}(m_1).$
Suppose that the claim holds for $k-1$ ($k>1$). 
Note that $n_i\mid n_{i+1}$ implies 
$\mathsf{Odd}(n_i)\mid \mathsf{Odd}(n_{i+1})$ 
for each $1\leq i\leq k-1$. The same is true for $m_i$. We have 
$$\prod_{1\leq i\leq k} \mathsf{Odd}(n_i)=\prod_{1\leq i\leq k} \mathsf{Odd}(m_i)$$
by setting $l=\mathsf{Odd}(n_k)\mathsf{Odd}(m_k)$. By setting $l=\mathsf{Odd}(n_k)$, we have 
$$\prod_{1\leq i\leq k} \mathsf{Odd}(n_i)
=\prod_{1\leq i\leq k}\gcd(\mathsf{Odd}(n_k),m_i)
=\prod_{1\leq i\leq k}\gcd(\mathsf{Odd}(n_k),\mathsf{Odd}(m_i)).$$
Thus we have $\gcd(\mathsf{Odd}(n_k),\mathsf{Odd}(m_i)) = \mathsf{Odd}(m_i)$ 
for $1\leq i\leq k$.
This implies that $\mathsf{Odd}(m_k)\mid \mathsf{Odd}(n_k)$. Similarly, we have 
$\mathsf{Odd}(n_k)\mid \mathsf{Odd}(m_k)$. 
Hence, we have $\mathsf{Odd}(n_k)= \mathsf{Odd}(m_k)$ and $n_k \overset{(2)}{=} m_k$. 
From $n_k \overset{(2)}{=} m_k$, it follows that  
$\prod_{1\leq i\leq k-1}\frac{l}{\gcd(l,n_i)} = \prod_{1\leq i\leq k-1}\frac{l}{\gcd(l,m_i)}$ for any odd positive $l$. From the assumption on the induction, we have 
$n_i \overset{(2)}{=} m_i$ for $1\leq i\leq k-1$. 
These imply that $n_i \overset{(2)}{=} m_i$ for $1\leq i\leq k$. 
\end{proof}

We introduced anti-symmetric matrices $\mathsf{P}_\lambda$
and $\overline{\mathsf{P}}_\lambda$ in \eqref{eq-anti-matric-P-def}. 
From Lemma~\ref{lem:invertible_KH}, their entries are in $n\mathbb Z$.
From Theorem \ref{thm-decom-symmetric}, we can suppose the
anti-symmetric matrix decomposition of $\mathsf{P}_\lambda$ is 
\begin{align}\label{anti-decom-P}
\text{diag}\left\{
    \begin{pmatrix}
        0   &n z_1\\
        -nz_1& 0
    \end{pmatrix},\cdots,
    \begin{pmatrix}
        0   & nz_k\\
        -nz_k& 0
    \end{pmatrix},
    0,\cdots,0
    \right\}
\end{align} and 
the
anti-symmetric matrix decomposition of $\overline{\mathsf{P}}_\lambda$ is 
\begin{align}\label{anti-decom-barP}
\text{diag}\left\{
    \begin{pmatrix}
        0   &n \bar z_1\\
        -n\bar z_1& 0
    \end{pmatrix},\cdots,
    \begin{pmatrix}
        0   & n\bar z_l\\
        -n\bar z_l& 0
    \end{pmatrix},
    0,\cdots,0
    \right\}.
\end{align}
\begin{prop}\label{prop-anti-decom-P}

(a) Under the assumption of Theorem \ref{thm:rank}, with the
anti-symmetric matrix decomposition of $\mathsf{P}_\lambda$ as in \eqref{anti-decom-P}, we have 
$$z_i\overset{(2)}{=}
\begin{cases}
1 &\text{for $1\leq i\leq \dfrac{r(\Si)-t}{2}$},\medskip\\
n &\text{for $\dfrac{r(\Si)-t}{2}+1\leq i\leq \dfrac{r(\Si)-t}{2}+\dfrac{n^2r(\Si) - tn - 2r(\Si) + 2t}{2}$}.
\end{cases}$$

(b) 
Under the assumption of Theorem \ref{thm-PI-reducedA} and a triangulation $\lambda=\mu$ of $\Sigma$ introduced in Section~\ref{sub:quantum-torus-reduced}, with the
anti-symmetric matrix decomposition of $\overline{\mathsf{P}}_\lambda$ as in \eqref{anti-decom-barP}, we have 
$$\bar z_i\overset{(2)}{=}
\begin{cases}
1 &\text{for $1\leq i\leq \dfrac{r(\Si)-t}{2}$},\medskip\\
n &\text{for $\dfrac{r(\Si)-t}{2}+1\leq i \leq\dfrac{|\overline V_{\lambda}|-t(n-1)-(b-t)\lfloor\frac{n}{2}\rfloor}2$,}
\end{cases}$$
where
$|\overline V_{\lambda}|=(n^2-1)r(\Sigma)-\binom{n}{2}(\#\partial\Sigma)$ given in \eqref{eq:cardinarity}. 
\end{prop}
\begin{proof}
(a) From Theorem \ref{thm:rank}, we have that    $$\rankZ \A= (m')^{r(\Sigma)-t}m^{n^2r(\Si) - tn - 2r(\Si) + 2t}.$$
From \eqref{eq-rank-torus-decom}, we have that
$$\rankZ \A = \prod_{1\leq i\leq k} (\dfrac{m''}{\gcd(m'',nz_i)})^2=
\prod_{1\leq i\leq k} (\dfrac{m'}{\gcd(m',z_i)})^2.$$
Thus we have 
$$\prod_{1\leq i\leq k} \dfrac{m'}{\gcd(m',z_i)} = (m')^{\frac{r(\Sigma)-t}2}m^{\frac{n^2r(\Si) - tn - 2r(\Si) + 2t}2}
= \prod_{1\leq i\leq k} \dfrac{m'}{\gcd(m',w_i)},$$
where 
$$w_i=
\begin{cases}
1 &\text{for $1\leq i\leq \dfrac{r(\Si)-t}{2}$},\medskip\\
n &\text{for $\dfrac{r(\Si)-t}{2}+1\leq i\leq \dfrac{r(\Si)-t}{2}+\dfrac{n^2r(\Si) - tn - 2r(\Si) + 2t}{2}$}.
\end{cases}$$
Then Lemma \ref{lem-2power-eq} completes the proof.

The proof of (b) is the same with (a).
\end{proof}

\begin{rem}
As in Proposition \ref{prop-anti-decom-P},
    we can give explicitly anti-symmetric matrix decompositions of $\mathsf{P}_\lambda$ and 
    $\overline{\mathsf{P}}_\lambda$ once we generalize Theorems \ref{thm:rank} and \ref{thm-PI-reducedA} to the case when $m''$ is even, which will be done in a future project. 
\end{rem}

\appendix
\section{Outline of proof of Proposition \ref{thm;azumaya}}\label{appendix}
Since Proposition \ref{thm;azumaya} is straightforward from the proof in Proposition 8.10 in \cite{Wan23}, we will only give an outline of the proof. 
\begin{proof}[Outline of a proof of Proposition \ref{thm;azumaya}]
Theorem 6.1 in \cite{LY23} shows $\cS_n(\Sigma,\mathbbm{v})$ satisfies conditions (1) and (2) in Definition \ref{def:almost_Azumaya}.

It suffices to show $\cS_n(\Sigma,\mathbbm{v})$ is a finitely generated $\mathcal Z(\cS_n(\Sigma,\mathbbm{v}))$-module.  
For any stated arc $\alpha\in \cS_n(\Sigma,\mathbbm{v})$, Lemmas~\ref{lem_cross} and \ref{lem7.5} implies that $\alpha^{(2md)}$ is transparent in $\cS_n(\Sigma,\mathbbm{v})$ ({\bf Claim 1}) (i.e. for any two isotopic stated $n$-webs $\beta_1$ and $\beta_2$ in $\Sigma\times (-1,1)$, we have $\alpha^{(2md)}\cup \beta_1
= \alpha^{(2md)}\cup \beta_2\in \cS_n(\Sigma,\mathbbm{v})$).

We use a saturated system to prove the theorem. Suppose $B=\{b_1,\cdots,b_r\}$ is a saturated system of $\Sigma$. We have a linear isomorphism $L_{*} \colon\cS_n(U(B),\mathbbm{v})\rightarrow \cS_n(\Sigma,\mathbbm{v})$ as in Equation \eqref{eq_iso}. 
Suppose $\mathcal A$ is the subalgebra of $\Oq$ generated by
$\{u_{ij}^{2md}\mid 1\leq i,j\leq n\}\subset\mathcal Z(\Oq)$. Then we have 
$\mathcal A\subset\mathcal Z(\Oq)$ because of {\bf Claim 1}.
It is well-known that $\Oq$ is linearly spanned by
$\{\prod_{ij}u_{ij}^{a_{ij}}\mid a_{ij}\in\mathbb N\text{ for } 1\leq i,j\leq n\}$ \cite{PW91}, where the product
$\prod_{ij}u_{ij}^{a_{ij}}$ is taken in the lexicographic order on $\{1,2,\cdots,n\}^2$.
Then $\Oq$ is generated by 
$\{\prod_{ij}u_{ij}^{a_{ij}}\mid 1\leq a_{ij}\leq 2md-1\text{ for } 1\leq i,j\leq n\}$
over $\mathcal A$.
We have $\mathcal A ^{\otimes r}\subset \mathcal Z(\cS_n(U(B),\mathbbm{v}))$ (here we identify the stated skein algebra of the bigon with $\Oq$),
and $\cS_n(U(B),\mathbbm{v})$ is finitely generated over $\mathcal A$.

From {\bf Claim 1}, we know $L_*(x)\in\mathcal Z(\cS_n(\Sigma,\mathbbm{v}))$
for $x\in \mathcal A ^{\otimes r}$. 
{\bf Claim 1} implies that $L_*(xy) = L_*(x)L_*(y)$ for any $x\in  \mathcal A ^{\otimes r}$ and $y\in \cS_n(U(B),\mathbbm{v})$.
Then $\cS_n(\Sigma,\mathbbm{v})$ is finitely generated over $L_*(\mathcal A ^{\otimes r})\subset \mathcal Z(S_n(\Sigma,\mathbbm{v}))$ because $L_{*} \colon\cS_n(U(B),\mathbbm{v})\rightarrow \cS_n(\Sigma,\mathbbm{v})$  is surjective. 

\end{proof}

\section{Addendum on the condition that “$m''$ is odd".}
This appendix is contained only in the arXiv version for reader's convenience. 

All the arguments in this paper work not only when $m''$ is odd but also when $m'$ is odd, i.e., the claims still hold when we replace “$m''$ is odd" with “$m'$ is odd" in the assumptions of Theorems~\ref{center_torus}, \ref{thm-main-center-skein}, \ref{thm:rank}, \ref{main-thm-reduced-center},  \ref{thm-PI-reducedA}, \ref{main-thm-reduced-PI}, and Lemma~\ref{lem5.1}. Since this paper has been accepted, we do not change the claims from the accepted version, and we explain the reasons why the claims still work when $m'$ is odd here.

\begin{enumerate}
    \item The proof of Theorem~\ref{center_torus}: Note that Equation~\eqref{eq_key} is equivalent to   
    \begin{equation}\label{Ap-m}
    \begin{cases}
    -2\bk_1^T+ \frac{1}{n}(D+C_1A) \bk_2^T=\bm{0},\\
    \frac{1}{n}(B-A)\bk_2^T=\bm{0},
    \end{cases}\text{ in $\mathbb Z_{m'}$},
    \end{equation}and Equation~\eqref{eq's} is equivalent to 
    \begin{equation}\label{Ap-eq's}
    \begin{cases}
\tfb_{ir_i} + \tfb_{i1} =\bm{0},\\
\tfb_{i1} + \tfb_{i2} =\bm{0},\\
 \tfb_{i2} + \tfb_{i3} =\bm{0},\\
\;\vdots\\
\tfb_{ir_i-1} + \tfb_{ir_i} =\bm{0},\\
\end{cases}\text{ in $\mathbb Z_{m'}$}.
\end{equation}
Then the arguments still work if we change Equation~\eqref{eq_key} to Equation~\eqref{Ap-m}, Equation~\eqref{eq's} to Equation~\eqref{Ap-eq's}, change all $m'$ to $m$,
change all $m''$ to $m'$, 
and change all $d'$ to $d$.
\item Theorems~\ref{thm-main-center-skein}, \ref{thm:rank}, and Lemma~\ref{lem5.1}: We can replace $m''$ with $m'$ because of the modified Theorem~\ref{center_torus}.
\item Theorems~\ref{main-thm-reduced-center},  \ref{thm-PI-reducedA}, and \ref{main-thm-reduced-PI} are stated for the reduced case.
We can replace $m''$ in them with $m'$ because the same reason for the stated case (the reduced case uses the same techniques as the stated case).
\end{enumerate}

\end{document}